\documentclass[10pt, reqno,amsmath,amsthm,amssymb,amscd]{amsart}
\usepackage{mathrsfs,amssymb, amscd,amsmath,amsthm}
\usepackage[enableskew,vcentermath]{youngtab}
\usepackage{multicol}\multicolsep=0pt
\usepackage{tikz}
\newcommand{\END}{\mathcal{E}nd}
\newcommand{\clr}{rgb:black,1;blue,4;red,1}

\newcommand{\bdot}{ node[circle, draw, fill=\clr, thick, inner sep=0pt, minimum width=4pt]{}}

\newcommand{\ob}[1]{\mathsf{#1}}

\newcommand{\down}{\downarrow}

\newcommand{\OB}{\mathcal{OB}}
\newcommand{\B}{\mathcal{B}}
\newcommand{\CB}{\mathcal{CB}}
\newcommand{\AB}{\mathcal{AB}}

\newcommand{\fh}{\mathfrak{h}}
\newcommand{\fn}{\mathfrak{n}}
\newcommand{\fb}{\mathfrak{b}}

\newcommand{\lcap}{
\begin{tikzpicture}[baseline = 3pt, scale=0.5, color=\clr]
        \draw[-,thick] (1,0) to[out=up, in=right] (0.53,0.5) to[out=left, in=right] (0.47,0.5);
        \draw[-,thick] (0.49,0.5) to[out=left,in=up] (0,0);
\end{tikzpicture}
}
\newcommand{\lcup}{
\begin{tikzpicture}[baseline = 6pt, scale=0.5, color=\clr]
        \draw[-,thick] (1,1) to[out=down, in=right] (0.53,0.5) to[out=left, in=right] (0.47,0.5);
        \draw[-,thick] (0.49,0.5) to[out=left,in=down] (0,1);
\end{tikzpicture}
}

\newcommand{\lcapo}{
\begin{tikzpicture}[baseline = 3pt, scale=0.5, color=\clr]
        \draw[-,thick] (1,0) to[out=up, in=right] (0.53,0.5) to[out=left, in=right] (0.47,0.5);
        \draw[->,thick] (0.49,0.5) to[out=left,in=up] (0,0);
\end{tikzpicture}
}
\newcommand{\lcupo}{
\begin{tikzpicture}[baseline = 6pt, scale=0.5, color=\clr]
        \draw[-,thick] (1,1) to[out=down, in=right] (0.53,0.5) to[out=left, in=right] (0.47,0.5);
        \draw[->,thick] (0.49,0.5) to[out=left,in=down] (0,1);
\end{tikzpicture}
}

\newcommand{\swap}{
\begin{tikzpicture}[baseline = 3pt, scale=0.5, color=\clr]
        \draw[-,thick] (0,0) to[out=up, in=down] (1,1);
        \draw[-,thick] (1,0) to[out=up, in=down] (0,1);
\end{tikzpicture}
}
\newcommand{\swapo}{
\begin{tikzpicture}[baseline = 3pt, scale=0.5, color=\clr]
        \draw[->,thick] (0,0) to[out=up, in=down] (1,1);
        \draw[->,thick] (1,0) to[out=up, in=down] (0,1);
\end{tikzpicture}
}

\newcommand{\xdot}{
\begin{tikzpicture}[baseline = 3pt, scale=0.5, color=\clr]

\draw[-,thick] (0,0) to[out=up, in=down] (0,1);
\draw(0,0.5) \bdot;
\end{tikzpicture}
}

 \providecommand{\og}{``}
\providecommand{\fg}{''} \providecommand{\smfandname}{and}

\usepackage{amssymb}
\usepackage{mathrsfs,amssymb}
\usepackage{multicol}\multicolsep=0pt
\usepackage{pstricks,pst-node}
\usepackage[enableskew,vcentermath]{youngtab}

\usepackage[sort]{cite}
\usepackage{xcolor,graphicx}

\def\crulefill{\leavevmode\leaders\hrule height 1pt\hfill\kern 0pt}
\long\def\QUERY#1{%
\leavevmode\newline%
\noindent$\star\star\star$\thinspace\textsf{Comment/Query}\crulefill\newline%
   \space #1\newline\hbox to 120mm{\crulefill}$\star\star\star$\newline}
\newtheorem{Theorem}{Theorem}[section]
\newtheorem{Lemma}[Theorem]{Lemma}
\newtheorem{Cor}[Theorem]{Corollary}
\newtheorem{Prop}[Theorem]{Proposition}

 \theoremstyle{definition}
\newtheorem{example}[Theorem]{Example}

\newtheorem{Defn}[Theorem]{Definition}

\newtheorem{rem}[Theorem]{Remark}

\numberwithin{equation}{section}
\theoremstyle{definition}

\newtheorem{THEOREM}{Theorem}

\newtheorem*{Assumptions}{Assumptions}

\makeatletter
\def\enumerate{\begingroup\ifnum\@enumdepth>3\@toodeep\else
      \advance\@enumdepth\@ne
      \edef\@enumctr{enum\romannumeral\the\@enumdepth}%
      \topsep\z@\parskip\z@
      \list{\csname label\@enumctr\endcsname}
        {\@nmbrlisttrue\let\@listctr\@enumctr
         \parsep\z@\itemsep\z@\topsep\z@
         \setcounter{\@enumctr}{0}
         \def\makelabel##1{\hss\llap{\rm ##1}}
       }\fi}

\makeatother

\let\bar=\overline
\let\epsilon=\varepsilon
\def\({\big(}
\def\){\big)}

\def\Z{\mathbb Z}

\def\0{\underline{0}}

\DeclareMathOperator{\End}{End}

\def\F{\mathcal F}

\def\Hom{\text{Hom}}

\def\U{\mathbf U}

{\catcode`\|=\active
  \gdef\set#1{\mathinner{\lbrace\,{\mathcode`\|"8000%
                                   \let|\midvert #1}\,\rbrace}}
  \gdef\seT#1{\mathinner{\Big\lbrace\,{\mathcode`\|"8000%
                                   \let|\midverT #1}\,\Big\rbrace}}
}
\def\midvert{\egroup\mid\bgroup}
\def\midverT{\egroup\,\Big|\,\bgroup}

\def\Set[#1]#2|#3|{\Big\{\ #2\ \Big| \
           \vcenter{\hsize #1mm\centering #3}\Big\}}

\def\Z{\mathbb{Z}}

\def\F{\mathbb{F}}

\def\Hom{{\rm Hom}}

\def\mfg{{\mathfrak g}}
\def\fh{{\mathfrak h}}

\def\Set{{\rm Set}}

\newcommand{\C}{\mathcal{C}}

\def\F{\mathcal F}%
\def\Hom{\text{Hom}}%
\def\U{\mathbf U}%
\def\textsf#1{{\textit{#1}}}%

\begin{document}
\title{Affine Brauer category and  parabolic category $\mathcal O$ in types $B, C, D$}
\author{ Hebing Rui, Linliang Song}
\address{H.R.  School of Mathematical Science, Tongji University,  Shanghai, 200092, China}\email{hbrui@tongji.edu.cn}
\address{L.S.  School of Mathematical Science, Tongji University,  Shanghai, 200092, China}\email{llsong@tongji.edu.cn}

\thanks{H. Rui is supported  partially by NSFC (grant No.  11571108).  L. Song is supported  partially by NSFC (grant No.  11501368) and the Fundamental Research Funds for the Central Universities (grant No.22120180001). }
\date{\today}
\sloppy \maketitle

\begin{abstract}A  strict monoidal category referred to as affine Brauer category $\AB$ is introduced over a commutative ring  $\kappa$ containing multiplicative identity $1$ and  invertible element $2$.   We prove that  morphism spaces in $\AB$ are  free over $\kappa$. The cyclotomic (or level $k$) Brauer category $\CB^f(\omega)$ is a  quotient category of $\AB$.  We prove that any  morphism space in  $\CB^f(\omega)$ is free over $\kappa$ with maximal rank if and only if
the $\mathbf u$-admissible condition  holds in the sense of \eqref{admc}. Affine Nazarov-Wenzl algebras~\cite{Na} and cyclotomic  Nazarov-Wenzl algebras~\cite{AMR} will be realized as certain  endomorphism algebras in $\AB$ and $\CB^f(\omega)$, respectively.
 We will establish higher  Schur-Weyl duality between cyclotomic Nazarov-Wenzl algebras and  parabolic BGG categories $\mathcal O$ associated to symplectic and orthogonal Lie algebras over the complex field $\mathbb C$.
 This enables us to use standard arguments in \cite{AST, RS1, RS2} to   compute decomposition matrices of cyclotomic  Nazarov-Wenzl algebras. The level two case was considered by  Ehrig and Stroppel in  \cite{ES}.
\end{abstract}


\section{Introduction
}\label{affb}
Throughout this paper,  $\kappa$ is a commutative ring  containing multiplicative identity $1$ and invertible element $2$.
A category $\C$ is said to be $\kappa$-linear if all Hom sets  are equipped with the structure of $\kappa$-modules in such a way that composition of morphisms is bilinear.
Unless otherwise stated,  we assume that all categories and functors and  associative algebras are defined over $\kappa$.

The aim of this paper is to  introduce the $\kappa$-linear strict  monoidal category   $\AB$ and its quotient category   $\CB^f$, where
$f(t)$ is a monic polynomial in $\kappa[t]$. They will be called   the
 \emph{affine Brauer category} and the \emph{cyclotomic Brauer category}, respectively.
There are some closed relationships between these  categories and the  representations of orthogonal and symplectic Lie algebras over the complex field  $\mathbb C$.
 Affine Nazarov-Wenzl algebras\cite{Na}
and cyclotomic Nazarov-Wenzl algebras\cite{AMR} appear naturally as some  endomorphism algebras in   $\AB$ and $\CB^f$, respectively.
In the future, we  hope to establish higher  Schur-Weyl duality between $\CB^f$ and finite $W$-algebras of types $B, C$ and $D$. Also, we hope to find some applications of
 $ \AB$ and $\CB^f$ in the theory of categorification and the representation theory of ortho-symplectic Lie superalgebras.
We refer the reader to the papers \cite{BDE,Br,CK,CW,DM,GRSS1,GZ,LZ,Re} for more examples of  applications of monoidal (super)categories.

 We start by recalling the notion of a strict monoidal category.
A  \emph{monoidal category}~\cite{T} is a category
$\C$ equipped with a  bifunctor
$-\otimes -: \C \times  \C \rightarrow  \C$, a unit object $\ob 1$, and three natural isomorphisms
 $$\alpha:(-\otimes -)\otimes -\cong -\otimes (-\otimes-),\ \  \text{ $\lambda: \ob 1\otimes - \cong -$ and $\mu:-\otimes  \ob 1\cong -$,}$$ called  coherence maps in the sense that for any objects $\ob a, \ob b, \ob c, \ob d$ in $\C$,
 \begin{enumerate}\item[(1)] $( 1_{\ob a}\otimes \alpha_{\ob b, \ob c, \ob d})\circ (\alpha_{\ob a, \ob b\otimes \ob c, \ob d})\circ (\alpha_{\ob a, \ob b, \ob c}\otimes  1_{\ob d})=(\alpha_{\ob a, \ob b, \ob c\otimes \ob d})\circ (\alpha_{\ob a\otimes \ob b, \ob c, \ob d})$,
  \item[(2)] $( 1_{\ob a}\otimes \lambda_{\ob b})\circ \alpha_{\ob a, 1, \ob b}=\mu_{\ob a}\otimes  1_{\ob b}$,
 \end{enumerate}
\noindent where $1_{\ob a}: \ob a\rightarrow \ob a$ is  the identity morphism in $\End_{\mathcal C}(\ob a)$.

Suppose that $\mathcal C$ and $\mathcal C'$  are two monoidal categories such that  $\alpha$ and $\ob1$ (resp., $\alpha'$ and $\ob1'$)  are the coherence map and the unit object of $\mathcal C$ (resp.,  $\mathcal C'$). A monoidal functor from $ \mathcal C$ to $ \mathcal C'$ is a pair $(F,J)$, where $F: \mathcal C\rightarrow \mathcal C'$ is a functor and
$J_{\ob a,\ob b}: F(\ob a)\otimes F(\ob b)\rightarrow F(\ob a\otimes \ob b)$ is a natural isomorphism such that $F(\ob 1)$ is isomorphic to $\ob 1'$ and
$$F(\alpha_{\ob a,\ob b,\ob c})\circ J_{\ob a\otimes \ob b,\ob c}\circ (J_{\ob a,\ob b}\otimes 1_{F(\ob c)})= J_{\ob a,\ob b\otimes \ob c}\circ (1_{F(\ob a)}\otimes J_{\ob b,\ob c})\circ \alpha'_{F(\ob a),F(\ob b),F(\ob c) }. $$

 A \emph{strict monoidal category} is one for which the natural isomorphisms $\alpha, \lambda$ and $\mu$ are  identities.
Unless otherwise stated, \emph{  all monoidal categories  are assumed to be strict}.

 Any monoidal functor $(F,J)$ between two monoidal categories $\mathcal C$ and $\mathcal C'$ is assumed to be strict in the sense that $F(\ob a\otimes \ob b)=F(\ob a)\otimes F(\ob b)$, $F(\ob 1)=\ob 1'$ and $J_{\ob a,\ob b}=1_{F(\ob a)\otimes F(\ob b)}$ for all objects $\ob a,\ob b$ in $\C$.

 Let $\tau_{\C}:\C\times \C\rightarrow\C\times \C$ be the functor such that  $\tau_\C (\ob a, \ob b)=(\ob b, \ob a)$.
 A \emph{braiding} on  $\C$ is a natural isomorphism $\sigma:  -\otimes -\rightarrow (-\otimes -)\circ \tau_{\C}$ such that
 \begin{equation}\label{symm} \sigma_{\ob a, \ob b\otimes \ob c}=(1_{\ob b}\otimes \sigma_{\ob a, \ob c}) \circ (\sigma_{\ob a, \ob b}\otimes 1_{\ob c})\text{ and
$\sigma_{\ob a\otimes \ob b,  \ob c}=(\sigma_{\ob a,  \ob c}\otimes  1_{\ob b}) \circ (1_{\ob a}\otimes \sigma_{\ob b, \ob c})$,   $\forall \ob a, \ob b, \ob c\in \C$.}\end{equation}
If $\sigma_{\ob a, \ob b}^{-1}=\sigma_{\ob b, \ob a}$ for all objects $\ob a, \ob b\in \C$, then  $\sigma$ is said to be  \emph{symmetric}. A monoidal category equipped with a braiding (resp., a symmetric braiding) is called a \emph{ braided (resp., symmetric) monoidal category}.

 A \emph{right dual} of an object $\ob a$ in a monoidal category $\mathcal C$ consists of a triple $(\ob a^*, \eta_{\ob a}, \epsilon_{\ob a})$, where  $\ob a^*$ is an object,    $\eta_{\ob a}: 1\rightarrow \ob a \otimes \ob a^*$,   $\epsilon_{\ob a}: \ob a^*\otimes \ob a\rightarrow \ob 1$  are morphisms in $\C$ such that \begin{equation}\label{rigdual} (1_{\ob a}\otimes \epsilon_{\ob a})\circ (\eta_{\ob  a} \otimes 1_{\ob a})=\ob 1_{\ob a}\text{ and  $ (\epsilon_{\ob a}\otimes \ob 1_{\ob a^*})\circ (1_{\ob a^*}\otimes\eta_{\ob  a}  )= 1_{\ob a^*}$.}\end{equation}
 The morphism $\eta_{\ob a}$ (resp., $\epsilon_{\ob a}$) is called the unit (resp., counit)  morphism.
There is also a notion of \emph{left dual}. An object in a symmetric monoidal category  has a right dual  if and only if it has a left dual. A monoidal category in which every object has both a left and right dual is called \emph{a rigid monoidal category}.

There is a well-defined graphical calculus for morphisms in a monoidal category $\C$ (c.f. \cite[\S1.6]{VT} or \cite[\S1.2]{BE}).
For any two objects $\ob a,\ob b$ in  $\C$,  $\ob a\ob b$ represents  $\ob a\otimes \ob b$.
    A  morphism $g:\ob a\to \ob b$ is drawn as
    $$
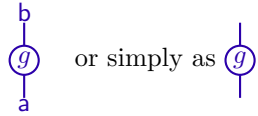
\begin{tikzpicture}[baseline = 12pt,scale=0.5,color=\clr,inner sep=0pt, minimum width=11pt]
        \draw[-,thick] (0,0) to (0,2);
        \draw (0,1) node[circle,draw,thick,fill=white]{$g$};
        \draw (0,-0.2) node{$\ob a$};
        \draw (0, 2.3) node{$\ob b$};
    \end{tikzpicture}
     \quad\text{ or simply as }
    \begin{tikzpicture}[baseline = 12pt,scale=0.5,color=\clr,inner sep=0pt, minimum width=11pt]
        \draw[-,thick] (0,0) to (0,2);
        \draw (0,1) node[circle,draw,thick,fill=white]{$g$};
    \end{tikzpicture}$$
 if there is no confusion on the objects. Note that $\ob a$ is  at the bottom while $\ob b$ is at the top.
 The composition  (resp., tensor product) of two morphisms  is given by vertical stacking (resp., horizontal concatenation):
\begin{equation}\label{com1}
      g\circ h= \begin{tikzpicture}[baseline = 19pt,scale=0.5,color=\clr,inner sep=0pt, minimum width=11pt]
        \draw[-,thick] (0,0) to (0,3);
        \draw (0,2.2) node[circle,draw,thick,fill=white]{$g$};
        \draw (0,0.8) node[circle,draw,thick,fill=white]{$h$};
    \end{tikzpicture}
    ~,~ \ \ \ \ \ g\otimes h=\begin{tikzpicture}[baseline = 19pt,scale=0.5,color=\clr,inner sep=0pt, minimum width=11pt]
        \draw[-,thick] (0,0) to (0,3);
        \draw[-,thick] (2,0) to (2,3);
        \draw (0,1.5) node[circle,draw,thick,fill=white]{$g$};
        \draw (2,1.5) node[circle,draw,thick,fill=white]{$h$};
    \end{tikzpicture}~.
\end{equation}
 Since $-\otimes -: \C\times\C \rightarrow \C$ is a bifunctor, we have the interchange law:
\begin{equation}\label{wwcirrten}
 (f\otimes g)\circ(k\otimes h)=(f\circ k)\otimes (g\circ h),
\end{equation}
for any $f\in\Hom_{\C}(\ob b,\ob c)$, $g\in\Hom_{\C}(\ob n,\ob t )$, $k\in\Hom_{\C}(\ob a,\ob b)$, $h\in\Hom_{\C}(\ob m,\ob n)$.
Both sides of \eqref{wwcirrten} are drawn as
 
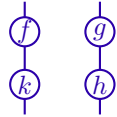
\begin{equation*}
    \begin{tikzpicture}[baseline = 19pt,scale=0.5,color=\clr,inner sep=0pt, minimum width=11pt]
        \draw[-,thick] (0,0) to (0,3);
        \draw[-,thick] (2,0) to (2,3);
        \draw (0,2.2) node[circle,draw,thick,fill=white]{$f$};
        \draw (2,2.2) node[circle,draw,thick,fill=white]{$g$};
        \draw (0,0.8) node[circle,draw,thick,fill=white]{$k$};
        \draw (2,0.8) node[circle,draw,thick,fill=white]{$h$};
    \end{tikzpicture}.
\end{equation*}

The Brauer category $\B$ and affine Brauer category $\AB$ that we will introduce are  monoidal categories generated by the unique object denoted by
 \begin{tikzpicture}[baseline = 10pt, scale=0.5, color=\clr]
                \draw[-,thick] (0,0.5)to[out=up,in=down](0,1.2);
    \end{tikzpicture}. Therefore,  the set of  objects of $\B$ (resp., $\AB$) is
$\{\begin{tikzpicture}[baseline = 10pt, scale=0.5, color=\clr]
                \draw[-,thick] (0,0.5)to[out=up,in=down](0,1.2);
    \end{tikzpicture}^{\otimes m}\mid m\in \mathbb N\}$, where
$\begin{tikzpicture}[baseline = 10pt, scale=0.5, color=\clr]
                \draw[-,thick] (0,0.5)to[out=up,in=down](0,1.2);
    \end{tikzpicture}^{\otimes 0}$ is the unit object.
     To simplify the notation, we denote
$ \begin{tikzpicture}[baseline = 10pt, scale=0.5, color=\clr]
                \draw[-,thick] (0,0.5)to[out=up,in=down](0,1.2);
                    \end{tikzpicture}^{\otimes m}$
by $\ob m$. Unlike the notation in a usual monoidal category, the unit object is denoted by $\ob 0$, while the object
\begin{tikzpicture}[baseline = 10pt, scale=0.5, color=\clr] \draw[-,thick] (0,0.5)to[out=up,in=down](0,1.2);
    \end{tikzpicture}
    is denoted by $\ob 1$.
\begin{Defn}\label{DefnB} The Brauer category $\B$ is the $\kappa$-linear  monoidal category  generated by the unique object $\ob 1$
and three morphisms $U:\ob0\rightarrow \ob2$, $ A:\ob2\rightarrow\ob0$ and $ S: \ob2\rightarrow\ob2$ subject to the following relations:

\begin{equation}\label{OBR1}
S\circ S=\text{1}_{\ob 2},\ \  (S\otimes\text{1}_{\ob 1})\circ (\text{1}_{\ob 1}\otimes  S)\circ (S\otimes\text{1}_{\ob 1})=(\text{1}_{\ob 1}\otimes  S)\circ (S\otimes\text{1}_{\ob 1})\circ (\text{1}_{\ob 1}\otimes  S),
\end{equation}
\begin{equation}\label{OBR2}
(\text{1}_{\ob 1}\otimes A)\circ (U\otimes\text{1}_{\ob 1})=\text{1}_{\ob 1}=(A\otimes\text{1}_{\ob 1})\circ (\text{1}_{\ob 1}\otimes U),
\end{equation}
\begin{equation}\label{OBR3}
U=S\circ U,\ \  A=A\circ S,
\end{equation}
\begin{equation}\label{OBR4}
(A\otimes\text{1}_{\ob 1}) \circ (\text{1}_{\ob 1}\otimes S)=(\text{1}_{\ob 1}\otimes A)\circ (S\otimes \text{1}_{\ob 1}), \ \
(S\otimes \text{1}_{\ob 1})\circ (\text{1}_{\ob 1}\otimes U)=(\text{1}_{\ob 1}\otimes S) \circ (U\otimes\text{1}_{\ob 1}).
\end{equation}
\end{Defn}
We draw  $U,A$ and $S$ as follows:
    \begin{equation}\label{uas}  U=\lcup,\quad   A=\lcap,\quad  \text{ and  $S=\swap$.}\end{equation}
Each endpoint  at both rows of the above diagrams represents the object $\ob 1$. If there is no endpoints at a row, then the object at this row is the unit object $\ob 0$.
So, the object at  each row of $U,A,S$ is determined  by the endpoints on it. For example, the object  at the bottom (resp., top) row of $U$ is $\ob 0$ (resp., $\ob 2$).
For any $m>0$, the identity morphism $1_{\ob m}$  is drawn as the object itself. For example, $\text{1}_{\ob 2}$ is drawn as $\begin{tikzpicture}[baseline = 10pt, scale=0.5, color=\clr]
                \draw[-,thick] (0,0.5)to[out=up,in=down](0,1.2);\draw[-,thick] (0.5,0.5)to[out=up,in=down](0.5,1.2);
    \end{tikzpicture}$.
    The objects at both rows of  $\begin{tikzpicture}[baseline = 10pt, scale=0.5, color=\clr]
                \draw[-,thick] (0,0.5)to[out=up,in=down](0,1.2);\draw[-,thick] (0.5,0.5)to[out=up,in=down](0.5,1.2);
    \end{tikzpicture}$ are also determined by the endpoints.
Thanks to \eqref{com1},  \eqref{OBR1}--\eqref{OBR4} are depicted as \eqref{OB relations 1 (symmetric group)}--\eqref{Brauer relation 4}, respectively:
 \begin{equation}\label{OB relations 1 (symmetric group)}
        \begin{tikzpicture}[baseline = 10pt, scale=0.5, color=\clr]
            \draw[-,thick] (0,0) to[out=up, in=down] (1,1);
            \draw[-,thick] (1,1) to[out=up, in=down] (0,2);
            \draw[-,thick] (1,0) to[out=up, in=down] (0,1);
            \draw[-,thick] (0,1) to[out=up, in=down] (1,2);
                    \end{tikzpicture}
        ~=~
        \begin{tikzpicture}[baseline = 10pt, scale=0.5, color=\clr]
            \draw[-,thick] (0,0) to (0,1);
            \draw[-,thick] (0,1) to (0,2);
            \draw[-,thick] (1,0) to (1,1);
            \draw[-,thick] (1,1) to (1,2);
        \end{tikzpicture}
        ,\qquad
        \begin{tikzpicture}[baseline = 10pt, scale=0.5, color=\clr]
            \draw[-,thick] (0,0) to[out=up, in=down] (2,2);
            \draw[-,thick] (2,0) to[out=up, in=down] (0,2);
            \draw[-,thick] (1,0) to[out=up, in=down] (0,1) to[out=up, in=down] (1,2);
        \end{tikzpicture}
        ~=~
        \begin{tikzpicture}[baseline = 10pt, scale=0.5, color=\clr]
            \draw[-,thick] (0,0) to[out=up, in=down] (2,2);
            \draw[-,thick] (2,0) to[out=up, in=down] (0,2);
            \draw[-,thick] (1,0) to[out=up, in=down] (2,1) to[out=up, in=down] (1,2);
        \end{tikzpicture},
    \end{equation}
    \begin{equation}\label{OB relations 2 (zigzags and invertibility)}
        \begin{tikzpicture}[baseline = 10pt, scale=0.5, color=\clr]
            \draw[-,thick] (2,0) to[out=up, in=down] (2,1) to[out=up, in=right] (1.5,1.5) to[out=left,in=up] (1,1);
            \draw[-,thick] (1,1) to[out=down,in=right] (0.5,0.5) to[out=left,in=down] (0,1) to[out=up,in=down] (0,2);
        \end{tikzpicture}
        ~=~
        \begin{tikzpicture}[baseline = 10pt, scale=0.5, color=\clr]
            \draw[-,thick] (0,0) to (0,1);
            \draw[-,thick] (0,1) to (0,2);
        \end{tikzpicture}
        ~=~
        \begin{tikzpicture}[baseline = 10pt, scale=0.5, color=\clr]
            \draw[-,thick] (2,2) to[out=down, in=up] (2,1) to[out=down, in=right] (1.5,0.5) to[out=left,in=down] (1,1);
            \draw[-,thick] (1,1) to[out=up,in=right] (0.5,1.5) to[out=left,in=up] (0,1) to[out=down,in=up] (0,0);
        \end{tikzpicture},
    \end{equation}

\begin{equation}\label{right cuaps - down cross}
    \begin{tikzpicture}[baseline = 5pt, scale=0.5, color=\clr]
        \draw[-,thick] (0,1) to[out=down,in=left] (0.5,0.35) to[out=right,in=down] (1,1);
    \end{tikzpicture}
    ~=~
    \begin{tikzpicture}[baseline = 5pt, scale=0.5, color=\clr]
        \draw[-,thick] (0,1) to[out=down,in=up] (1,0) to[out=down,in=right] (0.5,-0.5) to[out=left,in=down] (0,0) to[out=up,in=down] (1,1);
    \end{tikzpicture}
    ~,\qquad
    \begin{tikzpicture}[baseline = 5pt, scale=0.5, color=\clr]
        \draw[-,thick] (0,0) to[out=up,in=left] (0.5,0.65) to[out=right,in=up] (1,0);
    \end{tikzpicture}
    ~=~
    \begin{tikzpicture}[baseline = 5pt, scale=0.5, color=\clr]
        \draw[-,thick] (0,0) to[out=up,in=down] (1,1) to[out=up,in=right] (0.5,1.5) to[out=left,in=up] (0,1) to[out=down,in=up] (1,0);
    \end{tikzpicture},
    \end{equation}

    \begin{equation}\label{Brauer relation 4}
    \begin{tikzpicture}[baseline = 10pt, scale=0.5, color=\clr]
                \draw[-,thick] (1,0) to[out=up,in=right] (0.5,1.5) to[out=left,in=up] (0,1) to[out=down,in=up] (0,0);
        \draw[-,thick] (0.5,0) to[out=up,in=down] (1.3,1.5);
    \end{tikzpicture}~=~
    \begin{tikzpicture}[baseline = 10pt, scale=0.5, color=\clr]
         \draw[-,thick] (1,0) to[out=up,in=right] (0.5,1.5) to[out=left,in=up] (0.2,1) to[out=down,in=up] (0.2,0);
        \draw[-,thick] (0.7,0) to[out=up,in=down] (0,1.5);
    \end{tikzpicture}
    ~,\begin{tikzpicture}[baseline = 10pt, scale=0.5, color=\clr]
        \draw[-,thick]  (2,1.5) to[out=down, in=right] (1.5,0) to[out=left,in=down] (1,1.5);
        \draw[-,thick] (0.7,0) to[out=up,in=down] (1.5,1.5);
    \end{tikzpicture}~=~
    \begin{tikzpicture}[baseline = 10pt, scale=0.5, color=\clr]
        \draw[-,thick]  (2,1.5) to[out=down, in=right] (1.5,0) to[out=left,in=down] (1,1.5);
        \draw[-,thick] (2.3,0)to[out=up,in=down](1.5,1.5);
    \end{tikzpicture}.
\end{equation}

In order to give a $\kappa$-basis of   $\Hom_{\B}(\ob m, \ob s)$, we need the notion of $( m, s)$-Brauer diagram  as follows. An $(m,s)$-{\em  Brauer diagram} (or a Brauer diagram from $\ob m$ to $\ob s$) is a string diagram obtained by tensor product and composition of
 the generators $A, U, S$  and the identity  morphism \begin{tikzpicture}[baseline = 10pt, scale=0.5, color=\clr]
                \draw[-,thick] (0,0.5)to[out=up,in=down](0,1.5);\end{tikzpicture}
    of $\B$ such that there are $m$ (resp., $s$) endpoints  on the bottom (resp., top) row of the  resulting diagram.  For example, the following is a $(7,3) $-Brauer diagram:
 \begin{equation}\label{example of m ,s}
\begin{tikzpicture}[baseline = 25pt, scale=0.35, color=\clr]
        \draw[-,thick] (5,0) to[out=up,in=down] (8,5);
         \draw[-,thick] (10,0) to[out=up,in=down] (6,5);
          \draw[-,thick] (2.6,5) to (2.5,5) to[out=left,in=up] (1.5,4)
                        to[out=down,in=left] (2.5,3)
                        to[out=right,in=down] (3.5,4)
                        to[out=up,in=right] (2.5,5);
        \draw[-,thick] (-1,0) to[out=up,in=left] (1,1.5) to[out=right,in=up] (3,0);
         \draw[-,thick] (2,0) to[out=up,in=left] (4,2.5) to[out=right,in=up] (6,0);
          \draw[-,thick] (12.6,4.5) to (12.5,4.5) to[out=left,in=up] (10.5,2.5)
                        to[out=down,in=left] (12.5,0.5)
                        to[out=right,in=down] (14.5,2.5)
                        to[out=up,in=right] (12.5,4.5);
                        \draw[-,thick] (11,0) to [out=up,in=down](12,1.7);   \draw[-,thick] (12,1.7)to [out=up,in=down](11,5);
           \end{tikzpicture}. \end{equation}

For any $m,s\in\mathbb N$, let  \begin{equation}\label{mbs} \mathbb{B}_{m,s}=\{ \text{  all $(m,s)$-Brauer  diagrams}\}.\end{equation}
   Then  any element in $\mathbb{B}_{m,s}$  can be represented as a morphism in $\Hom_{\B}(\ob m,\ob s)$. Since $\Hom_{\B}(\ob m,\ob s)$ is generated by $A,U,S$ and $ \begin{tikzpicture}[baseline = 10pt, scale=0.5, color=\clr]
                \draw[-,thick] (0,0.5)to[out=up,in=down](0,1.5);\end{tikzpicture}$, it is spanned by $\mathbb{B}_{m,s}$.

  There are finitely many loops on any $d\in \mathbb{B}_{m,s}$. The loops are usually  called bubbles (c.f.\cite{BCNR}).
  For example, there are two bubbles in \eqref{example of m ,s}. The  endpoints at the bottom (resp., top) row  of $d$ are indexed by  $1,2,\ldots, m$ (resp., $m+1, m+2,\ldots,m +s$) from left to right.
Two endpoints   are  paired if they    are connected  by a strand.
Since each endpoint of $d$ is uniquely  connected with another  one, it  gives a partition of $\{1,2,\ldots, m+s\}$ into disjoint union  pairs.
For example, the $(7,3)$-Brauer diagram in \eqref{example of m ,s} gives the partition $\{\{1,3\}, \{2,5\}, \{4,9\},\{6,8\},\{7,10\}\}$.
So, $\mathbb{B}_{m,s}=\emptyset $  if $m+s$ is odd.
The  strand connecting the pairs on different rows (resp., the same row) is called a vertical  (resp.,  horizontal) strand. Moreover, the horizontal strand connecting the pairs on the top (resp., bottom) row is also called  a cup (resp., cap). For example, $U$ is a cup and   $A$ is   a cap.

We say that  $d, d'\in  \mathbb{B}_{m,s}$    are   \emph{equivalent} and write $d\sim d'$ if \begin{itemize}\item they have the same number of bubbles, \item they  give  the same partition of $\{1,2,\ldots, m+s\}$ into disjoint union of pairs.\end{itemize}
Let $\mathbb{B}_{m,s}/\sim$ be the set of all equivalence classes of $\mathbb{B}_{m,s}$. By Lemma~\ref{equi123}, $d=d'$  in $\B$ if  $d\sim d'$. Thus, any equivalence class of  $\mathbb{B}_{m,s}$ can be identified with any element in it.

\begin{THEOREM}\label{basisofb}
 For any    $  m,   s\in\mathbb N$,
 $\Hom_{\B}(\ob m,\ob s)$  has $\kappa$-basis   given by  $\mathbb{B}_{m,s}/\sim$.
 \end{THEOREM}

Theorem~\ref{basisofb} will be proven at the end of section~\ref{iso}. Let  $\Delta_0$
be  the crossing-free  bubble $ \begin{tikzpicture}[baseline = 5pt, scale=0.5, color=\clr]
        \draw[-,thick] (0.6,1) to (0.5,1) to[out=left,in=up] (0,0.5)
                        to[out=down,in=left] (0.5,0)
                        to[out=right,in=down] (1,0.5)
                        to[out=up,in=right] (0.5,1);\end{tikzpicture}$. So, $\Delta_0=A\circ U\in  \End_{\B}(\ob 0)$.
                        For any   $\omega_0\in \kappa$, define
                         $\B(\omega_0)$ to  be the $\kappa$-linear monoidal category obtained from $\B$ by imposing the additional relation  $\Delta_0=\omega_0$.
 Thanks to \cite[Theorem~2.16, Remark~3.3]{LZ}, $\B(\omega_0)$ is  isomorphic   to
  the \emph{category of Brauer diagrams} in \cite{LZ}. {For any $m,s\in\mathbb N$, define
\begin{equation}\label{defofbarbms}
\bar{\mathbb B}_{m,s}:=\{d\in \mathbb B_{m,s}\mid d \text{ has no bubbles}\}.
\end{equation}
Thanks  to Theorem~\ref{basisofb}, $\Hom_{\B(\omega_0)}(\ob m,\ob s)$ has  $\kappa$-basis given by $\bar{\mathbb B}_{m,s}/\sim$. This result has already been given in \cite[Theorem~2.16]{LZ}.}
\begin{Defn}\label{AOBC defn}
    The \emph{affine  Brauer category} $\AB$ is the $\kappa$-linear monoidal category generated by the unique  object $\ob1$; four  morphisms $U:\ob0\rightarrow \ob2, A: \ob2\rightarrow\ob0, S: \ob2\rightarrow\ob2$, $X:\ob1\rightarrow\ob1$,  subject to \eqref{OBR1}--\eqref{OBR4} and
     \eqref{AOB1}-\eqref{AOB2} as follows:
   \begin{equation}\label{AOB1}
   (X\otimes \text{1}_{\ob 1}) \circ S-S\circ (\text{1}_{\ob 1}\otimes X)=U\circ A-\text{1}_{\ob 2},
   \end{equation}
        \begin{equation}\label{AOB2}
        (\text{1}_{\ob 1}\otimes A)\circ(\text{1}_{\ob 1}\otimes X\otimes \text{1}_{\ob 1})\circ (U\otimes\text{1}_{\ob 1})=-X=(A\otimes\text{1}_{\ob 1})\circ(\text{1}_{\ob 1}\otimes X\otimes \text{1}_{\ob 1})\circ (\text{1}_{\ob 1}\otimes U).
   \end{equation}
   \end{Defn}
To simplify the notation, we keep the notations in \eqref{uas} and draw   the morphism $X$ as $\xdot $ later on.
 Again the objects at  both rows of $\xdot $ are indicated by  its  endpoints. Then the relations \eqref{AOB1}--\eqref{AOB2} are depicted as \eqref{AOBC relations}--\eqref{OB relations 2 (zigzags and invertibility2)}, respectively:
 \begin{equation}\label{AOBC relations}
                \begin{tikzpicture}[baseline = 7.5pt, scale=0.5, color=\clr]
            \draw[-,thick] (0,0) to[out=up, in=down] (1,2);
            \draw[-,thick] (0,2) to[out=up, in=down] (0,2.2);
            \draw[-,thick] (1,0) to[out=up, in=down] (0,2);
            \draw[-,thick] (1,2) to[out=up, in=down] (1,2.2);
             \draw(0,1.9)\bdot;
        \end{tikzpicture}
        ~-~
        \begin{tikzpicture}[baseline = 7.5pt, scale=0.5, color=\clr]
            \draw[-,thick] (0,0) to[out=up, in=down] (1,2);\draw[-,thick] (0,0) to[out=up, in=down] (0,-0.2);
             \draw[-,thick] (1,0) to[out=up, in=down] (0,2);\draw[-,thick] (1,0) to[out=up, in=down] (1,-0.2);
                        \draw(1,0.1)\bdot;
        \end{tikzpicture}
        ~=~
       \begin{tikzpicture}[baseline = 10pt, scale=0.5, color=\clr]
          \draw[-,thick] (2,2) to[out=down,in=right] (1.5,1.5) to[out=left,in=down] (1,2);
            \draw[-,thick] (2,0) to[out=up, in=right] (1.5,0.5) to[out=left,in=up] (1,0);
        \end{tikzpicture}
        ~-~
        \begin{tikzpicture}[baseline = 7.5pt, scale=0.5, color=\clr]
            \draw[-,thick] (0,0) to[out=up, in=down] (0,2);
            \draw[-,thick] (1,0) to[out=up, in=down] (1,2);
                   \end{tikzpicture},
    \end{equation}

  \begin{equation}\label{OB relations 2 (zigzags and invertibility2)}
               \begin{tikzpicture}[baseline = 10pt, scale=0.5, color=\clr]
            \draw[-,thick] (2,0) to[out=up, in=down] (2,1) to[out=up, in=right] (1.5,1.5) to[out=left,in=up] (1,1);
            \draw[-,thick] (1,1) to[out=down,in=right] (0.5,0.5) to[out=left,in=down] (0,1) to[out=up,in=down] (0,2);
           \draw( 1,1) \bdot;
        \end{tikzpicture}
        ~=~
        -\begin{tikzpicture}[baseline = 10pt, scale=0.5, color=\clr]
            \draw[-,thick] (0,0) to (0,2);
            \draw(0,1) \bdot;
                    \end{tikzpicture}
        ~=~
        \begin{tikzpicture}[baseline = 10pt, scale=0.5, color=\clr]
            \draw[-,thick] (2,2) to[out=down, in=up] (2,1) to[out=down, in=right] (1.5,0.5) to[out=left,in=down] (1,1);
            \draw[-,thick] (1,1) to[out=up,in=right] (0.5,1.5) to[out=left,in=up] (0,1) to[out=down,in=up] (0,0);
            \draw( 1,1) \bdot;
        \end{tikzpicture}.
    \end{equation}

  There is a monoidal  functor
  \begin{equation}\label{functorfrombtoab}
  \F:\B\rightarrow \AB\end{equation}
  sending the generators of $\B$ to the generators of $\AB$ with the same names. Therefore, any $d\in \mathbb{B}_{m,s}$   can be interpreted  as a morphism in $\AB$. After we prove the basis theorem for $\AB$, we see that this functor is faithful. So,  $\B$ can be considered as a subcategory of $\AB$. This is the
 reason why we call $\AB$ the affine Brauer category.

In order to give  a $\kappa$-basis of  $\Hom_{\AB}(\ob m,\ob s), \forall \ob m, \ob s\in \AB$, we need the notion of  a {\em  dotted   $(m,s)$-Brauer diagram} as follows.

A point on a strand of a Brauer diagram is called a critical point if it is either an endpoint or  a point such that  the tangent line at it  is a horizontal line.
    For example, there are three  critical points on both  $U$ and $A$.
 A segment of  a  Brauer diagram is defined to be a connected component of the diagram obtained when all crossings and critical points are deleted.
  A dotted   $(m,s)$-Brauer diagram  is  an $(m,s)$-Brauer diagram (with bubbles) such that there are finitely many $ \bullet$'s (called dots)
on each  segment of the diagram.
   If there are  $h$  $ \bullet$'s on a segment, then such $\bullet$'s  can be viewed as  $\underset {h} {\underbrace{\xdot\circ \ldots \circ \xdot}}$, and will be denoted by  $ \bullet~ h$.
    For example,
   $$\begin{tikzpicture}[baseline = 5pt, scale=0.5, color=\clr]
        \draw[-,thick] (0,2) to[out=down,in=left] (1,0.5) to[out=right,in=down] (2,2);
      \draw (1.6,0.8) \bdot;  \draw (1.9,1.2) \bdot;
    \end{tikzpicture}~ =~ \begin{tikzpicture}[baseline = 5pt, scale=0.5, color=\clr]
        \draw[-,thick] (0,2) to[out=down,in=left] (1,0.5) to[out=right,in=down] (2,2);
      \draw (1.6,0.8) \bdot;  \draw (2,0.8) node{\footnotesize{$2$}};
    \end{tikzpicture} =(\begin{tikzpicture}[baseline = 10pt, scale=0.5, color=\clr]
                \draw[-,thick] (0,0.5)to[out=up,in=down](0,1.5);
    \end{tikzpicture}\otimes X^2)\circ U; \quad \begin{tikzpicture}[baseline = 7.5pt, scale=0.5, color=\clr]
            \draw[-,thick] (0,0) to[out=up, in=down] (1,2);
            \draw[-,thick] (0,2) to[out=up, in=down] (0,2.2);
            \draw[-,thick] (1,0) to[out=up, in=down] (0,2);
            \draw[-,thick] (1,2) to[out=up, in=down] (1,2.2);
             \draw(0,1.9)\bdot;\draw(0.2,1.4)\bdot;
        \end{tikzpicture}=\begin{tikzpicture}[baseline = 7.5pt, scale=0.5, color=\clr]
            \draw[-,thick] (0,0) to[out=up, in=down] (1,2);
            \draw[-,thick] (0,2) to[out=up, in=down] (0,2.2);
            \draw[-,thick] (1,0) to[out=up, in=down] (0,2);
            \draw[-,thick] (1,2) to[out=up, in=down] (1,2.2);
              \draw(0.2,1.4)\bdot;\draw (-0.2,1.4) node{\footnotesize{$2$}};
        \end{tikzpicture}=(X^2\otimes \begin{tikzpicture}[baseline = 10pt, scale=0.5, color=\clr]
                \draw[-,thick] (0,0.5)to[out=up,in=down](0,1.5);)
    \end{tikzpicture})\circ S.
   $$
Let $\mathbb D_{m,s} $ be  the set of all dotted $(m,s)$-Brauer diagrams.   So, any  $d\in \mathbb D_{m,s}  $ is obtained by tensor product and composition of $U$, $A$, $S$, $X$ and  $\begin{tikzpicture}[baseline = 10pt, scale=0.5, color=\clr]
                \draw[-,thick] (0,0.5)to[out=up,in=down](0,1.5);
    \end{tikzpicture}$ and hence can be  interpreted as a morphism  in $\Hom_{\AB}(\ob m,\ob s)$. For example, $d\in \mathbb D_{6,2}$ if
 $$d=\begin{tikzpicture}[baseline = 25pt, scale=0.35, color=\clr]
        \draw[-,thick] (5,0) to[out=up,in=down] (8,5);
         \draw[-,thick] (10,0) to[out=up,in=down] (6,5);
          \draw[-,thick] (2.6,5) to (2.5,5) to[out=left,in=up] (1.5,4)
                        to[out=down,in=left] (2.5,3)
                        to[out=right,in=down] (3.5,4)
                        to[out=up,in=right] (2.5,5);
                         \draw (-1,0.5) \bdot;\draw (-0.4,0.5) node{\footnotesize{$k$}};\draw (6.2,4) \bdot;\draw (6.8,4) node{\footnotesize{$h$}};
                         \draw (3.5,4) \bdot;\draw (3.9,4) node{\footnotesize{$l$}};\draw (6.2,2.2) \bdot;
        \draw[-,thick] (-1,0) to[out=up,in=left] (1,1.5) to[out=right,in=up] (3,0);
         \draw[-,thick] (2,0) to[out=up,in=left] (4,2.5) to[out=right,in=up] (6,0);
           \end{tikzpicture}.  $$
 Conversely,
   if a diagram $d$  is  obtained by tensor product and composition of
the generators $A, U, S,X$  and the identity  morphism \begin{tikzpicture}[baseline = 10pt, scale=0.5, color=\clr]
               \draw[-,thick] (0,0.5)to[out=up,in=down](0,1.5);\end{tikzpicture}
    of $\AB$ such that there are $m$ (resp., $s$) points  on the bottom (resp., top) row, then $d\in\mathbb D_{m,s} $. So
$\Hom_{\AB}(\ob m,\ob s)$ is spanned $\mathbb D_{m,s} $.

\begin{Defn}\label{D:N.O. dotted  OBC diagram}
A $d\in \mathbb D_{m,s} $    is  \emph{normally ordered} if
\begin{enumerate}
\item all of its bubbles are crossing-free,  and there  are no other strands shielding any of them from the leftmost edge of the  picture;
    \item each bubble has even number of $\bullet$'s on the leftmost boundary  of the bubble (e.g.,  \begin{tikzpicture}[baseline = 5pt, scale=0.5, color=\clr]
        \draw[-,thick] (0.6,1) to (0.5,1) to[out=left,in=up] (0,0.5)
                        to[out=down,in=left] (0.5,0)
                        to[out=right,in=down] (1,0.5)
                        to[out=up,in=right] (0.5,1);
        \draw (0,0.5) \bdot;
        \draw (0.4,0.5) node{\footnotesize{$2$}};
    \end{tikzpicture} is allowed but \begin{tikzpicture}[baseline = 5pt, scale=0.5, color=\clr]
        \draw[-,thick] (0.6,1) to (0.5,1) to[out=left,in=up] (0,0.5)
                        to[out=down,in=left] (0.5,0)
                        to[out=right,in=down] (1,0.5)
                        to[out=up,in=right] (0.5,1);
        \draw (0,0.5) \bdot;
        \draw (0.4,0.5) node{\footnotesize{$3$}};
    \end{tikzpicture}, \begin{tikzpicture}[baseline = 5pt, scale=0.5, color=\clr]
        \draw[-,thick] (0.6,1) to (0.5,1) to[out=left,in=up] (0,0.5)
                        to[out=down,in=left] (0.5,0)
                        to[out=right,in=down] (1,0.5)
                        to[out=up,in=right] (0.5,1);
        \draw (1,0.5) \bdot;
        \draw (1.4,0.5) node{\footnotesize{$2$}};
    \end{tikzpicture} and  \begin{tikzpicture}[baseline = 5pt, scale=0.5, color=\clr]
        \draw[-,thick] (0.6,1) to (0.5,1) to[out=left,in=up] (0,0.5)
                        to[out=down,in=left] (0.5,0)
                        to[out=right,in=down] (1,0.5)
                        to[out=up,in=right] (0.5,1);
        \draw (0,0.5)\bdot;
       \draw (1,0.5) \bdot; \end{tikzpicture}
     are not allowed);
    \item whenever a $\bullet$ appears on a vertical strand, it is on the boundary of the top row;

    \item   whenever a $\bullet$ appears on either  a cap or a cup,  it is on  either  the leftmost boundary of the cap or the rightmost boundary  of a cup.
\end{enumerate}
\end{Defn}

For any $m, s \in \mathbb N$, define
\begin{equation} \label{ngms} \mathbb {ND}_{m,s}:= \{d\in \mathbb{D}_{m,s} \mid d \text{ is normally ordered}\}.\end{equation}  The following are two elements in $\mathbb{D}_{5,5}$. The  right one is in  $\mathbb {ND}_{5,5}$,  whereas  the left one is not:
\begin{equation}\label{two AOBC diagrams}
    \begin{tikzpicture}[baseline = 25pt, scale=0.35, color=\clr]
        \draw[-,thick] (2,0) to[out=up,in=down] (0,5);
        \draw[-,thick] (6,0) to[out=up,in=down] (6,5);
        \draw[-,thick] (7,0) to[out=up,in=down] (7,5);
            \draw[-,thick] (0,0) to[out=up,in=left] (2,1.5) to[out=right,in=up] (4,0);
        \draw[-,thick] (2,5) to[out=down,in=left] (3,4) to[out=right,in=down] (4,5);
        \draw[-,thick] (2,3.1) to (2,3) to[out=down,in=right] (0.5,2) to[out=left,in=down] (-1,3)
                        to[out=up,in=left] (0.5,4) to[out=right,in=up] (2,3);
    \draw(4,0.4)\bdot;
    \draw(-1,3)\bdot;
    \draw(2,0.4)\bdot;
    \end{tikzpicture}~,
    \qquad\qquad\qquad
    \begin{tikzpicture}[baseline = 25pt, scale=0.35, color=\clr]
        \draw[-,thick] (2,0) to[out=up,in=down] (0,5);
        \draw[-,thick] (6,0) to[out=up,in=down] (6,5);
         \draw[-,thick] (6.7,0) to[out=up,in=down] (6.7,5);
               \draw[-,thick] (0,0) to[out=up,in=left] (2,1.5) to[out=right,in=up] (4,0);
        \draw[-,thick] (2,5) to[out=down,in=left] (3,4) to[out=right,in=down] (4,5);
        \draw[-,thick] (-1.1,4) to (-1,4) to[out=left,in=up] (-2,3) to[out=down,in=left] (-1,2)
                        to[out=right,in=down] (0,3) to[out=up,in=right] (-1,4);
                        \draw(0,0.3)\bdot;
    \draw(-2,3)\bdot; \draw(6,4.6)\bdot;\draw(3.95,4.6)\bdot;
    \draw(0,4.6)\bdot;
     \draw (-1.5,3) node{$4$};\draw (0.5,0.4) node{$l$};
           \end{tikzpicture}~.
\end{equation}

 Suppose that  $d, d'\in \mathbb {D}_{m,s} $. We say that $d$ and $d'$    are \emph{equivalent} and write $d\sim d'$ if their underlying $(m,s)$-Brauer diagrams without $\bullet$'s  are equivalent and their corresponding strands have the same number of $\bullet$'s. One can restrict this equivalence relation on $ \mathbb {ND}_{m,s} $. Let  $\mathbb {ND}_{m,s}/\sim $ be the set of all equivalence classes of  $\mathbb {ND}_{m,s} $.

 By Lemma~\ref{reguequai}, $d=d'$ in $\AB$ if  $d, d'\in \mathbb {ND}_{m,s} $ and $d\sim d'$. So any equivalence class of  $\mathbb {ND}_{m,s} $ can be identified with any element in it.

\begin{THEOREM}\label{Cyclotomic basis conjecture}
    For any $m,s\in \mathbb N$,    $\Hom_{\AB}(\ob m,\ob s)$ has $\kappa$-basis given by  $\mathbb {ND}_{m,s}/\sim $.
   \end{THEOREM}
For any $k\in\mathbb N$, define 
 \begin{equation}\label{defofdelta}
\Delta_k:= \begin{tikzpicture}[baseline = 5pt, scale=0.5, color=\clr]
        \draw[-,thick] (0.6,1) to (0.5,1) to[out=left,in=up] (0,0.5)
                        to[out=down,in=left] (0.5,0)
                        to[out=right,in=down] (1,0.5)
                        to[out=up,in=right] (0.5,1);
        \draw (0,0.5) \bdot;
        \draw (-0.4,0.5) node{\footnotesize{$k$}};
    \end{tikzpicture}= \lcap\circ (\xdot~ \begin{tikzpicture}[baseline = 5pt, scale=0.5, color=\clr]
     \draw[-,thick] (0,0.15) to (0,1.15); \end{tikzpicture} ~)^k\circ \lcup\in \End_{\AB}(\ob 0).
\end{equation}
By ~\cite[Proposition~2.2.10]{EGNO},
 $\End_{\AB}(\ob 0)$ is commutative
and  $\Delta_k\circ \Delta_l=\Delta_k \otimes \Delta_l$ for all admissible $k,l$. Note that   $\Delta_k\otimes \Delta_l$ is also written as $\Delta_k\Delta_l$.
 In Lemma~\ref{admissible}, we will prove that
\begin{equation}\label{admidekts}
2\Delta_k=-\Delta_{k-1}+\sum_{j=1}^{k}(-1)^{j-1}\Delta_{j-1}\Delta_{k-j}, \text{ for }  k=1,3,\ldots.
\end{equation}Later on, we always assume that  \begin{equation}\label{omega} \omega=\{\omega_i\in \kappa\mid  i\in \mathbb N\}.\end{equation}
If $\Delta_j$'s are specialized at scalars $\omega_j$'s, then
\begin{equation}\label{admomega}
2\omega_k=-\omega_{k-1}+\sum_{j=1}^{k}(-1)^{j-1}\omega_{j-1}\omega_{k-j}, \text{ for } k=1,3,\ldots.
    \end{equation}
So, $\omega$  is  admissible in the sense of \cite[Definition~2.10]{AMR}.
By Theorem~\ref{Cyclotomic basis conjecture}, $ \End_{\AB}(\ob 0)$ is the polynomial algebra $\kappa[\Delta_0,\Delta_2,\ldots]$.
 $\AB$ can be viewed as a $\kappa[\Delta_0]$-linear category with $\Delta_0g:=\Delta_0\otimes g$ for any $g\in\Hom_{\AB}(\ob m,\ob s)$. Given $\omega_0\in\kappa$, define $\AB(\omega_0):=\kappa\otimes _{\kappa[\Delta_0]}\AB$ viewing $\kappa$ as a  $\kappa[\Delta_0]$-module with $\Delta_0$ acting as $\omega_0$.
  Then
  $\AB(\omega_0)$ is the   $\kappa$-linear category obtained from $\AB$ by specializing  $\Delta_0$ at $\omega_0$.
  Suppose that $\omega$ satisfies the admissible condition  in \eqref{admomega}.
  Similarly, one can define $\AB(\omega):=\kappa\otimes _{\kappa[\Delta_2,\Delta_4,\ldots]}\AB(\omega_0)$ viewing $\kappa$ as  a $\kappa[\Delta_2, \Delta_4,\ldots]$-module with $\Delta_{2j}$ acting as $\omega_{2j}$, $j\geq 1$.

  { For any $m,s\in\mathbb N$, define }
  \begin{equation}\label{barndms}\bar{\mathbb {ND}}_{m,s}:= \{d\in \mathbb{ND}_{m,s} \mid d \text{ has no  bubbles}\}.\end{equation}
  The following result follows from Theorem~\ref{Cyclotomic basis conjecture}, immediately.

\begin{Cor}\label{omega00s}
Suppose  $m,s\in\mathbb N$.
\begin{enumerate}
\item $\Hom_{\AB(\omega_0)}(\ob m, \ob s)$ has basis
$\widetilde{\mathbb {ND}}_{m,s}:=\{\Delta_2^{k_2} \Delta_4^{k_4} \cdots d\mid  d\in \bar{\mathbb {ND}}_{m,s}/\sim, k_{2i}\in\mathbb N, \forall i\geq 1\}$.

\item  $\Hom_{\AB(\omega)}(\ob m, \ob s)$ has basis $\bar{\mathbb {ND}}_{m,s}/\sim$.
\end{enumerate}
\end{Cor}

For any $\kappa$-linear  monoidal category $\mathcal C$, a \emph{right tensor ideal} $I$ of $\mathcal C$ is the collection of  $\kappa$-submodules $ I(\ob m,\ob s)\subset  \Hom_{\mathcal C}(\ob m,\ob s)$,  for all $\ob m,\ob s \in \mathcal C$, such that
\begin{equation}\label{tensotideal}
 h\circ g\circ f\in I(\ob m ,\ob t), \text{ and $ g\otimes 1_{\ob t}\in I(\ob n\otimes \ob t ,\ob s\otimes \ob t)$}\end{equation}
 whenever $f \in \Hom_{\mathcal C}(\ob m,\ob n)$,  $g\in I(\ob n,\ob s)$,  $h\in \Hom_{\mathcal C}(\ob s,\ob t)$,  for all  $\ob m,  \ob n,  \ob s,  \ob t\in \C$. The quotient category  $ \mathcal C/I$  is the category with the same objects as $\mathcal C$ and   $\Hom_{\mathcal C/I}(\ob m,\ob s)=\Hom_{\mathcal C}(\ob m,\ob s)/I(\ob m, \ob s)$,  for all $  \ob m, \ob s\in \C/I$.
\begin{Defn}\label{COBC defn}
Fix a positive integer  $a$. Let  $f(t)=\prod_{1\leq i\leq a}(t-u_i)\in\kappa[t]$. The \emph{cyclotomic Brauer category} $\CB^f:=\AB/I $,   where  $I$ is  the right tensor ideal of $\AB$ generated by $f(X)$.
\end{Defn}
 Since  $\CB^f$ is a  quotient category of  $\AB$,
 $\Hom_{\CB^f}(\ob m, \ob s)$ is spanned $\mathbb {D}_{m,s}$.
Write     $f(t)=\sum_{j=0}^a b_j t^j$, where  $f(t)$ is given in Definition~\ref{COBC defn}.
Then \begin{equation}\label{defofbju}
b_j=(-1)^{a-j}e_{a-j}(\mathbf u),\ \  0\leq j\leq a,
 \end{equation}
 where   $e_k(\mathbf u)$ is the elementary symmetric  function in $u_1,u_2,\ldots,u_a$ with degree $k$,   $0\leq k\leq a$. Later on, we always assume that   \begin{equation}\label{bfu} \mathbf u= \{u_j\in\kappa\mid 1\leq j\leq a\}.\end{equation} In  $\CB^f$, we have  $f(X)=0$. So
$$
\Delta_k=-\sum_{j=1}^ab_{a-j}\Delta_{k-j},
 \forall k\geq a.$$

\begin{Defn}\cite[Definition~2.5]{G09}\label{wka}  $\omega$ is called weakly admissible if it satisfies both \eqref{admomega} and $
\omega_k=-\sum_{j=1}^ab_{a-j}\omega_{k-j}$ for all $k\geq a$,
where $b_j$'s are given in \eqref{defofbju}.
\end{Defn}

Thanks to  Definition~\ref{wka},  $\omega$ is determined by $\{\omega_i\mid  i=0,2, \ldots, 2\lfloor \frac{a-1}{2}\rfloor\}$ if it is weakly admissible.

\begin{Defn}\label{COBC1}  Suppose that    $\omega$  is weakly admissible.
The (specialized) cyclotomic Brauer category  $\CB^f(\omega):=\AB/J$, where $J$ is the right tensor ideal  generated by $f(X)$ and $\Delta_k-\omega_k$ for all $k\in\mathbb N$.
\end{Defn}

\begin{Defn}
 \cite[Definition~3.6,~Lemma~3.8]{AMR}  Let  $u$ be an indeterminate. $\omega$ is called
$\mathbf u$-admissible    if
\begin{equation}\label{admc}
\sum_{i\geq 0} \frac{\omega_i}{ u^{i}}+u-\frac{1}{2}=(u-\frac{1}{2}(-1)^a)\prod_{i=1}^a \frac{u+u_i}{u-u_i}.
\end{equation}
\end{Defn}

Thanks to \cite[Corollary~3.9]{AMR},  $\omega$ is admissible in the sense of \cite[Definition~2.10]{AMR} if
  it is $\mathbf u$-admissible. In other words, it  satisfies \eqref{admomega}.  Moreover, it is weakly admissible (see \cite[Theorem~5.2]{G09}).

For any $k\in \mathbb N$, let  $q_k(\mathbf u)$ be the Schur's $q$-function in the indeterminates  $u_1,u_2,\ldots, u_a$. Then $q_k(\mathbf u)$'s   are polynomials  which satisfy
$$
\prod_{i=1}^a\frac{1+u_iu}{1-u_iu}=\sum_{k\geq 0}q_k(\mathbf u)u^k.
$$
By \cite[Lemma~3.8]{AMR},   $\omega$ is $\mathbf u$-admissible if and only if
\begin{equation}\label{wdeterminedbyu}
\omega_k=q_{k+1}(\mathbf u)-\frac{1}{2}(-1)^aq_{k}(\mathbf u) + \frac{1}{2}\delta_{k,0}, \text{ for $k\in\mathbb N$.}
\end{equation}
So, $\omega$ is determined by $\mathbf u$ if it is $\mathbf u$-admissible.

Let $\mathbb {ND}^a_{m,s}:= \{d\in \mathbb{ND}_{m,s} \mid   \text{ there are at most $a-1$  $\bullet$'s on each strand of $d$} \}$. Define
\begin{equation}\label{ndams} \bar {\mathbb {ND}}^a_{m,s}:=\mathbb {ND}^a_{m,s}\cap \bar{\mathbb {ND}}_{m,s}.\end{equation}
 One can restrict the equivalence relation on $\mathbb {ND}_{m,s}$ to its subset $\bar{\mathbb {ND}}^a_{m,s}$. Let  $\bar{\mathbb {ND}}^a_{m,s}/\sim$ be the set of equivalence classes of  $\bar{\mathbb {ND}}^a_{m,s}$.

By Lemma~\ref{reguequai},  $d=d'$ in $\CB^f(\omega)$ if $d, d' \in\bar{\mathbb {ND}}^a_{m,s}$ and $d\sim d'$. So, each equivalence class in  $\bar{\mathbb {ND}}^a_{m,s}/\sim$ can be identified with  any element in it.

\begin{THEOREM}\label{basis of cyc thm2} For any $m, s\in \mathbb N$,  $\Hom_{\CB^f(\omega)}(\ob m, \ob s)$ has $\kappa$-basis  given by  $\bar{\mathbb {ND}}^a_{m,s}/\sim$     if and only if   $\omega$ is    $\mathbf u$-admissible in the sense of \eqref{admc}.
 \end{THEOREM}

In summary, we give a table to collect  the different types of  diagrammatic basis. In the last row, we have to assume that   $\omega$ is  $\mathbf u$-admissible. 

\medskip
\centerline{\begin{tabular}{|c|c|}
\hline
Category  $\C$& $\kappa$-basis of  $\Hom_{\C}(\ob m,\ob s)$\\
\hline
 $\B$ & $\mathbb B_{m,s}/\sim$: equivalence classes of   $(m,s)$-Brauer diagrams\\
\hline
  $\B(\omega_0)$& $\bar{\mathbb B}_{m,s}/\sim$:=$\{d\in\mathbb B_{m,s}/\sim\mid d\text{ has no bubbles} \}$ \\
\hline
  $\AB$&$\mathbb {ND}_{m,s}/\sim:$ equivalence classes of normally ordered dotted $(m,s)$-Brauer diagrams \\
\hline
 $\AB(\omega_0)$& $\widetilde{\mathbb {ND}}_{m,s}$:= $\{d\in  \mathbb {ND}_{m,s}/\sim\mid d \text{ has no }\Delta_0\} $ \\
\hline
 $\AB(\omega)$& $\bar{\mathbb {ND}}_{m,s}/\sim:=\{d\in \mathbb {ND}_{m,s}/\sim\mid d \text{ has no bubbles}\} $\\
\hline
  $\CB^f(\omega)$  &$\bar{\mathbb {ND}}^a_{m,s}/\sim$:=$\{d\in \bar{\mathbb {ND}}_{m,s}/\sim \mid \text{each strand of $d$ has  at most $a-1$  $\bullet$'s}$ \} \\
\hline
\end{tabular}
}

\medskip
 We organize the paper as follows. In section~2, we establish connections between Brauer category $\B$ and simple  Lie algebras in types $B, C$ and $D$. As an application, we prove Theorem~A. In section~3,
 we give explicit bases of morphism  spaces  in affine Brauer category $\AB$. Thus we can establish   algebra
isomorphisms  between affine Nazarov-Wenzl algebras in \cite{Na} and certain endomorphism algebras in $\AB$. In section~4, we prove Theorem~C. This enables us to establish
 algebra isomorphisms between
  cyclotomic (or level $k$)  Nazarov-Wenzl algebras in  \cite{AMR} and certain endomorphism algebras in $\CB^f(\omega)$.
 As a by-product, we will give isomorphisms  between level $k$ Nazarov-Wenzl algebras and certain endomorphism algebras in the BGG parabolic categories  $\mathcal O$
associated to  orthogonal and symplectic Lie algebras.
Using standard arguments in \cite{AST, RS1, RS2}, we obtain formulae to compute  decomposition numbers of  level $k$ Nazarov-Wenzl algebras in section~ 5.  The  level $2$ case was dealt with  in \cite{ES}.

\section{Connections to orthogonal and symplectic Lie algebras}\label{iso}

 Throughout,  let $V$ be the $N$-dimensional  complex space.
   As a complex space,
  the general linear Lie algebra
  $\mathfrak{gl}_N$ is $\End_{\mathbb C}(V) $. The Lie bracket $[\ \ \ ]$  is defined as $[x, y]=xy-yx$ for all $x, y\in \mathfrak {gl}_N$.    Let $ (\ ,\ ): V\otimes V\rightarrow \mathbb C$ be  the non-degenerate bilinear form satisfying
 \begin{equation}
\label{bilinearf}
(x,y)=\varepsilon(y,x),\end{equation}
where $\epsilon\in \{-1, 1\}$.  Define  \begin{equation}\label{defofg}
\mathfrak g=\{g\in\mathfrak{gl}_N \mid (gx,y)+(x,gy)=0 \text{ for all $x,y\in V$}\}.
\end{equation}
If $\epsilon=1$, $\mfg$ is known as the orthogonal Lie algebra $\mathfrak{so}_{N}$. When $\epsilon=-1$, $N$ has to be even and $\mfg$ is known as the
 symplectic Lie algebra $ \mathfrak {sp}_N$. Later on, we will denote $\epsilon$ by $\epsilon_{\mfg}$ in order to emphasis $\mfg$.

 Define $\underline {2n+1}=\{1, 2, \ldots, n, 0, -n, \ldots -2, -1\}$ and $\underline {2n}=\underline{2n+1}\setminus\{0\}$.
 Choose  a basis $\{v_i\mid i\in \underline N\}$ of $V$ such that
 \begin{equation}\label{vivjb}
(v_i,v_j)=\delta_{i,-j}=\varepsilon_{\mathfrak g}(v_j,v_i), \text{for  all $i, j\in \underline N$, and   $i\ge 0$.}
\end{equation}
Let $V^*$ be the linear dual of $V$ with dual basis  $\{ v^*_i\mid i \in \underline N\}$  such that $  v^*_i(v_j)=\delta_{i, j}$ for all admissible $i$ and $j$.  Thanks to \eqref{defofg}, there is a $\mathfrak g$-module isomorphism  $\psi: V\rightarrow V^*$ such that   $\psi (v)(x)= (v,x)$ for all $x, v\in V$. So,
$V$ can be identified with $V^*$ under $\psi$ and
\begin{equation}\label{vistar}
v_i^*=\begin{cases}  v_{-i},  &  \text{if $i\leq 0$ and $i\in \underline N$,}\\
   \varepsilon_{\mathfrak g} v_{-i}, & \text{if $i>0$ and $i\in \underline{N}$}.\\
   \end{cases}
\end{equation}
For all $i, j\in \underline N$, let $E_{i, j}$ be the usual matrix unit with respect to the basis $\{v_k\mid k\in \underline N\}$ of $V$. Then   $\mathfrak{g}$ is spanned by $\{F_{i, j}\mid i, j\in \underline {N}\}$, where (see \cite[(7.9)]{Mol})
\begin{equation}\label{fij} F_{i,j}=E_{i,j}-\theta_{i,j} E_{-j,-i}, \ \  \text{ and
$\theta_{i,j}=\begin{cases}
                    1, & \text{if $\mathfrak{g}=\mathfrak {so}_{N}$,}\\
                      \text{sgn}(i)\text{sgn}(j), & \text{if $\mathfrak{g}=\mathfrak {sp}_{N}$.}\\
                      \end{cases}$}\end{equation}
Further,  $\mfg$  has basis $S(\mfg)$, where
                      \begin{equation}\label{basis} S(\mfg)=\begin{cases} \{F_{i,i}\mid 1\le i\le n\}\cup\{F_{\pm i, \pm j}\mid 1\le i<j\le n\}\cup \{F_{0, \pm i}\mid 1\le i\le n\}, &\text{if $\mathfrak g=\mathfrak{so}_{2n+1}$,}\\ \{F_{i, i}, F_{-i, i}, F_{i, -i}\mid 1\le i\le n\}\cup \{F_{\pm i, \pm j}\mid 1\le i<j\le n\}, &\text{if $\mathfrak g=\mathfrak{sp}_{2n}$,}\\
\{F_{i, i}\mid 1\le i\le n\}\cup \{F_{\pm i, \pm j}\mid 1\le i<j\le n\}, &\text{if $\mathfrak g=\mathfrak{so}_{2n}$.}\\
\end{cases}\end{equation}
The Cartan subalgebra $\mathfrak h$   has  basis $ \{h_i\mid 1\le i\le n\}$, where
\begin{equation}\label{hi} h_i=F_{i,i}, \text{  $1\leq i\leq n$.}\end{equation}
There is a triangular decomposition   $\mfg=\mathfrak n^{-}\oplus \mathfrak h\oplus \mathfrak n^+$ such that the positive part $\mathfrak n^+$ has  basis $S(\mathfrak n^+)$, where
\begin{equation}\label{basisofg} S(\mathfrak n^+)=\begin{cases} \{F_{ i,\pm j}\mid 1\le i<j\le n \}\cup \{F_{0,-i}\mid 1\le i\le n \}, &\text{if $\mathfrak{g}=\mathfrak {so}_{2n+1}$,}\\
\{F_{i,-i}\mid 1\le i\le n \}\cup\{F_{ i,\pm j}\mid 1\le i<j\le n \}, &\text{if $\mathfrak{g}=\mathfrak {sp}_{2n}$,}\\
\{F_{ i,\pm j}\mid 1\le i<j\le n\}, &\text{if $\mathfrak{g}=\mathfrak {so}_{2n}$.}\\
\end{cases}\end{equation}

Let $\mathfrak h^*$ be the linear dual  of $\mathfrak h$ with  dual basis $\{\epsilon_i\mid 1\le i\le n\}$ such that  $\epsilon_i(h_j)=\delta_{i, j}$,  for all admissible $i, j$.
Each element $\lambda\in \mathfrak h^*$, called a weight, is of form   $\sum_{i=1}^n \lambda_i\epsilon_i$.
 Let $\U(\mathfrak f )$ be   the universal enveloping algebra associated with any  complex Lie algebra  $\mathfrak f$.  A nonzero vector  $m$ in a left $\U(\mfg)$-module $ M$ is of weight $\lambda$ if $h_i m=\lambda_i m$, for     $1\le i\le n$. A weight vector  $m\in M$  is called a highest weight  vector  if $\mathfrak n^+ m=0$. A highest weight module is a left $\U(\mfg)$-module generated by a highest weight vector.

The root system of $\mfg$ is $R =R^+ \cup(-R^+)$, where the set of positive roots
\begin{equation}\label{positiveroot}
R^+=\left\{
         \begin{array}{ll}
           \{\varepsilon_i\pm \varepsilon_j\mid 1\leq i<j\leq n\}\cup \{\varepsilon_i\mid 1\leq i\leq n\}, & \hbox{$\mfg=\mathfrak {so}_{2n+1}$;} \\
           \{\varepsilon_i\pm \varepsilon_j\mid 1\leq i<j\leq n\}\cup \{2\varepsilon_i\mid 1\leq i\leq n\}, & \hbox{$\mfg=\mathfrak {sp}_{2n}$;} \\
            \{\varepsilon_i\pm \varepsilon_j\mid 1\leq i<j\leq n\}, & \hbox{$\mfg=\mathfrak {so}_{2n}$.}
         \end{array}
       \right.
\end{equation}    The set  of  simple roots is  $\Pi$, where $\Pi=\{\alpha_1, \alpha_2, \ldots, \alpha_{n-1}, \alpha_n\}$,
  $\alpha_i=\epsilon_i-\epsilon_{i+1}$,  $ 1\le i\le n-1$, and \begin{equation}\label{aln}\alpha_n=\begin{cases} \epsilon_n, &\text{if $\mathfrak g=\mathfrak{so}_{2n+1}$, }\\
2\epsilon_n, &\text{if $\mathfrak g=\mathfrak{sp}_{2n}$, }\\
\epsilon_{n-1}+\epsilon_n, &\text{if $\mathfrak g=\mathfrak{so}_{2n}$. }\\
\end{cases}
\end{equation}
 A weight $\lambda\in \mathfrak h^*$ is called a \emph{dominant integral  weight} if  $$\langle \lambda, \alpha\rangle=\frac{2(\lambda, \alpha)} {(\alpha, \alpha)}\in \mathbb N, \ \ \text{for any $\alpha\in \Pi$,}$$  where
  $(\ , \ )$ is the  symmetric bilinear form  defined on $\mathfrak h^*$ such that $(\epsilon_i, \epsilon_j)=\delta_{i, j}$.

Let $\U(\mathfrak g)\text{-mod}$ (resp., $\U(\mathfrak g)\text{-fmod}$) be the category of all (resp., finite dimensional) left $\U(\mathfrak g)$-modules. It is known that both  $\U(\mathfrak g)\text{-mod}$ and
 $\U(\mathfrak g)\text{-fmod}$  are   monoidal categories  in which the unit object is   the trivial $\mfg$-module  $\mathbb C$.\footnote{
  In this paper, we consider both $\U(\mathfrak g)\text{-fmod}$ and $\U(\mathfrak g)\text{-mod}$ as  strict monoidal categories  by identifying  $(M_1\otimes M_2)\otimes M_3$ (resp.,   $M_1\otimes \mathbb C$ and  $ \mathbb C\otimes M_1$)    with $M_1\otimes (M_2\otimes M_3)$ (resp., $M_1$)  for all admissible  $M_1,M_2,M_3$.}
    Moreover,  $\U(\mathfrak g)\text{-fmod}$ is symmetric with braiding $\sigma_{M,L}: M\otimes L  \rightarrow L\otimes M$ defined  by
\begin{equation}\label{eeddjsigma}
\sigma_{M,L}(m\otimes l)=l\otimes m, \text{$\forall m\in M, l\in L$.}
\end{equation}
Let $M^*$ be the linear dual of an object  $M$ in $\U(\mathfrak g)\text{-fmod}$.
If  $M$ has basis $\{m\in B_M\}$, then  we denote by  $\{m^*\mid m\in B_M\}$  the dual  basis of $M^*$ such that $m^*(l)=\delta_{m, l}$, $m,l\in B_M$.
Define two morphisms  $\eta_{M}: \mathbb C\rightarrow M \otimes M^*$ and $\epsilon_{M}: M^*\otimes M\rightarrow \mathbb C$ such that
\begin{equation}\label{etavshd}
\eta_{M}(1)=\sum_{m\in B_M}m\otimes m^*, \text{ and  $\epsilon_{M}(f\otimes m)=f(m)$, for all $f\in M^*$ and $m\in M$.}
\end{equation}
 It is known that  $\eta_M$ and $\epsilon_M$ are the unit morphism and the counit morphism in the sense of
\eqref{rigdual}. So, $\U(\mathfrak g)\text{-fmod}$ is rigid.


\begin{Prop}\label{level1s} Suppose  $\kappa=\mathbb C$.
There is a  monoidal functor
$\Phi: \B\rightarrow \U(\mathfrak g)\text{-fmod}$ such that
 $\Phi(\ob 0)=\mathbb C$, $\Phi(\ob 1)=V$ and
\begin{equation}\label{modstr}
 \Phi(U)( 1)= \sum_{i\in \underline N}v_i\otimes v_i^*, \ \
 \Phi(A)(u\otimes v)=(u,v), \ \
 \Phi(S)(u \otimes v)= \varepsilon_{\mathfrak g} v\otimes u,   \forall  u,v\in V.
\end{equation}
\end{Prop}
\begin{proof}The $\mathbb C$-linear maps $\Phi(U),\Phi(A), \Phi(S)$ are actually the maps $ \eta_{V}$, $\epsilon_{V}$, and $\epsilon_{\mfg}\sigma_{V,V}$ in \eqref{eeddjsigma}--\eqref{etavshd}, respectively. So, they are $\U(\mfg)$-homomorphisms.
It is proven  in \cite[Lemma~3.1]{LZ} that $\Phi$ keeps   the relations in \eqref{OBR1}--\eqref{OBR4}.
So, the functor $\Phi$ is well-defined and the result is proven.
\end{proof}
Let $\B(\epsilon_{\mfg}N)$ be  the $\kappa$-linear monoidal category obtained from $\B$ by imposing the relation $\Delta_0=\epsilon_{\mfg}N$.
By \eqref{modstr},
\begin{equation}\label{ewjdngn}
\Phi(\Delta_0)=\epsilon_{\mfg}N.
\end{equation}
So, the functor $\Phi$ factors through  $\B(\epsilon_{\mfg}N)$. See also \cite[Theorem~3.4]{LZ} when $\U(\mfg)$ is replaced by symplectic and orthogonal   groups.

We want to establish a  relationship between   $\B$  and  the oriented Brauer category $\OB$,   a $\kappa$-linear monoidal category with objects generated by
two objects $\uparrow, \down$  \cite{BCNR}. So,
the set of  objects in $\OB$  is  the set $\langle\uparrow,\downarrow\rangle$ of all words in the alphabets $\uparrow, \downarrow$. In particular, the unit object is  the empty word $\varnothing$.
\begin{Defn}\cite[\S1]{BCNR}
The \emph{oriented Brauer category} $\OB$
is  the $\kappa$-linear  monoidal  category generated by two objects $\uparrow, \down$;  three  morphisms $\lcupo:\varnothing\rightarrow\uparrow\down, \ \ \lcapo: \down\uparrow\rightarrow\varnothing, \ \ \swapo: \uparrow\uparrow\rightarrow\uparrow\uparrow$  subject to the following relations:
 \begin{equation}\label{aOB relations 1 (symmetric group)}
        \begin{tikzpicture}[baseline = 10pt, scale=0.5, color=\clr]
            \draw[-,thick] (0,0) to[out=up, in=down] (1,1);
            \draw[->,thick] (1,1) to[out=up, in=down] (0,2);
            \draw[-,thick] (1,0) to[out=up, in=down] (0,1);
            \draw[->,thick] (0,1) to[out=up, in=down] (1,2);
        \end{tikzpicture}
        ~=~
        \begin{tikzpicture}[baseline = 10pt, scale=0.5, color=\clr]
            \draw[-,thick] (0,0) to (0,1);
            \draw[->,thick] (0,1) to (0,2);
            \draw[-,thick] (1,0) to (1,1);
            \draw[->,thick] (1,1) to (1,2);
        \end{tikzpicture}
        \ ,\qquad
        \begin{tikzpicture}[baseline = 10pt, scale=0.5, color=\clr]
            \draw[->,thick] (0,0) to[out=up, in=down] (2,2);
            \draw[->,thick] (2,0) to[out=up, in=down] (0,2);
            \draw[->,thick] (1,0) to[out=up, in=down] (0,1) to[out=up, in=down] (1,2);
        \end{tikzpicture}
        ~=~
        \begin{tikzpicture}[baseline = 10pt, scale=0.5, color=\clr]
            \draw[->,thick] (0,0) to[out=up, in=down] (2,2);
            \draw[->,thick] (2,0) to[out=up, in=down] (0,2);
            \draw[->,thick] (1,0) to[out=up, in=down] (2,1) to[out=up, in=down] (1,2);
        \end{tikzpicture}\ ,
    \end{equation}
    \begin{equation}\label{aOB relations 2 (zigzags and invertibility)}
        \begin{tikzpicture}[baseline = 10pt, scale=0.5, color=\clr]
            \draw[-,thick] (2,0) to[out=up, in=down] (2,1) to[out=up, in=right] (1.5,1.5) to[out=left,in=up] (1,1);
            \draw[->,thick] (1,1) to[out=down,in=right] (0.5,0.5) to[out=left,in=down] (0,1) to[out=up,in=down] (0,2);
        \end{tikzpicture}
        ~=~
        \begin{tikzpicture}[baseline = 10pt, scale=0.5, color=\clr]
            \draw[-,thick] (0,0) to (0,1);
            \draw[->,thick] (0,1) to (0,2);\
        \end{tikzpicture}\
        ,\qquad
        \begin{tikzpicture}[baseline = 10pt, scale=0.5, color=\clr]
            \draw[-,thick] (2,2) to[out=down, in=up] (2,1) to[out=down, in=right] (1.5,0.5) to[out=left,in=down] (1,1);
            \draw[->,thick] (1,1) to[out=up,in=right] (0.5,1.5) to[out=left,in=up] (0,1) to[out=down,in=up] (0,0);
        \end{tikzpicture}
        ~=~
        \begin{tikzpicture}[baseline = 10pt, scale=0.5, color=\clr]
            \draw[-,thick] (0,2) to (0,1);
            \draw[->,thick] (0,1) to (0,0);
        \end{tikzpicture}~,\
      \end{equation}
       \begin{equation}\label{aOB relations 3 (zigzags and invertibility)} \begin{tikzpicture}[baseline = 10pt, scale=0.5, color=\clr]
            \draw[-,thick] (2,2) to[out=down, in=up] (2,1) to[out=down, in=right] (1.5,0.5) to[out=left,in=down] (1,1);
            \draw[->,thick] (1,1) to[out=up,in=right] (0.5,1.5) to[out=left,in=up] (0,1) to[out=down,in=up] (0,0);
            \draw[->,thick] (0.7,0) to[out=up,in=down] (1.3,2);
        \end{tikzpicture}
        ~\text{ is invertible}.
    \end{equation}

\end{Defn}
The arrows in $\lcupo$, $\lcapo$ and $\swapo$ give a consistent orientation to each strand in the diagrams.
The object on each row of a diagram is indicated by the orientations of the endpoints at this row.
If there is no endpoints at a row, then the object at this row is $\varnothing$.
For example, the object at the bottom (resp., top) row of $\lcupo $ is $ \varnothing$ (resp., $\uparrow\downarrow$).

 For any $\ob a,\ob b\in \langle\uparrow,\downarrow\rangle$,
 $\Hom_{\OB}(\ob a,\ob b)$ is spanned by all oriented Brauer diagrams  (with bubbles) of type $\ob a \rightarrow \ob b$\cite[\S1]{BCNR}(see also \cite[\S~3.1]{CK}).
 Here a bubble is $ \begin{tikzpicture}[baseline = -2pt, scale=0.5, color=\clr]
            \draw[->,thick] (0,0) to[out=up, in=left] (0.5,0.5);
            \draw[-,thick] (0.4,0.5) to (0.5,0.5) to[out=right,in=up] (1,0) to[out=down, in=right] (0.5,-0.5) to[out=left,in=down] (0,0);
        \end{tikzpicture}=\begin{tikzpicture}[baseline = -2pt, scale=0.5, color=\clr]
            \draw[-<,thick] (0,0) to[out=up, in=left] (0.5,0.5);
            \draw[-,thick] (0.4,0.5) to (0.5,0.5) to[out=right,in=up] (1,0) to[out=down, in=right] (0.5,-0.5) to[out=left,in=down] (0,0);
        \end{tikzpicture}$ (regardless of orientations).
Let
$\begin{tikzpicture}[baseline = 3pt, scale=0.5, color=\clr]
        \draw[->,thick] (0,0) to[out=up, in=down] (1,1);
        \draw[<-,thick] (1,0) to[out=up, in=down] (0,1);
\end{tikzpicture}: \uparrow\downarrow\rightarrow\downarrow\uparrow$ be  the two sided inverse of $\begin{tikzpicture}[baseline = 3pt, scale=0.5, color=\clr]
        \draw[<-,thick] (0,0) to[out=up, in=down] (1,1);
        \draw[->,thick] (1,0) to[out=up, in=down] (0,1);
\end{tikzpicture}:=\begin{tikzpicture}[baseline = 10pt, scale=0.5, color=\clr]
            \draw[-,thick] (2,2) to[out=down, in=up] (2,1) to[out=down, in=right] (1.5,0.5) to[out=left,in=down] (1,1);
            \draw[->,thick] (1,1) to[out=up,in=right] (0.5,1.5) to[out=left,in=up] (0,1) to[out=down,in=up] (0,0);
            \draw[->,thick] (0.7,0) to[out=up,in=down] (1.3,2);
        \end{tikzpicture}$. An oriented Brauer       diagram of type $\ob a \rightarrow \ob b$  is a diagram obtained by tensor product and composition of the defining generators of $\OB$ along with
  \begin{equation}\label{genofobddb}
  \begin{tikzpicture}[baseline = 3pt, scale=0.5, color=\clr]
        \draw[->,thick] (0,0) to[out=up, in=down] (1,1);
        \draw[<-,thick] (1,0) to[out=up, in=down] (0,1);
\end{tikzpicture} ,\  \
\ \begin{tikzpicture}[baseline = 3pt, scale=0.5, color=\clr]
        \draw[<-,thick] (0,0) to[out=up, in=down] (1,1);
        \draw[->,thick] (1,0) to[out=up, in=down] (0,1);
\end{tikzpicture},\ \
         \begin{tikzpicture}[baseline = 5pt, scale=0.5, color=\clr]
        \draw[->,thick] (0,1) to[out=down,in=left] (0.5,0.35) to[out=right,in=down] (1,1);
    \end{tikzpicture}
    :=~
    \begin{tikzpicture}[baseline = 5pt, scale=0.5, color=\clr]
        \draw[->,thick] (0,1) to[out=down,in=up] (1,0) to[out=down,in=right] (0.5,-0.5) to[out=left,in=down] (0,0) to[out=up,in=down] (1,1);
    \end{tikzpicture}
    ~,
    \begin{tikzpicture}[baseline = 5pt, scale=0.5, color=\clr]
        \draw[->,thick] (0,0) to[out=up,in=left] (0.5,0.65) to[out=right,in=up] (1,0);
    \end{tikzpicture}
    :=~
    \begin{tikzpicture}[baseline = 5pt, scale=0.5, color=\clr]
        \draw[->,thick] (0,0) to[out=up,in=down] (1,1) to[out=up,in=right] (0.5,1.5) to[out=left,in=up] (0,1) to[out=down,in=up] (1,0);
    \end{tikzpicture}, \text{ and } \ \
        \begin{tikzpicture}[baseline = 3pt, scale=0.5, color=\clr]
        \draw[<-,thick] (0,0) to[out=up, in=down] (1,1);
        \draw[<-,thick] (1,0) to[out=up, in=down] (0,1);
\end{tikzpicture}:=\begin{tikzpicture}[baseline = 10pt, scale=0.5, color=\clr]
            \draw[-,thick] (2,2) to[out=down, in=up] (2,1) to[out=down, in=right] (1.5,0.5) to[out=left,in=down] (1,1);
            \draw[->,thick] (1,1) to[out=up,in=right] (0.5,1.5) to[out=left,in=up] (0,1) to[out=down,in=up] (0,0);
            \draw[<-,thick] (0.7,0) to[out=up,in=down] (1.3,2);
        \end{tikzpicture}
        \end{equation}
       such that the resulting diagram  can be interpreted as a morphism in $\Hom_{\OB}(\ob a,\ob b)$.


Let $\ell(\ob a)$ be the total  numbers of $ \uparrow$ and $\downarrow$ in $\ob a\in \langle\uparrow,\downarrow\rangle$.
  We add an arrow to each strand of a $d\in \mathbb B_{m, s}$  such that
  the orientations at the endpoints of $d$ at the bottom (resp., top) row give  $\ob a$ (resp., $\ob b$).
  Then we get an oriented Brauer diagram $d'$ with bubbles of type $\ob a \rightarrow \ob b$ such that  $(\ell(\ob a),\ell(\ob b))=(m,s)$. All oriented Brauer diagrams  with bubbles of type $\ob a \rightarrow \ob b$
  can be obtained in this way.
  For example,
  $$d=\begin{tikzpicture}[baseline = 25pt, scale=0.35, color=\clr]
        \draw[-,thick] (5,0) to[out=up,in=down] (8,4);
         \draw[-,thick] (10,0) to[out=up,in=down] (6,4);
          \draw[-,thick] (2.6,4) to (2.5,4) to[out=left,in=up] (1.5,3)
                        to[out=down,in=left] (2.5,2)
                        to[out=right,in=down] (3.5,3)
                        to[out=up,in=right] (2.5,4);
         \draw[-,thick] (2,0) to[out=up,in=left] (4,1.5) to[out=right,in=up] (6,0);
           \end{tikzpicture}, \ \ \text{ and  }\ \ d'= \begin{tikzpicture}[baseline = 25pt, scale=0.35, color=\clr]
        \draw[->,thick] (5,0) to[out=up,in=down] (8,4);
         \draw[->,thick] (10,0) to[out=up,in=down] (6,4);
          \draw[->,thick] (2.6,4) to (2.5,4) to[out=left,in=up] (1.5,3)
                        to[out=down,in=left] (2.5,2)
                        to[out=right,in=down] (3.5,3)
                        to[out=up,in=right] (2.5,4);
         \draw[->,thick] (2,0) to[out=up,in=left] (4,1.5) to[out=right,in=up] (6,0);
           \end{tikzpicture}.$$
 In this case, $d\in \mathbb {B}_{4, 2}$ and $d'$ is an oriented Brauer diagram of type $\uparrow\uparrow\downarrow\uparrow\rightarrow\uparrow\uparrow$.

 Two oriented Brauer diagrams $d_1,d_2$ of type $\ob a\rightarrow\ob b$ are equivalent if $d_1'\sim  d_2'$, where $d_1', d_2'$ are obtained from $d_1$ and $d_2$ by deleting their  arrows.
  It is proved in \cite{BCNR} that
 two equivalent oriented Brauer diagrams represent the same morphism and $\Hom_{\OB}(\ob a,\ob b)$ has basis given by  all  equivalence   classes of oriented Brauer diagrams (with bubbles) of type $\ob a \rightarrow \ob b$.

 \begin{Lemma}\label{twocate}There is a  monoidal  functor $ \varphi: \OB\rightarrow\B$ such that   $\varphi(\uparrow)=\phi(\downarrow)= \ob 1$,
  $\varphi(\lcupo)=\lcup$, $\varphi(\lcapo)=\lcap$ and $\varphi(\swapo)=\swap$.
 \end{Lemma}
\begin{proof}
It is enough to show that the relations in \eqref{aOB relations 1 (symmetric group)}--\eqref{aOB relations 3 (zigzags and invertibility)} are satisfied under $\varphi$.
Thanks to \eqref{OB relations 1 (symmetric group)}--\eqref{OB relations 2 (zigzags and invertibility)}, \eqref{aOB relations 1 (symmetric group)}--\eqref{aOB relations 2 (zigzags and invertibility)} are satisfied.
By  \eqref{OB relations 2 (zigzags and invertibility)} and \eqref{Brauer relation 4}, $\varphi(\begin{tikzpicture}[baseline = 10pt, scale=0.5, color=\clr]
            \draw[-,thick] (2,2) to[out=down, in=up] (2,1) to[out=down, in=right] (1.5,0.5) to[out=left,in=down] (1,1);
            \draw[->,thick] (1,1) to[out=up,in=right] (0.5,1.5) to[out=left,in=up] (0,1) to[out=down,in=up] (0,0);
            \draw[->,thick] (0.7,0) to[out=up,in=down] (1.3,2);
        \end{tikzpicture})=\swap$. Since   $\swap$ is  invertible
        (see \eqref{OB relations 1 (symmetric group)}),  \eqref{aOB relations 3 (zigzags and invertibility)} is satisfied.
\end{proof}

\begin{rem}\label{obbdiaramd}
By  Lemma~\ref{twocate} and \eqref{OB relations 2 (zigzags and invertibility)}--\eqref{Brauer relation 4},
 $\varphi$ sends \begin{tikzpicture}[baseline = 3pt, scale=0.5, color=\clr]
        \draw[->,thick] (0,0) to[out=up, in=down] (1,1);
        \draw[<-,thick] (1,0) to[out=up, in=down] (0,1);
\end{tikzpicture}, $\begin{tikzpicture}[baseline = 3pt, scale=0.5, color=\clr]
        \draw[<-,thick] (0,0) to[out=up, in=down] (1,1);
        \draw[->,thick] (1,0) to[out=up, in=down] (0,1);
\end{tikzpicture}$,
        \begin{tikzpicture}[baseline = 3pt, scale=0.5, color=\clr]
       \draw[<-,thick] (0,0) to[out=up, in=down] (1,1);
       \draw[<-,thick] (1,0) to[out=up, in=down] (0,1);
\end{tikzpicture}
         (resp.,
    \begin{tikzpicture}[baseline = 5pt, scale=0.5, color=\clr]
        \draw[->,thick] (0,0) to[out=up,in=left] (0.5,0.65) to[out=right,in=up] (1,0);
    \end{tikzpicture}
    ,\begin{tikzpicture}[baseline = 5pt, scale=0.5, color=\clr]
        \draw[->,thick] (0,1) to[out=down,in=left] (0.5,0.35) to[out=right,in=down] (1,1);
    \end{tikzpicture}) to $\swap $ (resp.,\lcap, \lcup). If  $d$ is an oriented Brauer diagram of type $\ob a\rightarrow\ob b$, then
  $\varphi(d)$ is  obtained  from $d$ by removing all of its arrows.
 \end{rem}



\begin{Prop}\label{equi123}
 If $d,d'\in \mathbb{B}_{m,s}$ and $ d\sim  d'$, then $d=d'$   in $\B$ and $\AB$.
 \end{Prop}
\begin{proof}
The endpoints of a $d\in\mathbb{B}_{m,s}$ are indexed by $\{1,2,\ldots,m+s\}$.
If  $d,d'\in \mathbb{B}_{m,s}$ and $d\sim d'$, then
 $i$   and   $j$ are connected in $d$ if and only if $i$   and   $j$ are connected in $d'$ for all admissible $i$ and $j$.
In this case, let  $(i,j)_d$ be the strand of $d$ which connects  $i$  and   $j$.
 We add arrows on   all  $(i,j)_d$ and all bubbles of $d$ so as to get   an oriented Brauer diagram $\tilde  d$.
Similarly, we have  an oriented Brauer diagram $\tilde d'$ such that the arrow on each strand  $(i,j)_{d'}$ is the same as that at  $(i,j)_{d}$.
Then  $\tilde  d$ and $\tilde d'$ are oriented Brauer diagrams
 of the same type and $\tilde  d\sim \tilde d'$.  Thus,
$\tilde  d=\tilde d'$ in $\OB$.
By Remark~\ref{obbdiaramd},
$\varphi(\tilde d)=d $ and $\varphi (\tilde d')= d'$, where  $\varphi$ is the  monoidal functor in Lemma~\ref{twocate}. So,
 $d=d'$ in $\B$. Using the functor $\mathcal F$ in \eqref{functorfrombtoab} yields that  $d=d'$   in $\AB$.
\end{proof}

\begin{Cor}\label{00generat}The $\kappa$-module  $\Hom_{\B}(\ob m,\ob s)$  is spanned by $\{\Delta_0^j d\mid d\in \bar{\mathbb B}_{m,s}/\sim, j\in\mathbb N\}$, where $\bar{\mathbb B}_{m,s}$ is given  in \eqref{defofbarbms}. 
\end{Cor}
\begin{proof}
By Proposition~\ref{equi123}, $\Hom_{\B}(\ob m,\ob s)$ is spanned by $\mathbb B_{m,s}/\sim$. Since $\{\Delta_0^j d\mid d\in \bar{\mathbb B}_{m,s}/\sim, j\in\mathbb N\}$ is a complete set of representatives of  ${\mathbb B}_{m,s}/\sim$, the result follows.\end{proof}

\begin{Lemma}\label{eviso} Suppose $\mathcal C\in \{\B,\AB\}$.  For any  positive integer $m$,  define
\begin{equation}\label{usuflelem}
 \eta_{\ob m}=\begin{tikzpicture}[baseline = 25pt, scale=0.35, color=\clr]
        \draw[-,thick] (0,5) to[out=down,in=left] (2.5,2) to[out=right,in=down] (5,5);
        \draw[-,thick] (0.5,5) to[out=down,in=left] (2.5,2.5) to[out=right,in=down] (4.5,5);
        \draw(1.5,4.5)node{$\cdots$};\draw(3.5,4.5)node{$\cdots$};
        \draw[-,thick] (2,5) to[out=down,in=left] (2.5,4) to[out=right,in=down] (3,5);
        \draw (1,5.5)node {$\ob m$}; \draw (4,5.5)node {$\ob m$};
           \end{tikzpicture} \quad \text{ and }\quad
         \varepsilon_{\ob m}= \begin{tikzpicture}[baseline = 25pt, scale=0.35, color=\clr]
           \draw[-,thick] (0,2) to[out=up,in=left] (2.5,6) to[out=right,in=up] (5,2);
           \draw[-,thick] (0.2,2) to[out=up,in=left] (2.3,5.5) to[out=right,in=up] (4.8,2);
           \draw[-,thick] (1.8,2) to[out=up,in=left] (2.5,3.5) to[out=right,in=up] (3.2,2);
           \draw(1,2.5)node{$\cdots$};\draw(4,2.5)node{$\cdots$};
          \draw (1,1.5)node {$\ob m$}; \draw (4,1.5)node {$\ob m$};
           \end{tikzpicture}.
\end{equation}
 Then   $ \eta_{\ob m}\in Hom_{\mathcal C}(\ob 0,\ob {2m})$ and  $\varepsilon_{\ob m}\in Hom_{\mathcal C}(\ob {2m},\ob 0)$   such that
$$(1_{\ob m}\otimes \varepsilon_{\ob m})\circ (\eta_{\ob m}\otimes 1_{\ob m})=(\varepsilon_{\ob m}\otimes 1_{\ob m})\circ ( 1_{\ob m}\otimes\eta_{\ob m})=1_{\ob m}.$$

\end{Lemma}
\begin{proof} Thanks to Proposition~\ref{equi123}, we have $(1_{\ob m}\otimes \varepsilon_{\ob m})\circ (\eta_{\ob m}\otimes 1_{\ob m})=(\varepsilon_{\ob m}\otimes 1_{\ob m})\circ ( 1_{\ob m}\otimes\eta_{\ob m})=1_{\ob m}$  in  $\B$.  Using the functor $\mathcal F$ in \eqref{functorfrombtoab}  yields
the result for  $\AB$. \end{proof}

 \begin{Cor} The Brauer category  $\B$ is a rigid  symmetric monoidal category.\end{Cor}
\begin{proof} For all objects $\ob m, \ob s\in \B$,  define
$$\sigma_{\ob m,\ob s}= \begin{tikzpicture}[baseline = 25pt, scale=0.35, color=\clr]
        \draw[-,thick] (0,0) to[out=up,in=down] (5,4);
        \draw[-,   thick ] (1,0) to[out=up,in=down] (6,4);
        \draw(2.1,0.5)node{$\cdots$};\draw(7.1,3.5)node{$\cdots$};
        \draw(2.1,3.5)node{$\cdots$};\draw(7.1,0.5)node{$\cdots$};
        \draw[-,  thick] (3,0) to[out=up,in=down] (8,4);
        \draw[-,  thick] (4,0) to[out=up,in=down] (9,4);
        \draw (2,-0.5)node {$\ob m$};\draw (7,-0.5)node {$\ob s$};
        \draw[-,thick] (0,4) to[out= down,in=up] (5,0);
        \draw[-,   thick ] (1,4) to[out= down,in=up] (6,0);
        \draw[-,  thick] (3,4) to[out= down,in=up] (8,0);
        \draw[-,  thick] (4,4) to[out=down,in=up ] (9,0);
        \draw (2,4.5)node {$\ob s$}; \draw (7,4.5)node {$\ob m$};
           \end{tikzpicture}. $$
      Let  $\tau_{\B}:\B\times \B\rightarrow\B\times \B$ be the functor  sending $(\ob a, \ob b)$ to $(\ob b, \ob a)$.      By Proposition~\ref{equi123}, it is routine to check that $\sigma:  -\otimes -\rightarrow  (-\otimes -)\circ \tau_{\B}$ is a natural isomorphism. Moreover, it  satisfies the conditions in  \eqref{symm} and   $ \sigma_{\ob m,\ob s}^{-1}=\sigma_{\ob s,\ob m}$.
              Therefore   $\B$ is a symmetric monoidal category.
    By  Lemma~\ref{eviso} and \eqref{rigdual}, the left (resp., right)  dual of $\ob m$ is itself,  and the unit morphism and the counit morphism are $\eta_{\ob m}$ and $\epsilon_{\ob m}$, respectively.
So, $\B$ is a rigid symmetric monoidal category.\end{proof}



  \begin{Defn}\label{etae} Suppose $m,s,r\in \mathbb N$ such that  $m>0$ and $m+s=2r$. For any  $\mathcal C\in \{\B,\AB\}$, define two $\kappa$-linear maps  $\bar\eta_{\ob m}: \Hom_{\mathcal C}(\ob m,\ob s) \rightarrow \Hom_{\mathcal C}(\ob 0,\ob 2r)$ and
$\bar\varepsilon _{\ob m}: \Hom_{\mathcal C}(\ob 0,\ob {2r})\rightarrow\Hom_{\mathcal C}(\ob m,\ob s)$,
such that $\bar\eta_{\ob m}=(-\otimes 1_{\ob m}) \circ \eta_{\ob m}$ and $\bar\varepsilon _{\ob m} = (1_{\ob s}\otimes \varepsilon _{\ob m}) \circ(-\otimes 1_{\ob m})$.\end{Defn}
Thanks to \eqref{tensotideal}, $\bar\eta_{\ob m}(I(\ob m,\ob s))\subset I(\ob 0,\ob {2r})$ and $\bar\varepsilon _{\ob m}(I(\ob 0,\ob {2r}))\subset I(\ob m,\ob s)$ for any right tensor ideal $I$ of $\AB$. If  $\mathcal D\in \{\CB^f, \CB^f(\omega)\}$, then \begin{itemize} \item  $ \bar\eta_{\ob m}$ induces a linear map from $\Hom_{\mathcal D}(\ob m,\ob s)$ to $ \Hom_{\mathcal D}(\ob 0,\ob {2r})$,
 \item  $\bar\varepsilon _{\ob m}$ induces another linear map  from $\Hom_{\mathcal D}(\ob 0,\ob {2r})$ to $ \Hom_{\mathcal D}(\ob m,\ob s)$.\end{itemize}
  To simplify the notation,  two induced linear maps above  are also denoted by $\bar\eta_{\ob m} $ and  $\bar\varepsilon _{\ob m}$.

For any  $d\in\mathbb D_{m,s}$, we say $d$ is of degree $k$ and write $\text{deg}(d)=k$ if the total number of $\bullet$'s on all strands of $d$  is $k$.  For any  $k\in \mathbb N$, let
$$\Hom_{\AB}(\ob m,\ob s)_{\leq k}=\kappa\text{-span} \{d\in \mathbb D_{m,s} \mid \text{deg}(d)\leq k\}.$$

\begin{Lemma}\label{key1}

Suppose that   $m,s, r\in \mathbb N$ such that $m+s=2r$ and $m>0$.  Then $ \bar\eta_{\ob m}$  and $ \bar\varepsilon _{\ob m}$ are mutually inverse to each other.
 For any $\mathcal C\in \{\B,\AB, \CB^f,\CB^f(\omega)\}$,  we have the following $\kappa$-isomorphisms:
\begin{enumerate}
\item[(1)]$\Hom_{\mathcal C}(\ob m,\ob s)\cong \Hom_{\mathcal C}(\ob 0,\ob {2r})$,
\item[(2)]  $\Hom_{\AB}(\ob m,\ob s)_{\leq k}\cong \Hom_{\AB}(\ob 0,\ob {2r})_{\leq k}$  for all $k\in \mathbb N$.
\end{enumerate}
\end{Lemma}
\begin{proof} If  $\mathcal C\in \{\B,\AB\}$ and $b\in \Hom_{\mathcal C}(\ob m,\ob s)$, then
$$\begin{aligned}
\bar\varepsilon _{\ob m}\circ\bar\eta_{\ob m} (b) & = (1_{\ob s}\otimes\varepsilon _{\ob m} )\circ (b\otimes 1_{\ob {2m}})\circ (\eta_{\ob m}\otimes 1_{\ob m}), \text{(by Definition~\ref{etae})}\\
&= (b\otimes \varepsilon _{\ob m})\circ (\eta_{\ob m}\otimes 1_{\ob m})= b\circ (1_{\ob m} \otimes\varepsilon _{\ob m} )\circ (\eta_{\ob m}\otimes 1_{\ob m})
=b, \text{ (by  Lemma~\ref{eviso}).}
\end{aligned}$$
Similarly, we have $\bar\eta_{\ob m}\circ \bar\varepsilon _{\ob m}(b')=b'$ for any $b'\in\Hom_{\mathcal C}(\ob 0,\ob {2r})$.
Hence  $ \bar\eta_{\ob m}$  and $ \bar\varepsilon _{\ob m}$ are mutually inverse to each other. This proves (1).
 Thanks to \eqref{usuflelem}, there is no $\bullet$ on any strand of both $ \eta_{\ob m}$ and $\varepsilon_{\ob m}$. So,
  $ \bar\eta_{\ob m}$ sends  $\Hom_{\AB}(\ob m,\ob s)_{\leq k}$ into $\Hom_{\AB}(\ob 0,\ob {2r})_{\leq k}$,  and  $ \bar\varepsilon _{\ob m}$ sends $\Hom_{\AB}(\ob 0,\ob {2r})_{\leq k}$ into $\Hom_{\AB}(\ob m,\ob s)_{\leq k}$.
  Since  $ \bar\eta_{\ob m}$  and $ \bar\varepsilon _{\ob m}$ are mutually inverse to each other, (2) follows. \end{proof}

\begin{Lemma}\label{eaquvivalents}
Suppose $m,s,r\in\mathbb N$ such that $m>0$ and $m+s=2r$.
\begin{itemize}
\item [(1)]For any $d_1, d_2\in \mathbb D_{m,s}$, $d_1\sim d_2$ if and only if
$\bar\eta_{\ob m}(d_1)\sim\bar\eta_{\ob m}(d_2)$.
\item[(2)] For any $d_3, d_4\in \mathbb D_{0,2r}$, $d_3\sim d_4$ if and only if
$\bar\varepsilon_{\ob m}(d_3)\sim\bar\varepsilon_{\ob m}(d_4)$.
\end{itemize}
\end{Lemma}
\begin{proof}By Definition~\ref{etae}, $\bar\eta_{\ob m}(d_1)\sim\bar\eta_{\ob m}(d_2)$ (resp.,  $\bar\varepsilon_{\ob m}(d_3)\sim\bar\varepsilon_{\ob m}(d_4)$) if  $d_1\sim d_2$ (resp., $d_3\sim d_4$).
Since $ \bar\eta_{\ob m}$  and $ \bar\varepsilon _{\ob m}$ are mutually inverse to each other, the results follow.
\end{proof}

\begin{Cor}\label{bijectivemaps}Suppose $m,s,r\in\mathbb N$ such that $m>0$ and $m+s=2r$. If $\mathbb M\in \{\mathbb D, \mathbb B, \bar {\mathbb B}\}$, then
\begin{itemize}
\item[(1)] $\bar\eta_{\ob m}$ induces a bijection between $\mathbb M_{m,s}/\sim$ and $\mathbb M_{0,2r}/\sim$,
 \item[(2)] $\bar\varepsilon_{\ob m}$ induces a bijection between   $\mathbb M _{0,2r}/\sim$ and $\mathbb M_{m,s}/\sim$
.
\end{itemize}
\end{Cor}
\begin{proof}
By Lemma~\ref{eaquvivalents}, $\bar\eta_{\ob m}$ and $\bar\varepsilon_{\ob m}$ induce the required injective maps.
Thanks to  Lemma~\ref{key1},  these maps  are also surjective.
\end{proof}

For any positive integer $r$,
let  $\mathfrak S_{r}$ be the symmetric group on $r$ letters. Then $\mathfrak S_r$ is  generated by    $s_i$,  $1\leq i\leq r-1$, where $s_i$  swaps $i$ and $i+1$ and fixes the others.

\begin{Defn}\label{definofs} For  $1\leq i\leq r-1$, define
 $S_{i}:=\begin{tikzpicture}[baseline = 10pt, scale=0.5, color=\clr]
                \draw[-,thick] (0,0.5)to[out=up,in=down](0,1.5);
    \end{tikzpicture} ^{\otimes i-1}\otimes S\otimes \begin{tikzpicture}[baseline = 10pt, scale=0.5, color=\clr]
                \draw[-,thick] (0,0.5)to[out=up,in=down](0,1.5);
    \end{tikzpicture} ^{\otimes r-i-1}$, and $S_w:=S_{l_1}\circ S_{l_2}\circ \ldots \circ S_{l_t}$,
    where $w=s_{l_1}s_{l_2}\ldots s_{l_t} $ is a reduced expression  of $w$.
\end{Defn}
By Proposition~\ref{equi123},  $S_w$ is independent of a reduced expression of $w$.
\begin{Lemma}\label{numberss}
Suppose $m,s\in\mathbb N$.
If $m+s$ is odd, then $|\bar {\mathbb B}_{m,s} /\sim|=0$. If $m+s=2r$, then $|\bar{\mathbb B}_{m,s}/\sim|=(2r-1)!!$, where $\bar{\mathbb B}_{m,s}$ is given  in \eqref{defofbarbms}. 
\end{Lemma}

\begin{proof}
If $m+s$ is odd, then $\mathbb B_{m, s}=\emptyset$  and the result follows.
Suppose that $m+s=2r$. Thanks to  Corollary~\ref{bijectivemaps}, it is enough to prove the case that $m=0$ and $s=2r$.
Let $[2r]$ be the set of all partitions of the set $ \{1,2,\ldots,2r\}$ into disjoint union of pairs.
Then the cardinality of $[2r]$ is $(2r-1)!!$. Since any equivalence class of $\bar {\mathbb B}_{0,2r}$ gives an element in  $[2r]$, it suffices   to show that,  for any  $x\in [2r]$, there is a $d\in \bar {\mathbb B}_{0,2r}$ such that $d$  gives  $x$.

Suppose that $\{\{i_1,j_1\},\ldots, \{i_r,j_r\}\}$ is a partition  of $  \{1,2,\ldots,2r\}$ into disjoint union of pairs.
There is a $w\in\mathfrak S_{2r}$ such that $w(k)=i_k$ and $w(2k)=j_k$, $1\le k\le r$.
    Define
    \begin{equation}\label{dwa}
    d_w:= S_w\circ U^{\otimes r}\in \bar{\mathbb B}_{0,2r}.
    \end{equation}
    Then $d_w$  gives the partition   $\{\{i_1,j_1\},\ldots, \{i_r,j_r\}\}$.
        So, $|\bar{\mathbb B}_{0,2r}/\sim|=(2r-1)!!$.
\end{proof}

  \begin{Defn}\label{defofij} Suppose  $d\in \mathbb B_{0,2r}/\sim$.  We count the   endpoints at the top row of $d$ as $1, 2, \ldots, 2r$  from left to right.
 There are $r$ cups, say  $( i(d)_1, j(d)_{ 1}),\ldots,( i(d)_r,  j(d)_{r})$  of $d$ such that
  \begin{itemize}
  \item $i(d)_k$ (resp., $j(d)_k$) represents  the left (resp., right) endpoint of  $( i(d)_k, j(d)_{k})$, $1\leq k\leq r$,
 \item  $ j(d)_1< j(d)_2<\ldots< j(d)_r=2r$.
 \end{itemize}
\end{Defn}
 For example, $\{(m-i,m+i+1)\mid  0\le i\le m-1\}$ are all cups in  $\eta_{\ob m}$, where  $\eta_{\ob m}$ is given  in \eqref{usuflelem}.
In this case, $i(\eta_{\ob m})=(m,m-1,\ldots,1)$, $j(\eta_{\ob m})=(m+1,m+2,\ldots,2m)$.
 \begin{proof}[\textbf{Proof of Theorem}~\ref{basisofb}]
 First, we assume   $\kappa=\mathbb C$.  Let   $N\in \mathbb Z$ such that  $N\ge 2r$ (i.e. $n\geq r$).
For any   $d\in\bar{\mathbb B}_{0,2r}/\sim$, define  $\theta(d)\in \underline N^{2r}$ such that
 \begin{equation}\label{defofwdbed}
 \theta(d)_{i(d)_{ l}}=  l, ~~\theta(d)_{ j(d)_{ l}}=-l, \quad  1\leq l\leq r,
\end{equation}
where  $( i(d)_1, j(d)_{ 1}),\ldots,( i(d)_r,  j(d)_{r})$ are cups   of $d$.
This is well-defined since we are assuming $n\ge r$. Moreover, for all $1\leq h\leq 2r$, $1\leq l\leq r$,
\begin{equation}\label{notmiunes}
\theta(d)_{i(d)_{ l}}\neq -\theta(d)_h, \text{ if  $h\neq j(d)_l$}.
\end{equation}
Let  $d$ be $d_w$ in \eqref{dwa}, where $w(k)=i(d)_k,w(2k)=j(d)_k$, $1\leq k\leq r$.
 Thanks to Proposition~\ref{level1s} and \eqref{vistar}, up to a sign, we have  \begin{equation}\label{defofzd1} \Phi(d)(1)=  z(d):=  z(d)_1\otimes z(d)_2\otimes \ldots\otimes z(d)_{2r}\in V^{\otimes 2r},\end{equation}
where
\begin{equation}\label{defofzd}
  z(d)_{i(d)_{k}}\otimes z(d)_{  j(d)_{  k}}=\sum_{1\leq j\leq n} (\varepsilon_\mfg v_j\otimes v_{-j}+ v_{-j}\otimes v_j)+\delta_{\mfg, \mathfrak{so}_{2n+1}} v_0\otimes v_0, \text{ for }  1\leq k\leq r.
\end{equation}
 Note that $V^{\otimes 2r}$ has basis $\{v_{\bf i}\mid \mathbf i\in \underline N^{2r}\}$,  where $\mathbf i=(i_1,i_2,\ldots,i_{2r})$  and $v_{\mathbf i}=v_{i_1}\otimes v_{i_2}\otimes \ldots \otimes v_{i_{2r}}$. For each  $d'\in\bar{\mathbb B}_{0,2r}/\sim $,
 we can write  $\Phi( d')(1)$ as the linear combination of such basis  elements.
 Let $c_{d,d'}$ be  the coefficient of $v_{\theta(d)}$ in this expression.
 Thanks to \eqref{defofzd},  $c_{d, d}=\pm 1$. If $d\neq d'$, then there  is a cup  $(i(d')_k, j(d')_k)$ of $d'$ such that it is not a cup of $d$.
 By \eqref{defofzd}, we have  $i_{i(d')_k}=-i_{j(d')_k}$ if $v_{\mathbf i}$ is  a term in the expression of   $\Phi( d')(1)$. Thanks to \eqref{notmiunes}, $ c_{d,d'}=0$. Thus
 \begin{equation}\label{coeffi}
  c_{d,d'}=\pm\delta_{d,d'}.
 \end{equation}

  Now, we
assume $
\sum_{d\in \bar{\mathbb B}_{0,2r}/\sim } p_d(\Delta_0)d=0$
in $\Hom_{\B}(\ob 0,\ob {2r})$, where
  $p_d(t)\in \mathbb C[t]$.
We say that $p_d(t)=0$ for all $d\in \bar{\mathbb B}_{0,2r}/\sim$. Otherwise, there is at least a $d_0\in  \bar{\mathbb B}_{0,2r}/\sim$ such that $p_{d_0}(t)\neq 0$.  By the fundamental theorem of  algebra,  we pick up an $N$ such that $N\gg 0$ and   $ p_{d_0}(\varepsilon_{\mathfrak g}N)\neq 0$.
By \eqref{coeffi} and \eqref{ewjdngn}, the coefficient of $v_{\theta(d_0)}$ in $\sum_{d\in \bar{\mathbb B}_{0,2r}/\sim} \Phi(p_{d}(\Delta_0)d)(1) $ is the non-zero scalar  $\pm p_{d_0}( \varepsilon_{\mathfrak g}N)$. This is a   contradiction since  $
\sum_{d\in \bar{\mathbb B}_{0,2r}/\sim} p_d(\Delta_0)d=0$.
 By Corollary~\ref{00generat},  we immediately have Theorem~\ref{basisofb} when  $(m, s)=(0, 2r)$.

In general, by Corollary~\ref{bijectivemaps},
   $\bar\eta_{\ob m}$ gives a bijective   map from  $\mathbb B_{m,s}/\sim$ to $ \mathbb B_{0,2r}/\sim$.
By Corollary~\ref{00generat} and the result for $(m, s)=(0, 2r)$, we immediately have Theorem~A whenever  $m+s$ is even.
If  $m+s$ is odd, then $ \Hom_{\B}(\ob m,\ob s)=0$.
So, Theorem~\ref{basisofb} is proven whenever $\kappa=\mathbb C$.

If a set is $\mathbb C$-linear independent,  it is $\mathbb Z$-independent.  Using Corollary~\ref{00generat} again, we have Theorem~\ref{basisofb}  over $\mathbb Z$.
 Denote by $\B_\kappa$  the Brauer category over a commutative ring $\kappa$  containing $1$. Let $\mathcal C:=\kappa\otimes_{\mathbb Z}\B_{\mathbb Z}$. By base change property, $\Hom_{\mathcal C}(\ob m,\ob s)$
 has basis given by $\mathbb B_{m,s}/\sim$. Theorem~\ref{basisofb} follows from  this observation and the fact that
 there is an obvious  functor from $\B_\kappa$ to $\mathcal C$, which sends  the required  basis element of $\Hom_{\B_\kappa}(\ob m,\ob s)$ to the corresponding basis element in  $\Hom_{\mathcal C}(\ob m,\ob s)$.
\end{proof}
\section{Basis theorem for the  affine Brauer category }\label{proofb}
In this section we prove Theorem~\ref{Cyclotomic basis conjecture} by constructing a functor from $\AB$ to $ \END(\U(\mathfrak g)\text{-mod})$, where  $\U(\mathfrak g)\text{-mod}$ is the category of all left  $\U(\mathfrak g)$-modules and $ \END(\U(\mathfrak g)\text{-mod})$ is  the category of   endofunctors of $\U(\mathfrak g)\text{-mod}$, $\mathfrak g\in\{\mathfrak {so}_N, \mathfrak {sp}_N\}$.
First of all, we need some results  to explain  that  $ \Hom_{\AB}(\ob m,\ob s)$ is spanned by $\mathbb {ND}_{m,s}$.

\begin{Lemma}\label{capre}In  $\AB$, we have:
\begin{enumerate}
\item[(1)]  \begin{tikzpicture}[baseline = 5pt, scale=0.5, color=\clr]
        \draw[-,thick] (0,2) to[out=down,in=left] (1,0.35) to[out=right,in=down] (2,2);
       \draw( 0,1.5) \bdot;
    \end{tikzpicture}
    ~=~$-$
    \begin{tikzpicture}[baseline = 5pt, scale=0.5, color=\clr]
        \draw[-,thick] (0,2) to[out=down,in=left] (1,0.35) to[out=right,in=down] (2,2); \draw( 2,1.5)\bdot;
    \end{tikzpicture},
\item[(2)] \begin{tikzpicture}[baseline = 5pt, scale=0.5, color=\clr]
        \draw[-,thick] (0,0) to[out=up,in=left] (1,2) to[out=right,in=up] (2,0);\draw( 0,0.5)\bdot;
    \end{tikzpicture}
    ~=~$-$~\begin{tikzpicture}[baseline = 5pt, scale=0.5, color=\clr]
        \draw[-,thick] (0,0) to[out=up,in=left] (1,2) to[out=right,in=up] (2,0);\draw( 2,0.5) \bdot;
    \end{tikzpicture}.
\end{enumerate}
\end{Lemma}
\begin{proof} Thanks to \eqref{OB relations 2 (zigzags and invertibility2)} and  \eqref{OB relations 2 (zigzags and invertibility)}, we have
  \begin{equation*}
        \begin{tikzpicture}[baseline = 7.5pt, scale=0.5, color=\clr]
            \draw[-,thick] (0,2) to (0,0.5)
                            to[out=down, in=left] (0.5,0)
                            to[out=right, in=down] (1,0.5) to (1,2);
            \draw (1,0.75) \bdot;
        \end{tikzpicture}
        ~{=}~-
        \begin{tikzpicture}[baseline = 7.5pt, scale=0.5, color=\clr]
            \draw[-,thick] (0,2) to (0,0.5)
                            to[out=down, in=left] (0.25,0)
                            to[out=right, in=down] (0.5,0.5) to (0.5,1)
                            to[out=up, in=left] (0.75,1.5)
                            to[out=right, in=up] (1,1)
                            to[out=down, in=left] (1.25,0.5)
                            to[out=right, in=down] (1.5,1) to (1.5,2);
            \draw (1,1)\bdot;
        \end{tikzpicture}
        ~{=}~-
        \begin{tikzpicture}[baseline = 7.5pt, scale=0.5, color=\clr]
            \draw[-,thick] (0,2) to (0,0.5)
                            to[out=down, in=left] (0.5,0)
                            to[out=right, in=down] (1,0.5) to (1,2);
            \draw (0,0.75) \bdot;
        \end{tikzpicture}.
        ~
    \end{equation*}
 This verifies (1). One can verify (2) similarly.
\end{proof}

\begin{Lemma}\label{cp2}
 The relation \eqref{OB relations 2 (zigzags and invertibility2)} follows from Lemma~\ref{capre} and \eqref{OB relations 2 (zigzags and invertibility)}.
\end{Lemma}
\begin{proof}  Thanks to Lemma~\ref{capre}(1) and \eqref{OB relations 2 (zigzags and invertibility)}, we have
\begin{equation*}
 \begin{tikzpicture}[baseline = 10pt, scale=0.5, color=\clr]
            \draw[-,thick] (2,0) to[out=up, in=down] (2,1) to[out=up, in=right] (1.5,1.5) to[out=left,in=up] (1,1);
            \draw[-,thick] (1,1) to[out=down,in=right] (0.5,0.5) to[out=left,in=down] (0,1) to[out=up,in=down] (0,2);
           \draw( 1,1) \bdot;
        \end{tikzpicture}
        {=}~-~
        \begin{tikzpicture}[baseline = 10pt, scale=0.5, color=\clr]
            \draw[-,thick] (2,0) to[out=up, in=down] (2,1) to[out=up, in=right] (1.5,1.5) to[out=left,in=up] (1,1);
            \draw[-,thick] (1,1) to[out=down,in=right] (0.5,0.5) to[out=left,in=down] (0,1) to[out=up,in=down] (0,2);
           \draw( 2,1) \bdot;
        \end{tikzpicture}
       ~{=}~
        -~\begin{tikzpicture}[baseline = 10pt, scale=0.5, color=\clr]
            \draw[-,thick] (0,0) to (0,2);
            \draw(0,1) \bdot;
                    \end{tikzpicture}.
        \end{equation*}
This proves the first equality in \eqref{OB relations 2 (zigzags and invertibility2)}. One can check  the second one similarly.
\end{proof}
By Lemma~\ref{cp2}, the category $\AB$ can also be defined by  replacing  \eqref{OB relations 2 (zigzags and invertibility2)} with
 Lemma~\ref{capre}(1)-(2).

\begin{Lemma}\label{slides relations vertical} In $\AB$, we have
 \begin{itemize}
\item[(a)]\begin{tikzpicture}[baseline = 7.5pt, scale=0.5, color=\clr]
            \draw[-,thick] (0,0) to[out=up, in=down] (1,2);\draw[-,thick] (0,0) to[out=up, in=down] (0,-0.2);
             \draw[-,thick] (1,0) to[out=up, in=down] (0,2);\draw[-,thick] (1,0) to[out=up, in=down] (1,-0.2);
                        \draw(0,0.1)\bdot;
        \end{tikzpicture}
        ~$-$~
       \begin{tikzpicture}[baseline = 7.5pt, scale=0.5, color=\clr]
            \draw[-,thick] (0,0) to[out=up, in=down] (1,2);
            \draw[-,thick] (0,2) to[out=up, in=down] (0,2.2);
            \draw[-,thick] (1,0) to[out=up, in=down] (0,2);
            \draw[-,thick] (1,2) to[out=up, in=down] (1,2.2);
             \draw(1,1.9)\bdot;
        \end{tikzpicture}
        ~=~
       \begin{tikzpicture}[baseline = 10pt, scale=0.5, color=\clr]
          \draw[-,thick] (2,2) to[out=down,in=right] (1.5,1.5) to[out=left,in=down] (1,2);
            \draw[-,thick] (2,0) to[out=up, in=right] (1.5,0.5) to[out=left,in=up] (1,0);
        \end{tikzpicture}
        ~$-$~
        \begin{tikzpicture}[baseline = 7.5pt, scale=0.5, color=\clr]
            \draw[-,thick] (0,0) to[out=up, in=down] (0,2);
            \draw[-,thick] (1,0) to[out=up, in=down] (1,2);
                   \end{tikzpicture},
  \item[(b)]
\begin{tikzpicture}[baseline = 7.5pt, scale=0.5, color=\clr]
            \draw[-,thick] (0,0) to[out=up, in=down] (0,2);
            \draw[-,thick] (1,0) to[out=up, in=down] (1,2);
               \draw(1,1)\bdot;    \end{tikzpicture}
        ~=~ \begin{tikzpicture}[baseline = 10pt, scale=0.5, color=\clr]
           \draw[-,thick] (0,0) to[out=up, in=down] (1,1);
            \draw[-,thick] (1,1) to[out=up, in=down] (0,2);
            \draw[-,thick] (1,0) to[out=up, in=down] (0,1);
            \draw[-,thick] (0,1) to[out=up, in=down] (1,2);
             \draw(0,1)\bdot;       \end{tikzpicture}
              ~$+$~  \swap  ~$-$~
  \begin{tikzpicture}[baseline = 10pt, scale=0.5, color=\clr]
          \draw[-,thick] (2,2) to[out=down,in=right] (1.5,1.5) to[out=left,in=down] (1,2);
            \draw[-,thick] (2,0) to[out=up, in=right] (1.5,0.5) to[out=left,in=up] (1,0);
        \end{tikzpicture}.
                   \end{itemize}
\end{Lemma}
\begin{proof}
Define  $ P= \begin{tikzpicture}[baseline = 7.5pt, scale=0.5, color=\clr]
            \draw[-,thick] (0,0) to[out=up, in=down] (1,2);
            \draw[-,thick] (0,2) to[out=up, in=down] (0,2.2);
            \draw[-,thick] (1,0) to[out=up, in=down] (0,2);
            \draw[-,thick] (1,2) to[out=up, in=down] (1,2.2);
                        \draw(0,1.9)\bdot;
        \end{tikzpicture}
        ~-~
        \begin{tikzpicture}[baseline = 7.5pt, scale=0.5, color=\clr]
            \draw[-,thick] (0,0) to[out=up, in=down] (1,2);\draw[-,thick] (0,0) to[out=up, in=down] (0,-0.2);
             \draw[-,thick] (1,0) to[out=up, in=down] (0,2);\draw[-,thick] (1,0) to[out=up, in=down] (1,-0.2);
                        \draw(1,0.1)\bdot;
        \end{tikzpicture}$,   $Q=\begin{tikzpicture}[baseline = 7.5pt, scale=0.5, color=\clr]
            \draw[-,thick] (0,0) to[out=up, in=down] (1,2);\draw[-,thick] (0,0) to[out=up, in=down] (0,-0.2);
             \draw[-,thick] (1,0) to[out=up, in=down] (0,2);\draw[-,thick] (1,0) to[out=up, in=down] (1,-0.2);
                        \draw(0,0.1)\bdot;
        \end{tikzpicture}
        ~-~
       \begin{tikzpicture}[baseline = 7.5pt, scale=0.5, color=\clr]
            \draw[-,thick] (0,0) to[out=up, in=down] (1,2);
            \draw[-,thick] (0,2) to[out=up, in=down] (0,2.2);
            \draw[-,thick] (1,0) to[out=up, in=down] (0,2);
            \draw[-,thick] (1,2) to[out=up, in=down] (1,2.2);
             \draw(1,1.9)\bdot;
        \end{tikzpicture}$ and $R=\begin{tikzpicture}[baseline = 10pt, scale=0.5, color=\clr]
          \draw[-,thick] (2,2) to[out=down,in=right] (1.5,1.5) to[out=left,in=down] (1,2);
            \draw[-,thick] (2,0) to[out=up, in=right] (1.5,0.5) to[out=left,in=up] (1,0);
        \end{tikzpicture}
        ~-~
        \begin{tikzpicture}[baseline = 7.5pt, scale=0.5, color=\clr]
            \draw[-,thick] (0,0) to[out=up, in=down] (0,2);
            \draw[-,thick] (1,0) to[out=up, in=down] (1,2);
                   \end{tikzpicture}$.
                   By \eqref{OB relations 1 (symmetric group)} and \eqref{right cuaps - down cross} ,
                   we have
                   \begin{equation}\label{swape2}
                   \swap \circ P\circ \swap= Q, \quad \text{ and }\quad  \swap \circ R\circ \swap=R.
                   \end{equation}
                   So,  (a)  follows from \eqref{AOBC relations}
                    and \eqref{swape2}. By composing   both sides of (a) with $S$, (b) follows from \eqref{OB relations 1 (symmetric group)} and
                    \eqref{right cuaps - down cross}.
                    \end{proof}
Suppose that $d\in \mathbb D_{m,s}$.
There are two kinds  of movements on $d$ as follows:
\begin{equation}\label{movements}
\begin{aligned}
&\text{Movement I:~~ slide   $\bullet$'s along each strand;}\\
&\text{Movement II:}\quad \begin{tikzpicture}[baseline = 7.5pt, scale=0.5, color=\clr]
            \draw[-,thick] (0,0) to[out=up, in=down] (0,2);
            \draw[-,thick] (1,0) to[out=up, in=down] (1,2);
               \draw(1,1)\bdot; \draw(1.4,1)node{$h$};   \end{tikzpicture}\leftrightsquigarrow \begin{tikzpicture}[baseline = 10pt, scale=0.5, color=\clr]
            \draw[-,thick] (0,0) to[out=up, in=down] (1,1);\draw[-,thick] (1,1) to[out=up, in=down] (0,2);
            \draw[-,thick] (1,0) to[out=up, in=down] (0,1);
            \draw[-,thick] (0,1) to[out=up, in=down] (1,2);
             \draw(0,1)\bdot; \draw(0.4,1)node{\footnotesize$h$};\end{tikzpicture}, h\in\mathbb N.
\end{aligned}
\end{equation}
If   $\text{deg}(d)=k$ and  $d'$ is obtained from $d$ by moving $\bullet$'s on $d$ via movements in types I and II, by Lemmas~\ref{capre}, ~\ref{slides relations vertical} and \eqref{AOBC relations},

    \begin{equation}\label{slideleft}
  d\sim d' ~\text{ and } ~d\equiv \pm d' \pmod{\Hom_{\AB}(\ob m, \ob s)_{\le k-1}}.
    \end{equation}

   \begin{Lemma} \label{admissible} For any odd positive integer  $k$,  $2\Delta_k=-\Delta_{k-1}+\sum_{j=1}^{k}(-1)^{j-1}\Delta_{j-1}\Delta_{k-j}$.
\end{Lemma}

 \begin{proof}   By  \eqref{defofdelta}, we have
 $$\begin{aligned}\Delta_k=\begin{tikzpicture}[baseline = 5pt, scale=0.5, color=\clr]
        \draw[-,thick] (0.6,1) to (0.5,1) to[out=left,in=up] (0,0.5)
                        to[out=down,in=left] (0.5,0)
                        to[out=right,in=down] (1,0.5)
                        to[out=up,in=right] (0.5,1);
        \draw (0,0.5) \bdot;
        \draw (0.4,0.5) node{\footnotesize{$k$}};
    \end{tikzpicture}&\overset{\text{Lemma~\ref{capre}}}=-\begin{tikzpicture}[baseline = 5pt, scale=0.5, color=\clr]
        \draw[-,thick] (1.1,2) to (1,2) to[out=left,in=up] (0,1)
                        to[out=down,in=left] (1,0)
                        to[out=right,in=down] (2,1)
                        to[out=up,in=right] (1,2);
        \draw (0,1) \bdot;\draw (2,1)\bdot;
        \draw (0.75,1) node{\tiny{$k-1$}}; \draw (2.4,1) node{\footnotesize{$1$}};
    \end{tikzpicture}\overset{\eqref{right cuaps - down cross}}=-
    \begin{tikzpicture}[baseline = 5pt, scale=0.5, color=\clr]
        \draw[-,thick](1.1,2) to(1,2)to[out=left,in=up](0,1) ;
        \draw[-,thick](1,2)to[out=right,in=up](2,1);
          \draw [-,thick] (0,1) to (0,0.75); \draw [-,thick] (2,1) to (2,0.75);
            \draw [-,thick](2,-0.5) to[out=up,in=down](0,0.75);  \draw [-,thick] (0,-0.5)to [out=up,in=down](2,0.75);
        \draw (0,1) \bdot;\draw (2,1) \bdot;
        \draw (0.75,1) node{\tiny {$k-1$}};\draw (2.4,1) node{\tiny {$1$}};
        \draw [-,thick] (0,-0.5) to [out=down,in=left](1,-1.8) to [out=right,in=down](2,-0.5);
    \end{tikzpicture} \overset{ Lemma~\ref{slides relations vertical}   }=- \begin{tikzpicture}[baseline = 5pt, scale=0.5, color=\clr]
        \draw[-,thick](1.1,2) to(1,2)to[out=left,in=up](0,1) ;
        \draw[-,thick](1,2)to[out=right,in=up](2,1);
          \draw [-,thick] (0,1) to (0,0.75); \draw [-,thick] (2,1) to (2,0.75);
            \draw [-,thick] (2,-0.4)to[out=up,in=down](0,0.75);  \draw [-,thick](0,-0.4) to[out=up,in=down](2,0.75);
        \draw (0,1) \bdot;\draw (0,-0.6) \bdot;
        \draw (0.75,1) node{\tiny {$k-1$}};\draw (0.4,-0.6) node{\tiny {$1$}};
        \draw [-,thick] (0,-0.4) to [out=down,in=left](1,-1.8) to [out=right,in=down](2,-0.4);
    \end{tikzpicture}+\Delta_{k-1}\Delta_0-\Delta_{k-1}
\\
&= (-1)^k \begin{tikzpicture}[baseline = 5pt, scale=0.5, color=\clr]
        \draw[-,thick](1.1,2) to(1,2)to[out=left,in=up](0,1) ;
        \draw[-,thick](1,2)to[out=right,in=up](2,1);
          \draw [-,thick] (0,1) to (0,0.75); \draw [-,thick] (2,1) to (2,0.75);
               \draw [-,thick] (2,-0.4)to[out=up,in=down](0,0.75);  \draw [-,thick](0,-0.4) to[out=up,in=down](2,0.75);
        \draw (0,-0.6) \bdot;
        \draw (0.4,-0.6) node{\tiny {$k$}};
        \draw [-,thick] (0,-0.4) to [out=down,in=left](1,-1.8) to [out=right,in=down](2,-0.4);
    \end{tikzpicture}
     +\sum_{b=1}^k(-1)^b\Delta_{k-1}+ \sum_{j=1}^{k}(-1)^{j-1}\Delta_{j-1}\Delta_{k-j} \text{ (by  Lemmas~\ref{capre}, \ref{slides relations vertical}(a))}  \\
    &\overset{\eqref{right cuaps - down cross}}= -\Delta_k-\Delta_{k-1}+ \sum_{j=1}^{k}(-1)^{j-1}\Delta_{j-1}\Delta_{k-j}.
 \end{aligned}   $$
\end{proof}

Given  a $d\in\mathbb {ND}_{m,s}$,  let $\hat d$ be obtained from $d$  by removing all bubbles and all  $\bullet$'s   on it. Then
  $\hat{d}\in\bar{\mathbb{B}}_{m,s}$.
 For example, if
\begin{equation}\label{undot example}
    d=~
    \begin{tikzpicture}[baseline = 25pt, scale=0.35, color=\clr]
        \draw[-,thick] (2,0) to[out=up,in=down] (0,5);
        \draw[-,thick] (6,0) to[out=up,in=down] (6,5);
         \draw[-,thick] (6.7,0) to[out=up,in=down] (6.7,5);
               \draw[-,thick] (0,0) to[out=up,in=left] (2,1.5) to[out=right,in=up] (4,0);
        \draw[-,thick] (2,5) to[out=down,in=left] (3,4) to[out=right,in=down] (4,5);
        \draw[-,thick] (-1.1,4) to (-1,4) to[out=left,in=up] (-2,3) to[out=down,in=left] (-1,2)
                        to[out=right,in=down] (0,3) to[out=up,in=right] (-1,4);
                        \draw(0,0.3)\bdot;
    \draw(-2,3)\bdot; \draw(6,4.6)\bdot;\draw(3.95,4.6)\bdot;
    \draw(0,4.6)\bdot;
     \draw (-1.5,3) node{$4$};\draw (0.5,0.4) node{$l$};
           \end{tikzpicture}~,
    \quad\text{then}~
    \hat d=~
     \begin{tikzpicture}[baseline = 25pt, scale=0.35, color=\clr]
        \draw[-,thick] (2,0) to[out=up,in=down] (0,5);
        \draw[-,thick] (6,0) to[out=up,in=down] (6,5);
         \draw[-,thick] (6.7,0) to[out=up,in=down] (6.7,5);
               \draw[-,thick] (0,0) to[out=up,in=left] (2,1.5) to[out=right,in=up] (4,0);
        \draw[-,thick] (2,5) to[out=down,in=left] (3,4) to[out=right,in=down] (4,5);
           \end{tikzpicture}~
    ~.
\end{equation}
Let $\alpha(d)_k$ (resp.,$\beta(d)_k$) be the number of $\bullet$'s on the boundary near the $k$th endpoint at  the top
(resp., bottom) row of $d$. Since $d$ is normally ordered,\begin{itemize} \item  $\alpha(d)_k=0$  if the $k$th endpoint at the top  row  is the left  endpoint of a cup,
\item  $\beta(d)_k=0$  if the $k$th endpoint at the bottom row is either the right endpoint of a cap or the endpoint of a vertical strand.\end{itemize} Thanks to Definition~\ref{D:N.O. dotted  OBC diagram},
\begin{equation}\label{degdregu}
    d=  x_{s}^{\alpha(d)_{s}}\circ\ldots\circ x_1^{\alpha(d)_1} \circ\hat d\circ y_{m}^{\beta(d)_{m}}\circ\ldots\circ y_1^{\beta(d)_1}, \text{ if $d\in\bar{\mathbb {ND}}_{m,s}$},
\end{equation}
where
\begin{equation}\label{defofxy}
x_k=1_{\ob {k-1}}\otimes X\otimes 1_{\ob {s-k}} \text{ and $ y_j=1_{\ob {j-1}}\otimes X\otimes 1_{\ob {m-j}}$, for $1\leq k\leq s$, $1\leq j\leq m$.}
\end{equation}
Note that $x_k=y_k$ if $m=s$.
Suppose $d\in\mathbb {ND}_{m,s}$   such that the number of $\Delta_i$ is $k(d)_i$ for $i\in 2\mathbb N$ (only finitely many   $k(d)_i$'s are non-zero). Then
\begin{equation}\label{regularm}
d=\Delta_0^{k(d)_0}\Delta_2^{k(d)_2} \cdots   x_{s}^{\alpha(d)_{s}}\circ\ldots\circ x_1^{\alpha(d)_1} \circ\hat d\circ y_{m}^{\beta(d)_{m}}\circ\ldots\circ y_1^{\beta(d)_1},
\end{equation}
and
\begin{equation}\label{degd} \text{deg}(d)= \sum_{j} jk(d)_j+\sum_{i}\alpha(d)_i+\sum_{l}\beta(d)_l.\end{equation}

\begin{Lemma}\label{reguequai}
If $d,d'\in \mathbb {ND}_{m,s}$  and $d\sim d'$, then $d=d'$ as morphisms in $\AB$, $\CB^f$ and $\CB^f(\omega)$.
\end{Lemma}
\begin{proof} Write $d$ and $d'$   as those  in \eqref{regularm}.
 By Definition~\ref{D:N.O. dotted  OBC diagram}, we have that $\hat d\sim \widehat {d'}$, $k(d)_i=k(d')_i$, $\alpha(d)_j=\alpha(d')_j$ and $\beta(d)_l=\beta(d')_l$ for all admissible $i,j,l$.
  Thanks to Proposition~\ref{equi123}, $\hat d=\widehat {d'}$  as morphisms of $\AB$. Thus, $d=d'$   as morphisms of $\AB$, $\CB^f$ and $\CB^f(\omega)$.
\end{proof}
Thanks to Lemma~\ref{reguequai},
 we can  identify  each  equivalence class of  $\mathbb {ND}_{m,s}$ with any element in it. We will freely use it later on.

\begin{Lemma}\label{equivajssdj}Suppose that $d\in\mathbb D_{m,s}$ such that $m+s=2r$. \begin{enumerate}
\item There is a $d_1\in \bar{\mathbb{ND}}_{m,s}$ such that
$d\sim \Delta(d) d_1$,
where $\Delta(d)=\Delta_{0}^{k(  d )_0}\Delta_{1}^{k(  d )_1}\cdots $ and $k(  d)_i$ is the number of the bubbles of $ d$  such that there are $i$ $\bullet$'s on each of these bubbles, $i\in\mathbb N$.
\item   Suppose $\text{deg}(d)=k$. Then $d\equiv \pm \Delta(d) d_1 \pmod {\Hom_{\AB}(\ob m,\ob s)_{\le k-1}} $, where  $d_1\in \bar{\mathbb{ND}}_{m,s}$ such that $d_1$ satisfies the condition in (a).
\end{enumerate}
\end{Lemma}

\begin{proof}
Thanks to \eqref{slideleft}, we can assume that $d$ satisfies the conditions in Definition~\ref{D:N.O. dotted  OBC diagram}(c)--(d). So
 \begin{equation}\label{NMO}    d= x_{s}^{\alpha( d )_{s}}\circ\ldots\circ x_1^{\alpha( d )_1} \circ  \bar d\circ y_{m}^{\beta( d )_{m}}\circ\ldots\circ y_1^{\beta( d)_1},\end{equation}
where  \begin{itemize} \item $ \bar d\in\mathbb D_{m,s}$ is  obtained from $  d$ by removing  all $\bullet$'s of $  d$ which are not on the bubbles,
 \item $\alpha(  d )_i$ is    the number of $\bullet$'s on the boundary near the $i$th   endpoint  at  the top
row of $  d$,
\item  $ \beta(  d )_{j}$ is    the number of $\bullet$'s on the boundary near the $j$th  endpoint  at  the bottom
row of $  d$.
\end{itemize}

Thanks to Lemma~\ref{numberss},  there is  a $\hat d_1\in \bar{\mathbb B}_{m,s} $ such that $\hat d_1$ and $\bar d$  correspond to  the same partition of $\{1,2,\ldots,2r\}$.
Let $$ d_1=x_{s}^{\alpha(d_1)_{s}}\circ\ldots\circ x_1^{\alpha(d_1)_1} \circ\hat d_1\circ y_{m}^{\beta(d_1)_{m}}\circ\ldots\circ y_1^{\beta(d_1)_1}$$
such that $\alpha(d_1)_i=\alpha(  d)_i$ and $\beta(d_1)_j=\beta(  d )_{j}$ for all admissible $i,j$.
Then $\alpha(d_1)_k=0$ (resp., $\beta(d_1)_k=0$) if the $k$th endpoint at the top (resp., bottom) row  is the left  endpoint of a cup (resp., the right endpoint of a cap or the endpoint of a vertical strand).
So,  $d_1\in\bar{\mathbb{ND}}_{m,s}$  and $d\sim \Delta(d) d_1$. This completes the proof of (a).

Suppose that $k( d)_i=0$ for all $i>0$.
Then  $\text{deg} (\bar d)=0$, where $\bar d$ is given in \eqref{NMO}.
Since $\Delta_{0}^{k(  d )_0}d_1\sim   d$, we have $\Delta_{0}^{k(  d )_0} \hat d_1\sim  \bar d $.
Thanks to   Proposition~\ref{equi123},
  $\Delta_{0}^{k(  d )_0} \hat d_1= \bar d$ in $\AB$. So, $  d= \Delta_{0}^{k(  d )_0}d_1$.

Suppose $k(d)_h>0$ for some $h\geq 1$, i.e.,
$d$ has  a bubble $D$ on which  there are $h$ $\bullet$'s. Thanks to \eqref{slideleft}, we can assume that all the  $h$ $\bullet$'s are on a segment of the  leftmost boundary of $d$.
%
%
 In the following, we only deal with the case that     $$D=\begin{tikzpicture}[baseline = 5pt, scale=0.5, color=\clr]
        \draw[-,thick] (1.2,2) to (1,2) to[out=left,in=up] (0,1)
                        to[out=down,in=left]
                      (1,0) to[out=right,in=down] (1.8,0.6);
                    \draw[-,thick]   (1.8,1.4) to[out=up,in=right] (1,2);
                    \draw[-,thick](1.8,0.6)to(2.2,1.4); \draw[-,thick](1.8,1.4)to(2.2,0.6);
        \draw[-,thick] (3.2,2) to (3,2) to[out=left,in=up] (2.2,1.4);
           \draw[-,thick]           (2.2,0.6)  to[out=down,in=left](3,0);
           \draw[-,thick]             (3,0)
                        to[out=right,in=down] (4,1)
                        to[out=up,in=right] (3,2);
        \draw (0,1) \bdot;
        \draw (0.4,1) node{\footnotesize{$h$}};
       \end{tikzpicture}. $$
 In general, one can verify it in a similar way.
  We have
\begin{equation}\label{jhuhd}
   d= \begin{tikzpicture}[baseline = 5pt, scale=0.5, color=\clr]
 \draw[-,thick] (0.7,-0.5) to (0.7,2.5) to (4.5,2.5) to (4.5,-0.5)to (0.7,-0.5);
        \draw[-,thick] (1.2,2) to (1,2) to[out=left,in=up] (0,1)
                        to[out=down,in=left]
                      (1,0) to[out=right,in=down] (1.8,0.6);
                    \draw[-,thick]   (1.8,1.4) to[out=up,in=right] (1,2);
                    \draw[-,thick](1.8,0.6)to(2.2,1.4); \draw[-,thick](1.8,1.4)to(2.2,0.6);
        \draw[-,thick] (3.2,2) to (3,2) to[out=left,in=up] (2.2,1.4);
           \draw[-,thick]           (2.2,0.6)  to[out=down,in=left](3,0);
           \draw[-,thick]             (3,0)
                        to[out=right,in=down] (4,1)
                        to[out=up,in=right] (3,2);
        \draw (0,1) \bdot;
        \draw (0.4,1) node{\footnotesize{$h$}};
           \end{tikzpicture}~  \\
      \overset{\eqref{OB relations 2 (zigzags and invertibility)}}
      =
      \begin{tikzpicture}[baseline = 5pt, scale=0.5, color=\clr]
               \draw[-,thick] (-2,0.5)to  (-2,1.2);
               \draw[-,thick]  (0,1.5) to[out=down, in=right] (-0.5,1) to[out=left,in=down] (-1,1.5);
           \draw[-,thick] (-1,1.5) to[out=up,in=right] (-1.5,2) to[out=left,in=up] (-2,1.5) to[out=down,in=up] (-2,0.5);
        \draw (-2,0.8) \bdot;
        \draw (-1.5,1) node{\footnotesize{$h$}};
 \draw[-,thick] (0.7,-0.5) to (0.7,2.5) to (4.5,2.5) to (4.5,-0.5)to (0.7,-0.5);
        \draw[-,thick] (1.2,2) to (1,2) to[out=left,in=up] (0,1.5);
        \draw[-,thick]   (0,0.3)   to[out=down,in=left] (1,0)       to[out=right,in=down] (1.8,0.6);
      \draw[-,thick]   (1.8,1.4) to[out=up,in=right] (1,2);
                    \draw[-,thick](1.8,0.6)to(2.2,1.4); \draw[-,thick](1.8,1.4)to(2.2,0.6);
        \draw[-,thick] (3.2,2) to (3,2) to[out=left,in=up] (2.2,1.4);
           \draw[-,thick]           (2.2,0.6)  to[out=down,in=left](3,0);
           \draw[-,thick]             (3,0)
                        to[out=right,in=down] (4,1)
                        to[out=up,in=right] (3,2);
       \draw[-,thick]  (0,0.3) to[out=up, in=right] (-0.5,0.8) to[out=left,in=up] (-1,0);
            \draw[-,thick] (-1,0) to[out=down,in=right] (-1.5,-0.5) to[out=left,in=down] (-2,0)to (-2,1);
           \end{tikzpicture}
           \end{equation}
where    each rectangle    denotes  the diagram obtained from $d$ by removing   $D$.
  Let  $$\tilde d=\begin{tikzpicture}[baseline = 5pt, scale=0.5, color=\clr]
               \draw[-,thick]  (0,1.5) to[out=down, in=right] (-0.5,1) to[out=left,in=down] (-1,1.5);
 \draw[-,thick] (0.7,-0.5) to (0.7,2.5) to (4.5,2.5) to (4.5,-0.5)to (0.7,-0.5);
        \draw[-,thick] (1.2,2) to (1,2) to[out=left,in=up] (0,1.5);
        \draw[-,thick]   (0,0.3)   to[out=down,in=left] (1,0)       to[out=right,in=down] (1.8,0.6);
      \draw[-,thick]   (1.8,1.4) to[out=up,in=right] (1,2);
                    \draw[-,thick](1.8,0.6)to(2.2,1.4); \draw[-,thick](1.8,1.4)to(2.2,0.6);
        \draw[-,thick] (3.2,2) to (3,2) to[out=left,in=up] (2.2,1.4);
           \draw[-,thick]           (2.2,0.6)  to[out=down,in=left](3,0);
           \draw[-,thick]             (3,0)
                        to[out=right,in=down] (4,1)
                        to[out=up,in=right] (3,2);
       \draw[-,thick]  (0,0.3) to[out=up, in=right] (-0.5,0.8) to[out=left,in=up] (-1,0);
           \end{tikzpicture}.$$
 Then $\Delta(\tilde d)$  is obtained from $ \Delta(d)$ by removing one $\Delta_h$ (i.e., $ \Delta(d)=\Delta_h\Delta(\tilde d)$).
 Note that
   $\tilde d\sim \begin{tikzpicture}[baseline = 10pt, scale=0.5, color=\clr]
                \draw[-,thick] (0,0.5)to[out=up,in=down](0,1.5);
    \end{tikzpicture}\otimes d_1 \Delta(\tilde d) \sim \Delta(\tilde d)\begin{tikzpicture}[baseline = 10pt, scale=0.5, color=\clr]
                \draw[-,thick] (0,0.5)to[out=up,in=down](0,1.5);
    \end{tikzpicture}\otimes d_1$, where $\begin{tikzpicture}[baseline = 10pt, scale=0.5, color=\clr]
                \draw[-,thick] (0,0.5)to[out=up,in=down](0,1.5);
    \end{tikzpicture}\otimes d_1$ is obtained from $d_1$ by horizontal stacking the \begin{tikzpicture}[baseline = 10pt, scale=0.5, color=\clr]
                \draw[-,thick] (0,0.5)to[out=up,in=down](0,1.5);
    \end{tikzpicture} from the left. So, $\begin{tikzpicture}[baseline = 10pt, scale=0.5, color=\clr]
                \draw[-,thick] (0,0.5)to[out=up,in=down](0,1.5);
    \end{tikzpicture}\otimes d_1\in \bar{\mathbb {NB}}_{m+1,s+1}$.
           By induction assumption on $\sum _{i\geq 1}{k(d)_i}$,
        \begin{equation}\label{dedeff}
        \tilde d \equiv \pm~ \Delta(\tilde d)\begin{tikzpicture}[baseline = 10pt, scale=0.5, color=\clr]
                \draw[-,thick] (0,0.5)to[out=up,in=down](0,1.5);
    \end{tikzpicture}\otimes d_1 \equiv \pm~ \begin{tikzpicture}[baseline = 10pt, scale=0.5, color=\clr]
                \draw[-,thick] (0,0.5)to[out=up,in=down](0,1.5);
    \end{tikzpicture}\otimes d_1 \Delta(\tilde d)\pmod {\Hom_{\AB}(\ob {m+1},\ob {s+1})_{\le k-h-1}}.
        \end{equation}
               ~
Then we have
$$\begin{aligned}
d&\equiv \pm ~\Delta_h d_1\Delta(\tilde d), ~\text{~  by \eqref{jhuhd}--\eqref{dedeff}},\\
&\equiv \pm~\Delta_h \Delta(\tilde d)d_1, ~\text{ ~ by induction assumption on $\sum _{i\geq 1}{k(d)_i}$},\\
&\equiv \pm~  \Delta( d)d_1~\pmod {  ~ \Hom_{\AB}(\ob m,\ob s)_{\le k-1}}.
\end{aligned}
$$

\end{proof}

 \begin{Prop}\label{regularspan}
 Suppose  $m, s\in \mathbb N$.
 \begin{enumerate}
 \item The $\kappa$-module $\Hom_{\AB}(\ob m, \ob s)$ is spanned by  $\mathbb {ND}_{m,s}/\sim$.
 \item The $\kappa$-module $\Hom_{\CB^f}(\ob m, \ob s)$ is spanned by  $\mathbb {ND}^a_{m,s}/\sim$.
 \item The $\kappa$-module $\Hom_{\CB^f(\omega)}(\ob m, \ob s)$ is spanned by  $\bar{\mathbb {ND}}^a_{m,s}/\sim$.
 \end{enumerate}
\end{Prop}

 \begin{proof}If $2\nmid (m+s)$, then $\mathbb D_{m,s}=\emptyset$  and hence the result is trivial. Suppose   $2\mid (m+s)$.
In order to prove (a), thanks to Lemma~\ref{reguequai}, it is enough to show that $\Hom_{\AB}(\ob m, \ob s)_{\leq k}$ is spanned by  $\{d\in \mathbb {ND}_{m,s} \mid \text{deg}(d)\le k\}$, $\forall  k\in \mathbb N$.

Suppose that $d\in \mathbb D_{m,s}$ such that $\text{deg}(d)=k$.
By lemma~\ref{equivajssdj},
$$d\equiv \pm    \Delta_0^{k(d)_0} \Delta_1^{k(d)_1}\ldots d_1\pmod {\Hom_{\AB}(\ob m,\ob s)_{\le k-1}},$$
where $d_1$ and $k(d)_i$'s are given in lemma~\ref{equivajssdj}(a).
If $ k(d)_{j}\neq 0$ for some odd number  $j$,  by Lemma~\ref{admissible}, we have  $d\in \Hom_{\AB}(\ob m, \ob s)_{\leq k-1}$. Otherwise, $\Delta_0^{k(d)_0} \Delta_2^{k(d)_2}\ldots d_1\in \mathbb {ND}_{m,s}$. Note that $ \Hom_{\AB}(\ob m,\ob s)_{\le -1}=0$.  Using inductive assumption on $k$ yields
  (a), as required.

  Suppose  $d\in \mathbb D_{m, s}$   and $\text{deg}(d)=k$.
 If $d$ has $a$ $\bullet$'s on  one of its strands, we can use movements I, II in \eqref{movements} repeatedly  to  get a $d'\in \mathbb D_{m, s}$ such that the $a$ $\bullet$'s  are on a segment  of the leftmost boundary of $d'$.
  By \eqref{slideleft},
    $$d\equiv \pm d' \pmod{\Hom_{\CB^f}(\ob m, \ob s)_{\le k-1}}, $$
where $\Hom_{\CB^f}(\ob m, \ob s)_{\leq k}$ is defined similarly as $\Hom_{\AB}(\ob m, \ob s)_{\leq k}$.  Since $f(X)=0$ in $\CB^f$, $X^a$ can be written  as a linear combination of $X^b$'s,  $0\le b\leq a-1$.   So $d'\in \Hom_{\CB^f}(\ob m, \ob s)_{\leq k-1}$.
   Using inductive assumption on $k$ yields
  (b), as required. Finally,  since $\Delta_k1_{\ob m}=\omega_k1_{\ob m}$ in $\CB^f(\omega)$ for $k,m\in\mathbb N$,  (c) follows from the proof of   (b).  \end{proof}

Suppose  $\mfg\in\{\mathfrak {so}_N,\mathfrak {sp}_N\}$.  Thanks to  \cite[(2.11)]{DRV}, we define
\begin{equation} \label{cashi} \Omega=\frac{1}{2} \sum_{i, j\in \underline N} F_{i, j}\otimes F_{j, i}.\end{equation}
Then $\Omega=\frac {1}{2} ( \Delta(C)-C\otimes 1-1\otimes C)$ where $\Delta:\U(\mfg)\rightarrow \U(\mfg)\otimes \U(\mfg)$ is the co-multiplication and
$C$ is the quadratic Casimir element.
For any two objects $M,L$ in  $\U(\mathfrak g)\text{-mod}$, define
\begin{equation}
\Omega(m\otimes l)=\frac{1}{2} \sum_{i, j\in \underline N} F_{i, j}m\otimes F_{j, i}l, \quad m\in M, l\in L.
\end{equation}
Since $C$ is a central element in $\U(\mathfrak g)$, $\Omega\in \End_{\U(\mathfrak g)}(M\otimes L)$.

The category
$\END(\U(\mathfrak g)\text{-mod})$
is a  monoidal category. The unit object is the identity functor $ \text{Id}$. The tensor product of two functors
$F,H: \END(\U(\mathfrak g)\text{-mod})\rightarrow\END(\U(\mathfrak g)\text{-mod})$ is  defined as $F\otimes H:= H\circ F$. The tensor product of two
 natural transformations $\zeta:F\rightarrow G$ and $ \xi: H\rightarrow K $ is defined as $ (\zeta\otimes \xi)_{\ob a}:=\xi_{G(\ob a)}\circ H(\zeta_{\ob a})$ for any
 $\ob a\in\U(\mathfrak g)\text{-mod} $.
\begin{Theorem}\label{affaction}
There is a monoidal functor $\Psi:\AB\to\END(\U(\mathfrak g)\text{-mod})$ such that $\Psi(\ob 0)=\text{Id}$,  $\Psi(\ob 1)=-\otimes V$.
Moreover, for any  $m$ in  a  $\U(\mathfrak g)$-module $M$ and any $v\in V$,
\begin{equation} \label{act111}
\begin{aligned}
    &\Psi(U)_{M}
     =\text{Id}_M\otimes \Phi(U), \ \
    \Psi(A)_{M}=\text{Id}_M\otimes \Phi(A) ,\ \ \Psi(S)_{M}  =\text{Id}_M\otimes \Phi(S),\\
     &    \Psi(X)_{M}(m \otimes v)= \varepsilon_\mfg(\Omega+\frac{1}{2}(1\otimes C))(m\otimes v ),
\end{aligned}\end{equation}
where $\Phi$ is the monoidal functor  in Proposition~~\ref{level1s}.
\end{Theorem}

\begin{proof}    Thanks to  Proposition~\ref{level1s}, $\Psi(U)_{M}, \Psi(A)_{M}$ and $\Psi(S)_{M}$ satisfy the defining relations of $\B$. This verifies
\eqref{OB relations 1 (symmetric group)}--\eqref{Brauer relation 4}.
It is well known that   $C$ acts  on  $L$   via   $(\lambda,\lambda+2\rho)$ (see e.g. \cite[Lemma~8.5.3]{Mu}), where
$L$ is a highest weight $\U(\mfg)$-module with the highest weight $\lambda$ and
\begin{equation}\label{rho123} \rho=\begin{cases}
                         \sum_{i=1}^n (n-i+1)\varepsilon_i, & \text{if $\mathfrak{g}$=$\mathfrak {sp}_{2n}$,} \\
                         \sum_{i=1}^n (n-i)\varepsilon_i, & \text{if $\mathfrak{g}$=$\mathfrak {so}_{2n}$,} \\
                         \sum_{i=1}^n (n-i+\frac{1}{2})\varepsilon_i, & \text{if $\mathfrak{g}$=$\mathfrak {so}_{2n+1}$.}
                        \end{cases}
                         \end{equation} Since $V$ is the irreducible highest weight module with the highest weight $\epsilon_1$,   $C$ acts on $V$ via the scalar $N-\epsilon_\mfg$.
 So,
   \begin{equation}\label{actionofxxx}
   \Psi(X)_M=\varepsilon_\mfg(\Omega+\frac{1}{2}(N-\epsilon_g))|_{M\otimes V}.
   \end{equation}
Define
\begin{equation}\label{defofxk}
X_k=\varepsilon_\mfg(\Omega+\frac{1}{2}(N-\epsilon_g))|_{(M\otimes V^{\otimes k-1})\otimes V}, \text{ for } k=1,2,\ldots.
\end{equation}
It follows from the proof of \cite[Theorem~2.2]{DRV} that \begin{equation}\label{reqqq}\begin{aligned} &
(X_1\otimes \text{Id}_V) \Psi(S)_{M}-\Psi(S)_{M}  X_2=\Psi(U)_{M}\Psi(A)_{M}- \text{Id}_{M\otimes V^{\otimes 2}},\\
& (X_1\otimes\text{Id}_V+X_2) \Psi(U)_{M}\Psi(A)_{M} =0,\\
& \Psi(U)_{M}\Psi(A)_{M}(X_1\otimes\text{Id}_V+X_2)=0.\end{aligned}\end{equation}
Thus,
$(X_1\otimes\text{Id}_V+X_2) \Psi(U)_{M}\Psi(A)_{M}\Psi(U)_{M}=0$ and $\Psi(A)_{M} \Psi(U)_{M}\Psi(A)_{M}(X_1\otimes\text{Id}_V+X_2)=0$.
By \eqref{ewjdngn},  $ \Psi(A)_{M}\Psi(U)_{M}$ can be identified with the non-zero scalar $\varepsilon_{\mathfrak g}N$. So,
$(X_1\otimes\text{Id}_V+X_2) \Psi(U)_{M}=0$ and $\Psi(A)_{M}(X_1\otimes\text{Id}_V+X_2)=0$. Therefore, two equations in Lemma~\ref{capre}(1)-(2) are satisfied. Thanks to  Lemma~\ref{cp2},  \eqref{OB relations 2 (zigzags and invertibility2)} is satisfied. Finally,
 \eqref{AOBC relations} follows from the first equation in \eqref{reqqq}.
\end{proof}


We consider $\U(\fh )$ as a $\U(\fb )$-module by inflation, where $\fb=\fn^+\oplus \fh$.
Motivated by \cite{CK}, define \begin{equation}\label{geri}
 M^{\rm gen}:=\U(\mathfrak g )\otimes_{\U(\fb)} \U(\fh )
\end{equation}
and call it the \emph{generic Verma module} later on.
Let $\left\{f_{1}, \dotsc , f_{s} \right\}$  be a basis of $\fn^{-}$  such that
$f_i=F_{k,j}$ for some $k,j\in\underline N$.
Thanks to  the PBW theorem,  $M^{\rm gen} $ is a free right $\U(\fh)$-module with   basis given by  all  elements
\begin{equation}\label{M basis}
f_1^{a_1}\cdots f_s^{a_s}\otimes1, \quad a_k\in\mathbb N, 1\leq k\leq s.
\end{equation}
Since $\U(\fh)\cong \mathbb C[h_1,\ldots,h_n]$ (see \eqref{hi}),
 there is a natural  $\Z$-grading on $\U(\fh)$  by setting   $\text{deg}h_{i}=1$ for all admissible $i$. Let  $ \U (\fh)_{\le t}
 $ be the subspace  spanned by all monomials in $\U(\fh)$ with degree less than or equal to $t$.
There is a $\Z$-graded  filtration $0\subset M^{\rm gen}_0\subset M^{\rm gen}_1\subset \ldots$ of $M^{\rm gen} $   such that $M^{\rm gen}_{t}$ is  spanned by all $f_1^{a_1}\cdots f_s^{a_s}\otimes \U (\fh)_{\le t}$ for all $a_k$'s$\in\mathbb N$.   This gives rise to a $\mathbb Z$-graded
 structure on  the free right $\U(\fh) $-module $M^{\rm gen}  \otimes V^{\otimes r}$ with  basis given by  all  elements
 \begin{equation}\label{Mr basis}
 f_1^{a_1}\cdots f_s^{a_s}\otimes h_1^{b_1}\cdots h_n^{b_n}\otimes v_{\mathbf i}, \quad \mathbf i\in\underline N^r, a_k, b_j\in\mathbb N, 1\leq k\leq s, 1\leq j\leq n,
\end{equation}
where $v_{\mathbf i} =v_{i_1}\otimes \ldots \otimes v_{i_r}$.
 Later on, we say that  the basis element in \eqref{Mr basis}
 is  of degree $\sum_{i=1}^n b_i$.
 We write any $x\in M^{\rm gen}  \otimes V^{\otimes r}$ as a linear combination of elements in \eqref{Mr basis}.  For any $y$ in \eqref{Mr basis}, we say $y$ is a term of $x$ if $y$ appears  in this expression  with non-zero coefficient.


\begin{Lemma}\label{eM filtered degree 1}
    For any  $1\leq i,j\leq n$,  let $\phi_{\pm i,\pm j}$ (resp., $\phi_{\pm i}$)$:M^{\rm gen} \rightarrow M^{\rm gen}  $ be the linear  map
    such that $\phi_{\pm i,\pm j}(m)=F_{\pm i,\pm j}m$ (resp., $\phi_{\pm i}(m)=F_{0,\pm i}m$) for any $m\in M^{\rm gen} $.
     Then
    \begin{enumerate}
        \item[(1)] $\phi_{-i,\pm j}$ (resp., $\phi_{i}$, $\phi_{-i,i}$ ) is homogeneous of degree 0 for  $i<j$  (resp., $1\leq i\leq n$),
        \item[(2)] $\phi_{i,\pm j}$ (resp., $\phi_{-i}$, $\phi_{i,-i}$, $\phi_{i,i}$) is of  filtered degree $1$ for     $i<j$  (resp., $ 1\leq i\leq n$),
        \item[(3)]  $X_k\in \End_{\mathbb C} (M^{\rm gen} \otimes V^{\otimes k})$ is of  filtered degree $1$, where $X_k$ is given in \eqref{defofxk}.
    \end{enumerate}
\end{Lemma}

\begin{proof}Thanks to \eqref{fij}, we have \begin{equation} \label{comc}[F_{i,j}, F_{k, l}]=\delta_{k, j} F_{i, l}-\delta_{i, l} F_{k, j}+\delta_{k, -i} F_{-l, j}+\delta_{l,-j}F_{k, -i}\end{equation}  for all admissible $i, j, k$ and $l$.
So, $F_{k,l}$ maps $f_1^{a_1}\cdots f_s^{a_s}\otimes1$ to  $M^{\rm gen}_0$ (resp.,$M^{\rm gen}_{\leq 1}$)  if $F_{k,l}\in \mathfrak n^-$ (resp.,  $\mathfrak n^+\oplus\fh$). In the later case, one needs to use  induction on $\sum_{i=1}^sa_i$ together with \eqref{comc}.
Now, (1) and (2)  follow immediately from  this observation   together with the fact   that  $M^{\rm gen} $ is a $(\U(\mathfrak g ), \U(\fh))$-bimodule with the grading coming from the right $\U(\fh)$-action. Finally, (3) follows from the definition of $\Omega$ in \eqref{cashi} together with  (1)--(2).
    \end{proof}

\begin{Lemma}\label{x1 calculation} Let $\hat m:=1\otimes 1\in M^{\rm gen}$.
\begin{enumerate}
 \item[(1)] If $1\leq i\leq n$, then  $\varepsilon_\mfg\hat m h_i\otimes v_i$ (resp.,  $-\varepsilon_\mfg\hat m h_i\otimes v_{-i}$) is the unique  term of    $\Psi_{M^{\rm gen}}(X)(\hat m\otimes v_i)$ (resp., $\Psi_{M^{\rm gen}}(X)(\hat m\otimes v_{-i})$) with degree $1$.
      \item [(2)]All   terms   of  $\Psi_{M^{\rm gen}}(X)(\hat m\otimes v_0)$  are   of degree $0$.
 \item[(3)]When  $k$ is a  positive even integer,  $2\varepsilon_{\mathfrak g}\hat m(h_n^k+\ldots+h_1^k)$  is the unique term   of  $\Psi_{M^{\rm gen}}(\Delta_k)(\hat m)$ with the highest  degree $k$.
 \item[(4)] If $i_k\neq 0$ for some $1\leq k\leq r$, then, up to a sign,  $\hat m h_{|i_k|}\otimes v_{i_1}\otimes\ldots\otimes v_{i_r}$ is the unique term  of  $\Psi_{M^{\rm gen}}(x_k)(\hat m\otimes v_{i_1}\otimes\ldots\otimes  v_{i_r})$ with  degree $1$, where $x_k$ is given in \eqref{defofxy}.
 \item [(5)]    If $i_k= 0$ for some $1\leq k\leq r$, then all  terms   of  $\Psi_{M^{\rm gen}}(x_k)(\hat m\otimes v_{i_1}\otimes\ldots\otimes  v_{i_r})$ are of degree $0$.
\end{enumerate}
\end{Lemma}

\begin{proof} If $i\neq j$, then $F_{i,j}\notin\mathfrak h$ and  $F_{i,j} \hat m$ is of degree $0$. When $i=j$, we have
     $F_{i,i} \hat m=\hat m h_i$ and $F_{-i,-i}\hat m =-\hat m h_i$ for $i>0$. Recall $ \Psi_{M^{\rm gen}}(X)$ is given in \eqref{actionofxxx}. Thus,  both   (1) and (2) follow immediately from the definition of $\Omega$ in  \eqref{cashi}.

Suppose $k$ is a positive even integer.   By Lemma~\ref{eM filtered degree 1}(3), up to a linear combination of some terms with degree less than $k$,  we have
\begin{align*}
        \Psi_{M^{\rm gen}}(\Delta_k)(\hat m)\overset{\eqref{modstr}}=&\Psi_{M^{\rm gen}}(A)\circ \Psi_{M^{\rm gen}}(X^k)\otimes \text{Id}_V(\hat m\otimes \sum_{i\in\underline N}v_i\otimes v^*_{i})\\
                             =\ \ & \Psi_{M^{\rm gen}}(A)(\sum_{1\leq i\leq n}(\varepsilon_{\mathfrak g}\hat m h_i^k \otimes v_i\otimes v_{-i} +(-1)^k\hat mh_i^k \otimes v_{-i}\otimes v_i) ), \text{ by (1)-(2) and \eqref{vistar}} \\
           \overset{\eqref{modstr}} =&  \varepsilon_{\mathfrak g}\sum_{1\leq i\leq n} (\hat m h_i^k + (-1)^k\hat m h_i^k)=
            2\varepsilon_{\mathfrak g}\hat m(h_n^k+\ldots+h_1^k).
                   \end{align*}
   This verifies (3). 
Let $(k,1)\in \mathfrak S_r$ such that it swaps $k$ and $1$ and fixes others.
Thanks to  Definition~\ref{definofs} and Lemma~\ref{slides relations vertical}(b), $x_k-S_{(k,1)}\circ x_1\circ S_{(k,1)}$  is a linear combination of  some elements in  $\mathbb B_{r, r}$ whenever $k>1$. For any    $d\in \mathbb B_{r, r}$, $$\Psi_{M^{\rm gen}}(d)(\hat m\otimes v_{i_1}\otimes\ldots\otimes v_{i_r})\overset{\eqref{act111}}=\hat m\otimes \Phi(d)( v_{i_1}\otimes\ldots\otimes v_{i_r}).$$
  Thus any term  of $\Psi_{M^{\rm gen}}(d)(\hat m\otimes v_{i_1}\otimes\ldots\otimes v_{i_r})$ is  of degree $0$. Now, (4)-(5) follow from this  observation and  (1)-(2)
 \end{proof}

 Thanks to Proposition~\ref{regularspan},  $\Hom_{\AB}(\ob 0,\ob {2r})$ is spanned by
 \begin{equation}\label{lbasis}\{
p_d(\Delta_0, \Delta_2, \ldots) d\mid d\in\bar{\mathbb {ND}}_{0,2r}/\sim \text{ and } p_d(t_0,t_2,\ldots)\in \kappa[t_0, t_2, \ldots]\}.
\end{equation}
 Suppose that $d\in\bar{\mathbb {ND}}_{0,2r}/\sim$. Let $\hat d\in \bar{\mathbb B}_{0,2r}/\sim$  be obtained from $d$ by removing all $\bullet$'s. By \eqref{regularm},
\begin{equation}\label{defwrtofd}
    d=  x_{2r}^{\alpha(d)_{2r}}\circ\ldots\circ x_1^{\alpha(d)_1} \circ\hat d,
\end{equation}
  and $\alpha(d)_k=0$ if the $k$th point at the top row of $d$ is the left endpoint of a cup.
Recall the notations $i(\hat d), j(\hat d), \theta( \hat d)$ and $z(\hat d)$  in Definition~\ref{defofij}, \eqref{defofwdbed}, and  \eqref{defofzd}.
Suppose that  $v_{\mathbf i}\in V^{\otimes 2r}$, $\mathbf i\in \underline N^{2r}$. The  $v_{\mathbf i}$-component of $M^{\rm gen}\otimes V^{\otimes 2r}$ is defined  to be the subspace  spanned by $ y\otimes v_{\mathbf i}$ for all  $y\in M^{\rm gen}$.
\begin{Lemma}\label{top component lemma}
    Suppose that $d, d'\in\bar{\mathbb {ND}}_{0,2r}/\sim$ and $\hat m=1\otimes 1\in M^{\rm gen}$.
  \begin{enumerate}
  \item[(1)] If there is a   term   of $\Psi_{M^{\rm gen}}(d')(\hat m )$  in the
   $v_{\theta(\hat d)}$-component of $M^{\rm gen}\otimes V^{\otimes 2r}$ with the highest degree $\text{deg}(d')$,  then  $\hat d= \widehat d'$.
  \item [(2)] Assume $h(d')= h_{1}^{\alpha(d')_{j_1}} \cdots h_{r}^{\alpha(d')_{j_r}}$, where
       $i(\widehat {d'})=(i_1,\ldots,i_r)$ and  $j(  \widehat{d'})=(j_1,\ldots, j_r)$.  If  $\hat{d}= \widehat{d'}\in\bar{\mathbb B}_{0,2r}/\sim$, then, up to a sign,  $\hat mh(d') \otimes v_{\theta(\hat{d})}$ is  the unique term of  $\Psi_{M^{\rm gen}}(d')(\hat m)$  in the  $v_{\theta(\hat d)}$-component of  $M^{\rm gen}\otimes V^{\otimes 2r}$ with the highest  degree  $\text{deg} (d')$.\end{enumerate}
\end{Lemma}
\begin{proof}   By \eqref{defwrtofd} and \eqref{defofzd1}, up to a sign,
\begin{equation}\label{phid}
\Psi_{M^{\rm gen}}(d')(\hat m)=\Psi_{M^{\rm gen}}(x_{2r}^{\alpha(d')_{2r}}\circ\ldots\circ x_1^{\alpha(d')_1})(\hat m\otimes z(\widehat{d'})),
\end{equation}
 where $z(\widehat{d'})$ is given in \eqref{defofzd} and $\alpha(d')_{i_l}=0$ for $1\leq l\leq r$. By \eqref{act111} and \eqref{actionofxxx}, $\Psi_{M^{\rm gen}}(x_k)=X_k\otimes \text{Id}_{V^{\otimes 2r-k}}$, $1\leq k\leq 2r$, where $X_k$ is given in \eqref{defofxk}.
 Thanks to Lemma~\ref{eM filtered degree 1}(3), the highest degree term of $\Psi_{M^{\rm gen}}(d')(\hat m)$ is of degree $\text{deg}(d')$.
By Lemma~\ref{x1 calculation}(4)--(5) and \eqref{phid},   the  summation of  terms of $\Psi_{M^{\rm gen}}(d')(\hat m)$ with the highest degree $\text{deg}(d')$ is in some
$v_{\mathbf i}$-component of $M^{\rm gen}\otimes V^{\otimes 2r}$ such that each  $v_{\bf i}$ appears in $z(\widehat{d'})$.
  In  the proof of Theorem~\ref{basisofb}, we have proven that $ \hat d= \widehat d'$  if $v_{\theta(\hat d)}$ appears in $z(\widehat{d'})$.  This verifies  (1).
 If $\hat{d}= \widehat{d'}$, by arguments in the proof of Theorem~A,   $v_{\theta(\hat d)}$ appears in $z(\widehat{d'})$ with coefficient $\pm 1$.
  By \eqref{defofwdbed}, $\theta(\hat d)_{j_l}=-l$, for $1\leq l\leq r$.  Using  Lemma~\ref{x1 calculation}(4)--(5) repeatedly together with \eqref{phid} yields (2).
\end{proof}

\begin{Prop}\label{Basis theorem for affine Hecke-Clifford with bubbles} Suppose  $\kappa=\mathbb C$. Then $\Hom_\AB(\ob 0,\ob {2r})$ has  basis given by $\mathbb{ND}_{0,2r}/\sim$.
\end{Prop}

\begin{proof}
Thanks to Proposition~\ref{regularspan},
     it is enough to prove  that   the required elements are  linear independent over $\mathbb C$. By \eqref{lbasis}, it suffices to prove  $p_{d}(t_0,t_2, t_4,\ldots)=0$ for all $d\in P$ if
    \begin{equation}\label{linear combo}
        \sum_{d\in P}p_{d}(\Delta_0,\Delta_2,\Delta_4,\ldots)d=0,
    \end{equation}
    where $P$ is any  finite subset of $\bar{\mathbb {ND}}_{0,2r}/\sim$,  and   $p_{d}$'s$\in \mathbb C[t_0, t_2, t_4, \ldots]$.
    For the simplification of notation, we denote $ p_{d}(\Delta_0,\Delta_2,\Delta_4,\ldots)
  $ by $p_d(\Delta)$. Let $W=\{d \in P\mid p_{d}\neq 0 \}$.

   We claim that  $W=\emptyset $. Otherwise,    $p_d=\sum_{\beta\in Q}f_\beta(t_0)t^\beta$
     for any $d\in W$,   where $f_\beta(t_0)\in\mathbb C[t_0] $, $t^\beta:= t_2^{\beta_1}t_4^{\beta_2}\cdots$ and $Q=\{\beta=(\beta_1,\beta_2,\ldots)\mid \beta_i\in \mathbb N, f_\beta(t_0)\neq 0\}$. Note that $Q$ is a finite set. We can find an $n\in \mathbb N$ and  $n\gg 0$ such that  $f_\beta(\varepsilon_{\mathfrak g}N)\neq 0$ for all $\beta\in Q$.
 Since $W$ is finite and  all the power sum symmetric polynomials are algebraic independent, we can choose $n\gg 0$ such that
   $$p_{d}(\varepsilon_{\mathfrak g}N,2\varepsilon_\mfg(h_n^2+\ldots+h_1^2),2\varepsilon_\mfg(h_n^4+\ldots+h_1^4),\ldots)\neq 0, \text{ for all } d\in W.$$
  Let $g_d(h_1,h_2,\ldots,h_n)$ be the summation of terms of   $p_{d}(\varepsilon_{\mathfrak g}N,2\varepsilon_\mfg(h_n^2+\ldots+h_1^2),2\varepsilon_\mfg(h_n^4+\ldots+h_1^4),\ldots)$ with the highest degree. Then $g_d(h_1,h_2,\ldots,h_n)$ is a homogenous symmetric polynomial.

  By Lemma~\ref{x1 calculation}(3), the summation of terms of  $\Psi_{M^{\rm gen}}(p_d(\Delta))( \hat m )$ with the highest degree  is $ \pm \hat m g_d(h_1,h_2,\ldots,h_n) $,  where $d\in W$.   Choose  $d_0\in W$ such that
       $ \text{deg} d_0 + \text{deg} g_{d_0}$ is maximal.
  Define
    $$B=\{d\in W \mid \hat d= \widehat d_0\text{ and }\text{deg} d_0 + \text{deg} g_{d_0}=\text{deg} d+\text{deg} g_{d}\}.$$
  Thanks to  \eqref{defwrtofd}, we write $d=x_{2r}^{\alpha(d)_{2r}}\circ\ldots\circ x_1^{\alpha(d)_1}\circ \widehat d_0$ for all  $d\in B$.
       By Lemma~\ref{top component lemma},  the summation of terms of  $\Psi_{M^{\rm gen}}(\sum_{d\in P}p_d(\Delta_0,\Delta_2,\Delta_4,\ldots)d)(\hat m)$, which is  in the $v_{\theta(\widehat d_0)} $-component of $M^{\rm gen}\otimes V^{\otimes 2r}$ with the  highest degree,  is
       equal to
    \begin{equation}\label{top degree component under Psi_M}
        \sum_{d\in B}\pm \hat m g_d(h_1,h_2,\ldots,h_n)h(d)    \otimes v_{\theta(\widehat  d_0)},
    \end{equation}
    where
   $h(d)$'s are  defined in Lemma~\ref{top component lemma}(2).

For any two monomials $h_{n}^{r_{n}}\dotsb h_{1}^{r_{1}},  h_{n}^{s_{n}}\dotsb h_{1}^{s_{1}}\in \U(\fh)$, we say  $h_{n}^{r_{n}}\dotsb h_{1}^{r_{1}}$ is  less than  $h_{n}^{s_{n}}\dotsb h_{1}^{s_{1}}$ and write $h_{n}^{r_{n}}\dotsb h_{1}^{r_{1}} < h_{n}^{s_{n}}\dotsb h_{1}^{s_{1}}$ if either $\sum_{i} r_{i} < \sum_{i} s_{i}$ or $\sum_{i} r_{i} = \sum_{i} s_{i}$ and there is a $t\in \{1, 2, \ldots, n\}$ such that  $r_t<s_t$ and $r_j=s_j$ whenever $j>t$. This gives rise to a total  ordering on the set of all monomials of
  $\U(\fh)$.

 Since   $g_d(h_1,h_2,\ldots,h_n)$ is a homogenous symmetric polynomial, we can    choose an  $n\in \mathbb N$ and $n\gg 0$  such that
      for each $d\in B$,  the leading monomial of  $g_d(h_1,h_2,\ldots,h_n)$ with respect to the previous  total  ordering  does not contain a factor in $\{h_{1},h_{2},\ldots,h_{r}\}$.
        This shows that  any  factor of $h(d)$ could not be a factor of the  leading monomial of each $g_d(h_1,h_2,\ldots,h_n)$.
      Note that $h(d)\neq h(d')$ if $d, d'\in B$ and $d\neq d'$.
         Therefore,
 the leading monomials of $g_d(h_1,h_2,\ldots,h_n)h(d)$'s are pairwise distinct for all $d\in B$, forcing
  the summation in  \eqref{top degree component under Psi_M} is nonzero. This   contradicts   \eqref{linear combo}. So, $W=\emptyset$  and the result is proven.\end{proof}

\begin{proof}[\textbf{Proof of Theorem~\ref{Cyclotomic basis conjecture}}] If $m+s$ is not even, then $ \Hom_{\AB}(\ob m,\ob s)=0$ and Theorem~\ref{Cyclotomic basis conjecture} is trivial. Suppose that $m+s=2r$ for some $r\in \mathbb N$.
First, we consider $\AB$ over  $\mathbb Z_{(2)}$ and $(m, s)=(0, 2r)$, where $\mathbb Z_{(2)}$ is the localization of $\mathbb Z$ by the set $\{2^a \mid a\in \mathbb N\}$. Thanks to Propositions~\ref{regularspan}, it is enough to verify  that  $\mathbb {ND}_{0,2r}/\sim$
is linear independent over $\mathbb Z_{(2)}$. By  Proposition~\ref{Basis theorem for affine Hecke-Clifford with bubbles}, we immediately have  such  a result over $\mathbb C$ and hence over $\mathbb Z_{(2)}$. This proves Theorem~\ref{Cyclotomic basis conjecture} over $\mathbb Z$ when $(m, s)=(0, 2r)$.

Suppose  $m>0$. Recall  $ \bar\eta _{\ob m}$
  in Definition~\ref{etae}.
  By \eqref{usuflelem}, $\bar \eta _{\ob m}(d)\in \mathbb D_{0,2r}$ with degree $k$ if
     $d\in\mathbb{ND}_{m,s}$ and  $\text{deg}(d)=k$.
  Moreover,  conditions (a)--(b) in Definition~\ref{D:N.O. dotted  OBC diagram} are satisfied for $\bar \eta _{\ob m}(d)$.
  We move $\bullet$'s on each cup of $\bar \eta _{\ob m}(d)$ to the right endpoint and get a $d'\in\mathbb {ND}_{0,2r} $ such that $d'\sim \bar \eta _{\ob m}(d) $ and $\text{deg}(d')=k$.
   For example, if   $d\in \Hom_{\AB}(\ob 4,\ob 2)_{\leq 2}$ as follows, then the  corresponding $\bar\eta _{\ob m}(d)$ and  $d'$  are depicted as follows:
$$
d=\begin{tikzpicture}[baseline = 10pt, scale=0.5, color=\clr]
        \draw[-,thick] (0,0) to[out=up, in=down] (1,2);
        \draw[-,thick] (1,0) to[out=up, in=down] (0,2);
               \draw (0,1.8)\bdot;
                         \draw[-,thick] (2,0) to[out=up,in=left] (3,1) to[out=right,in=up] (4,0);
         \draw (2.1,0.2)\bdot;
                 \end{tikzpicture},\quad
    \bar\eta _{\ob m}(d)=\begin{tikzpicture}[baseline = 10pt, scale=0.5, color=\clr]
        \draw[-,thick] (0,1) to[out=up, in=down] (1,2);
        \draw[-,thick] (1,1) to[out=up, in=down] (0,2);
               \draw (0,1.8)\bdot;
                         \draw[-,thick] (2,1) to[out=up,in=left] (2.5,1.8) to[out=right,in=up] (3,1);
         \draw (2.1,1.2)\bdot;
         \draw[-,thick] (0,1) to[out=down,in=left] (3.5,-0.8) to[out=right,in=down] (7,1);\draw[-,thick] (7,1)to(7,2);
           \draw[-,thick] (1,1) to[out=down,in=left] (3.5,-0.6) to[out=right,in=down] (6,1);\draw[-,thick] (6,1)to(6,2);
           \draw[-,thick] (2,1) to[out=down,in=left] (3.5,0) to[out=right,in=down] (5,1);\draw[-,thick] (5,1)to(5,2);
           \draw[-,thick] (3,1) to[out=down,in=left] (3.5,0.2) to[out=right,in=down] (4,1);\draw[-,thick] (4,1)to(4,2);
                            \end{tikzpicture},\quad
              d'=\begin{tikzpicture}[baseline = 10pt, scale=0.5, color=\clr]
        \draw[-,thick] (0,1) to[out=up, in=down] (1,2);
        \draw[-,thick] (1,1) to[out=up, in=down] (0,2);
               \draw (5,1.8)\bdot;
                         \draw[-,thick] (2,1) to[out=up,in=left] (2.5,1.8) to[out=right,in=up] (3,1);
         \draw (6,1.8)\bdot;
         \draw[-,thick] (0,1) to[out=down,in=left] (3.5,-0.8) to[out=right,in=down] (7,1);\draw[-,thick] (7,1)to(7,2);
           \draw[-,thick] (1,1) to[out=down,in=left] (3.5,-0.6) to[out=right,in=down] (6,1);\draw[-,thick] (6,1)to(6,2);
           \draw[-,thick] (2,1) to[out=down,in=left] (3.5,0) to[out=right,in=down] (5,1);\draw[-,thick] (5,1)to(5,2);
           \draw[-,thick] (3,1) to[out=down,in=left] (3.5,0.2) to[out=right,in=down] (4,1);\draw[-,thick] (4,1)to(4,2);
                            \end{tikzpicture}.
                 $$
   By \eqref{slideleft},
 \begin{equation}\label{etadff}
 \bar\eta _{\ob m}(d)\equiv \pm d' \pmod{\Hom_{\AB} (\ob 0, \ob{2r})_{\le k-1}}.
 \end{equation}
Suppose that  $d_1, d_2\in \mathbb{ND}_{m,s}$ and $\text{deg}(d_1)=\text{deg}(d_2)=k$.  By Lemma~\ref{eaquvivalents}, $\bar \eta _{\ob m}(d_1)\sim\bar \eta _{\ob m}(d_2)$ if and only if $d_1\sim d_2$. Hence, $d_1'\sim d_2'$ if and only if $d_1\sim d_2$.
  So   the map $d  \mapsto d'$ is an injective map between the set of all elements in $ \mathbb {ND}_{m,s}/\sim$ with degree $k$  and the set of  all elements in $ \mathbb {ND}_{0,2r}/\sim$ with degree $k$.
Suppose that $$\sum_{d\in B}c_d d=0,$$
where $c_d\in \kappa$ and $B$ is a finite subset of $ \mathbb {ND}_{m,s}/\sim$.
Let $k=\max\{\text{deg}d\mid d\in B\}$.

If $k=0$, then $\bar\eta_{\ob m}(d)\in \mathbb {B}_{0,2r}/\sim$  and $\bar\eta_{\ob m}(d)= d'$ in \eqref{etadff}.
Moreover, $\sum_{d\in B}c_d  \bar\eta_{\ob m}(d)=0$. Thanks to the result on $\Hom_{\AB}(\ob 0,\ob {2r})$, $c_d=0$ for all $d\in B$.

Suppose that  $k>0$. Let $B_1=\{d\in B\mid \text{deg}d=k\}$.
It follows from  the proof of Proposition~\ref{regularspan} that $\Hom_{\AB} (\ob 0, \ob{2r})_{\le k-1}$ is spanned by all elements in $ \mathbb {ND}_{0,2r}/\sim$ with degree  $<k$.  By \eqref{etadff}, we have $$\sum_{d\in B_1}\pm c_d  d'\equiv 0\pmod { \Hom_{\AB}(\ob 0, \ob {2r})_{\le k-1}},$$
where $d'$'s are   elements in $ \mathbb {ND}_{0,2r}/\sim$  with degree $k$.
It follows from the result on $\Hom_{\AB}(\ob 0,\ob {2r})$ again that $c_d=0$ for all $d\in B_1$.
By induction on $k$ we see that $c_d=0$ for all $d\in B$.
Therefore, the set  $ \mathbb {ND}_{m,s}/\sim$
 is   linearly independent.
  This proves Theorem~B   over $\mathbb Z_{(2)}$.

  When $\kappa$ is a commutative ring containing multiplicative identity $1$ and invertible element $2$,  we consider $\kappa$ as  a $\mathbb Z_{(2)}$-module in a standard way.  Now, standard arguments on base change property  shows Theorem~\ref{Cyclotomic basis conjecture}  over $\kappa$ (see  the arguments at the end of the  proof of Theorem~\ref{basisofb}). \end{proof}

At the end of this section, we establish a relationship between $\AB$ and the affine Wenzl algebra in \cite{Na}. This algebra is also called  the    affine Nazarov-Wenzl algebra in \cite{AMR}.

\begin{Defn}\label{definition of Na}
\cite{Na} The affine Nazarov-Wenzl algebra $\bar{\mathcal W}_r^{\rm aff}$ is the unital associative $\kappa$-algebra generated by
 $\{e_i, s_i\mid 1 \le i\le r-1\} $ and $\{\ob  x_j\mid
1\le j\leq r\}$  and a class of central elements  $\hat\omega=\{\hat\omega_j\mid j\in \mathbb N\}$, subject to the  relations:
\begin{multicols}{2}
\begin{enumerate}
\item [(1)] $s_i^2=1$,  $1\le i< r$,
\item[(2)] $s_is_j=s_js_i$,  $|i-j|>1$,
\item[(3)] $s_is_{i+1}s_i\!=\!s_{i+1}s_is_{i+1}$, $1\!\le\! i\!<\!r\!-\!1$,
\item[(4)]$e_i s_i=e_i=s_ie_i$, $1\leq i\leq r-1$,
\item[(5)]  $e_1^2=\hat\omega_0 e_1$,
\item[(6)] $s_ie_j=e_js_i$, $e_ie_j=e_je_i$, if  $|i-j|>1$,
\item[(7)]$s_ie_{i+1}e_i=s_{i+1}e_i$, $1\leq i\leq r-2$,
\item [(8)]$e_i e_{i+1}s_i =e_i s_{i+1}$, $1\leq i\leq r-2$,
\item[(9)]$e_i e_{i+1}e_i =e_{i+1}$, $1\leq i\leq r-2$,
\item[(10)]$ e_{i+1}e_i e_{i+1} =e_i$, $1\leq i\leq r-2$,
\item [(11)]  $e_1\ob x_1^ke_1=\hat\omega_k e_1$,  $\forall k\in \mathbb Z^{>0}$,
\item[(12)]$s_i\ob x_j=\ob x_js_i$, $j\neq i,i+1$,
\item[(13)]$e_i\ob x_j =\ob x_j e_i$, if $j\neq i,i+1$,
\item[(14)]$\ob x_i\ob x_j=\ob x_j\ob x_i$, for $1\leq i,j\leq r$,
\item[(15)]$s_i\ob x_i-\ob x_{i+1}s_i=e_i-1$,
\item[(16)]$\ob x_i s_i-s_i\ob  x_{i+1}=e_i-1 $,
\item [(17)]  $e_i(\ob x_i+\ob x_{i+1})=(\ob x_i+\ob x_{i+1})e_i=0$.
\end{enumerate}
\end{multicols} \end{Defn}
Let $\mathcal W_r^{\rm aff}$ be the $\kappa$-algebra obtained from $\bar{\mathcal W}_r^{\rm aff}$ by imposing the additional relation
in \eqref{admomega} with all $\omega_j$ replaced by $\hat \omega_j$ for $j\in\mathbb N$.
Given $\omega_0\in\kappa$, let $\mathcal W_r^{\rm aff}(\omega_0)$ be the algebra obtained from $\mathcal W_r^{\rm aff}$ by imposing the additional relation that $\hat\omega_0=\omega_0$. Then $\mathcal W_r^{\rm aff}(\omega_0)$ is the affine Wenzl algebra in  \cite{Na}.

\begin{Theorem}\label{affiso} We have a $\kappa$-algebra isomorphism
$\End_{\AB(\omega_0)}(\ob r) \cong \mathcal W_r^{\rm aff}(\omega_0)$.
\end{Theorem}
\begin{proof}
We define   the  algebra homomorphism $ \gamma:\mathcal W_r^{\rm aff}\rightarrow \End_{\AB}(\ob r) $ such that, for all admissible $i, j$ and $k$,
$\gamma(\hat\omega_k)=\Delta_k1_{\ob r}$, $\gamma(e_i)=E_i$, $\gamma(\ob x_j)= x_j$ and $\gamma(s_i)=S_i$,
where $x_j$ is defined in \eqref{defofxy} and
$$E_i=\begin{tikzpicture}[baseline = 25pt, scale=0.35, color=\clr]
\draw[-,thick](-2,1.1)to[out= down,in=up](-2,3.9);
         \draw(-1.5,1.1) node{$ \cdots$}; \draw(-1.5,3.9) node{$ \cdots$};
         \draw[-,thick](-0.5,1.1)to[out= down,in=up](-0.5,3.9);
        \draw[-,thick] (0,4) to[out=down,in=left] (0.5,3.5) to[out=right,in=down] (1,4);
         \draw[-,thick] (0,1) to[out=up,in=left] (0.5,1.5) to[out=right,in=up] (1,1);
         \draw(0,4.5)node{\tiny$i$}; \draw(1.3,4.5)node{\tiny$i+1$};
         \draw[-,thick](2,1.1)to[out= down,in=up](2,3.9);
         \draw(3.5,1.1) node{$ \cdots$}; \draw(3.5,3.9) node{$ \cdots$};
         \draw[-,thick](4.5,1.1)to[out= down,in=up](4.5,3.9);
           \end{tikzpicture}~ \text{ and }~
      S_i= \begin{tikzpicture}[baseline = 25pt, scale=0.35, color=\clr]
       \draw[-,thick](0,1.1)to[out= down,in=up](0,3.9);
       \draw(0.5,1.1) node{$ \cdots$}; \draw(0.5,3.9) node{$ \cdots$};
       \draw[-,thick](1.5,1.1)to[out= down,in=up](1.5,3.9);
       \draw[-,thick] (2,1) to[out=up, in=down] (3,4);
       \draw(2,4.5)node{\tiny$i$}; \draw(3.2,4.5)node{\tiny$i+1$};
        \draw[-,thick] (3,1) to[out=up, in=down] (2,4);
         \draw[-,thick](4,1.1)to[out= down,in=up](4,3.9);
         \draw(4.5,1.1) node{$ \cdots$}; \draw(4.5,3.9) node{$ \cdots$};
         \draw[-,thick](5.5,1.1)to[out= down,in=up](5.5,3.9);
           \end{tikzpicture}.$$
  To see that $\gamma$ is well-defined, it is enough to check that   $E_i, S_i$, $x_j$ and $\Delta_k1_{\ob r}$ satisfy  the defining relations in \eqref{admomega} and Definition~\ref {definition of Na} for all admissible $i, j$ and $k$.
Thanks to Lemma~\ref{admissible}, \eqref{admomega} is satisfied. By Lemma~\ref{reguequai},  $E_i, S_i,1\leq i\leq r-1$   satisfy Definition~\ref{definition of Na}(1)-(4) and (6)-(10). By \eqref{defofdelta}, $E_1x_1^kE_1=\Delta_k E_1, \forall k\in \mathbb N$. This verifies  Definition~\ref{definition of Na}(5),(11). By the definition of composition of morphisms in $\AB$, we see that
   $E_i, S_i, x_j$ satisfy    Definition~\ref{definition of Na}(12)-(14). Thanks to  \eqref{AOBC relations} and   Lemma~\ref{slides relations vertical}(a), $x_i, x_{i+1}, S_i,  E_i$ satisfy  Definition~\ref{definition of Na}(15)-(16).
  Finally, by Lemma~\ref{capre}, $E_i, x_i, x_{i+1}$ satisfy Definition~\ref{definition of Na}(17). Thus $ \gamma$ is a well-defined  algebra homomorphism.

  Suppose  $d\in\bar{\mathbb B}_{r,r}$ such that  $d$ has $t$ cups, $0\leq t\leq \lfloor r/2\rfloor$.
  Then $d\sim S_{w_1}\circ E_1\circ E_3\circ \cdots \circ E_{2t-1}\circ S_{w_2}$ for some $w_1,w_2\in \mathfrak S_r$,  where $S_w$ is given in Definition~\ref{definofs}.
So, any  $(r,r)$-Brauer diagram is equivalent  to a $(r,r)$-Brauer diagram  which is obtained by composing of  $E_i, S_i$, for $1\leq i\leq r-1$.
Thanks to Theorem~\ref{Cyclotomic basis conjecture}, the algebra $\End_{\AB}(\ob r)$ is generated by $\Delta_k1_{\ob r}$, $k\in 2\mathbb N$, $E_i$, $ x_j$ and $S_l$ for all admissible $i,j,l$. So, $\gamma$ is surjective.

 There is an obvious epimorphism $\pi:\End_{\AB}(\ob r)\rightarrow \End_{\AB(\omega_0)}(\ob r)$ such that $\pi$ sends $E_i$, $ x_j$ and $S_l$ to the elements with the same names for all admissible $i,j,l$ and
 \begin{equation}\label{eejjehhd}
 \pi(\Delta_01_{\ob r})=\omega_01_{\ob r}, ~\pi(\Delta_k1_{\ob r})=\Delta_k1_{\ob r}, ~ k\in 2\mathbb N, k>0.
 \end{equation}
  Let $\gamma_1:= \pi\circ \gamma$. Then $\gamma_1$ is surjective. By \eqref{eejjehhd}, $\gamma_1$ factors through  $\mathcal W_r^{\rm aff}(\omega_0)$.
  Let $$\gamma_2:\mathcal W_r^{\rm aff}(\omega_0)\rightarrow \End_{\AB(\omega_0)}(\ob r)$$ be the induced surjective homomorphism.

Let $B_r$ be the subalgebra of $\End_{\AB(\omega_0)}(\ob r)$ generated by $ \bar{\mathbb B} _{r,r}/\sim  $.  For any $d_1,d_2\in \bar{\mathbb B} _{r,r}/\sim $,   $d_1\circ d_2\in \bar{\mathbb B} _{r,r}/\sim $ whenever no bubble appears.  If  the composition $d_1\circ d_2$ contains $k$ bubbles,
then there is a $d_3\in \bar{\mathbb B} _{r,r}/\sim $ obtained from $d_1\circ d_2\in\mathbb B_{r,r}/\sim$ by removing the bubbles. Moreover,   $d_1\circ d_2 =\Delta_0^kd_3$ in $\AB$.
So, $d_1\circ d_2=\omega_0^kd_3$ in $\AB(\omega_0)$. Therefore, $B_r$ is spanned by $ \bar{\mathbb B} _{r,r}/\sim  $. Thanks  to Corollary~\ref{omega00s}(a),
$B_r$ has a basis given by $\bar{\mathbb B} _{r,r}/\sim $.
 Let $\mathcal W_r$ be the subalgebra of $\mathcal W_r^{\rm aff}(\omega_0)$ generated by $e_i, s_i, 1\le i\le r-1$. Then $\mathcal W_r$ is  isomorphic to the Brauer algebra over $\kappa$ with the defining parameter $\omega_0$~\cite{B}.
By \cite[Lemma~2.8]{LZ},  the restriction of $\gamma_2$ to ${\mathcal W_r}$ gives an isomorphism between  $\mathcal W_r$ and  $B_r$.   We identify  $\mathcal W_r$ and  $B_r$ under this isomorphism. Then  $\mathcal W_r$ has $\kappa$-basis given by $\bar{\mathbb B}_{r,r}/\sim$ and $ \gamma_2$ fixes each element in $\bar{\mathbb B}_{r,r}/\sim$.

 Recall that  any   $d\in  \bar{\mathbb {ND}}_{r,r}/\sim$    is of form   in \eqref{degdregu}.  If we use $\ob x_i$ (resp., the multiplication of $\mathcal W_r^{\rm aff}(\omega_0)$) to replace $x_i$ and $y_i$ (resp., each $\circ$) in \eqref{degdregu}, $1\le i\le r$,
we obtain an element, say $\tilde d\in \mathcal W_r^{\rm aff}(\omega_0)$.
So, $\gamma_2(\tilde d)=d$. Consider
\begin{equation}\label{regulainw}  \hat\omega_2^{k_2}\hat\omega_4^{k_4}\cdots \tilde d\in \mathcal W_r^{\rm aff}(\omega_0),\ \  k_2, k_{4},\ldots \in \mathbb N,  \end{equation} and call it  a regular  monomial in $  \mathcal W_r^{\rm aff}(\omega_0)$ (see \cite[(4.18)]{Na}). We remark that there are finitely many $\hat\omega_{2i}$'s in the product  in \eqref{regulainw}. By \cite[Lemmas~4.4,4.5]{Na},   $\mathcal W_r^{\rm aff}(\omega_0)$ is spanned by all the regular monomials in \eqref{regulainw}, where  $ d$ ranges over  $\bar{\mathbb {ND}}_{r,r}/\sim$.
  Moreover, $\gamma_2 ( \hat\omega_2^{k_2}\hat\omega_4^{k_4}\cdots \tilde d)=   \Delta_2^{k_2} \Delta_4^{k_4} \cdots d$.
  Thanks to Corollary~\ref{omega00s}(a), $\gamma_2$ is an isomorphism.
                      \end{proof}

 Suppose that $\omega$ is  admissible in the sense  \eqref{admomega}. Define   $\mathcal W_r^{\rm aff}(\omega)=\mathcal W_r^{\rm aff}/J$ where $J$ is the two-sided ideal generated by  $\hat \omega_j-\omega_j$, for $j\in\mathbb N$.
Then  $\mathcal W_r^{\rm aff}(\omega)$ is the affine Nazarov-Wenzl algebra in \cite{AMR}. The following result follows from   Theorem~\ref{affiso}, immediately.
 \begin{Cor} We have a $\kappa$-algebra isomorphism
$\End_{\AB(\omega)}(\ob r) \cong \mathcal W_r^{\rm aff}(\omega)$.
\end{Cor}


\section{Cyclotomic Brauer category and parabolic BGG category $\mathcal O$}\label{proofc}
In this section, we establish some relationships between  $\CB^f(\omega)$ and parabolic BGG category $\mathcal O$ associated with  $\mathfrak {so}_{2n}, \mathfrak {so}_{2n+1}$ and $\mathfrak {sp}_{2n}$. Recall that $\Pi$ is the set of simple roots  and  $R$ is the root system.
\begin{Defn}\label{defofpi}
Fix positive integers $q_1,q_2,\ldots,q_k$ such that $\sum_{j=1}^{k}q_j=n$.
Define
 $$I_1=\Pi\setminus \{\alpha_{p_1}, \alpha_{p_2}, \ldots, \alpha_{p_{k}} \} \text{ and $I_2= I_1\cup \{ \alpha_{n}\}$, }$$ where $ p_j=\sum_{l=1}^jq_l$,   $1\le j\le k$.
\end{Defn}

For each $I_i$, $i=1,2$, there is a root system $R_{I_i}:= R\cap \mathbb Z I_i$ with positive roots $R_{I_i}^+=R^+\cap \mathbb Z I_i$.  There are some subalgebras of $\mfg$ associated with the root system $R_{I_i} $:
 \begin{itemize}
\item $\mathfrak p_{I_i}$:  the standard  parabolic subalgebra of $\mathfrak g$  with respect to  $I_i$ in Definition~\ref{defofpi},
  \item   $\mathfrak l_{I_i}$: $\mathfrak h\oplus \sum _{\alpha\in R_{I_i}}\mfg_\alpha$, the Levi subalgebra of $\mathfrak p_{I_i}$,
  \item $\mathfrak u_{I_i}$: $ \sum _{\alpha\in R^+\setminus R^+_{I_i}}\mfg_\alpha$, the nilradical of $\mathfrak p_{I_i}$,
  \item $\mathfrak u_{I_i}^-$: $ \sum _{\alpha\in -R^+\setminus -R^+_{I_i}}\mfg_\alpha$, \end{itemize}
  where $\mfg_\alpha$ is the root space of $\mfg$ with weight $\alpha$.
Then $\mathfrak p_{I_i}=\mathfrak l_{I_i}\oplus \mathfrak u_{I_i}$ and
\begin{equation}\label{defoflevi}
\mathfrak l_{I_i}\cong \bigoplus_{j=1}^{k-1}\mathfrak{gl}_{q_j}\oplus \mathfrak g_i,
\end{equation}
where
\begin{equation}\label{defofgi}
\mathfrak g_i=\left\{
               \begin{array}{ll}
                 \mathfrak{gl}_{q_k}, & \hbox{if $i=1$;} \\
                 \mathfrak {so}_{2q_k+1}, & \hbox{if $i=2$ and $\mfg =\mathfrak{so}_{2n+1}$;}\\
 \mathfrak {so}_{2q_k}, & \hbox{if $i=2$ and $\mfg =\mathfrak{so}_{2n}$;}\\
              \mathfrak {sp}_{2q_k}, & \hbox{if $i=2$ and $\mfg =\mathfrak{sp}_{2n}$.}
\end{array}
             \right.
\end{equation}
\begin{rem}In order to avoid some overlaps, we assume that $q_k\geq 4$ (reps., $ q_k\geq 2$, $q_k\geq 3$) when we consider $I_2$ for $\mathfrak {so}_{2n}$ (resp., $\mathfrak {so}_{2n+1}$, $\mathfrak {sp}_{2n}$).
\end{rem}
Let $  \Lambda^{\mathfrak p_{I_i}}$  be the set of  all $\mathfrak p_{I_i}$-dominant integral weights. Then
 \begin{equation}\label{pdwt}  \Lambda^{\mathfrak p_{I_i}}=\{\lambda\in \mathfrak h^*\mid  \langle \lambda, \alpha\rangle \in \mathbb N, \forall \alpha\in I_i\}, i\in\{1,2\}.\end{equation}
For each $\lambda\in \Lambda^{\mathfrak p_{I_i}}$, the irreducible $\mathfrak l_{I_i}$-module $L_{I_i}(\lambda)$ with the highest weight $\lambda$ is  finite dimensional.
Recall that the BGG  category  $\mathcal O$ is  the category of finitely generated $\mathfrak {g}$-modules which are locally finite over $\mathfrak{n}^+$ and semi-simple over $\mathfrak h$.
Let $\mathcal{O}^{\mathfrak p_{I_i}}$ be  the full subcategory of $\mathcal O$  consisting of all $\mfg$-modules which are locally   $\mathfrak p_{I_i}$-finite.
Any  irreducible $\mathfrak l_{I_i}$-module $L_{I_i}(\lambda)$ can be considered as a $\mathfrak p_{I_i}$-module by inflation.  The  parabolic
Verma module with the highest weight $\lambda\in \Lambda^{\mathfrak p_{I_i}}$ is
$$M^{\mathfrak p_{I_i}}(\lambda):=\U(\mfg) \otimes_{\U(\mathfrak p_{I_i})} L_{I_i}(\lambda), ~i\in\{1,2\}.$$
\begin{Defn} \label{deltac}For $i=1,2$, define  $\lambda_{I_i, \mathbf c}=\sum_{j=1}^{k} c_j(\epsilon_{p_{j-1}+1}+\epsilon_{p_{j-1}+2}+\ldots+\epsilon_{p_j})$, where $p_0=0$, $p_j$'s are given in Definition~\ref{defofpi} and $\mathbf c=(c_1, c_2, \ldots, c_k)\in  \mathbb C^k$ such that
  $c_k=0$ if $i=2$. \end{Defn}
 Since $\langle \lambda_{I_i, \mathbf c}, \alpha\rangle=0, \forall \alpha\in I_i$,
$\lambda_{I_i, \mathbf c}\in \Lambda^{\mathfrak p_{I_i}}$, $i=1,2$.  Moreover,  $\text{dim}_{\mathbb C }L_{I_i}(\lambda_{I_i, \mathbf c})=1$.
In the remaining part of the paper, we denote by $m_i$   the highest weight vector of  $L_{I_i}(\lambda_{I_i, \mathbf c})$. It is unique up to a non-zero scalar.  Then
\begin{equation}\label{lmihighset}
L_{I_i}(\lambda_{I_i, \mathbf c})=\mathbb C\text{-span}\{m_i\}, ~ i\in\{1,2\}.
\end{equation}
If we allow $c_k\neq 0$ in the definition of $\lambda_{I_2, \mathbf c}$, the dimension of  simple $\mathfrak l_{I_2}$-module with the highest weight $\lambda_{I_2, \mathbf c}$ may not be  $1$.  This is the reason why
 we assume  $c_k=0$ in this case.

\begin{Defn} For $i\in \{1, 2\}$, define
 $\mathcal  B_{I_i}=\{F_{ -f,\pm g}\mid 1\le f<g \le n,\varepsilon_f\pm\varepsilon_g\in R^+\setminus R^+_{I_i}\}\cup T_i$,
where  $F_{f,g}$'s are given in \eqref{fij} and
\begin{itemize} \item $T_i=\emptyset$ if $\mathfrak{g}=\mathfrak {so}_{2n}$,
   \item $T_i=\{F_{0,f}\mid 1\le f\le n, \varepsilon_f\in R^+\setminus R^+_{I_i} \}$
if  $\mathfrak{g}=\mathfrak {so}_{2n+1}$,
\item $ T_i=\{F_{-f,f}\mid 1\le f\le n, 2\varepsilon_f\in R^+\setminus R^+_{I_i}\}$
if $\mathfrak{g}= \mathfrak {sp}_{2n}$.
 \end{itemize}
 \end{Defn}
 It is known that  $\mathcal  B_{I_i}$ is a basis of $\mathfrak u^-_{I_i}$.
For any positive integer $b$,  we denote  $(f_b, f_{b-1}, \ldots, f_1)$
(resp., $(g_b, g_{b-1}, \ldots, g_1)$) by $\mathbf f$ (resp., $\mathbf g$) if
 $F_{f_j,g_j}\in \mathcal B_{I_i}$, $1\le j\le b$.
  Write \begin{equation}\label{eijv} F_{{\mathbf f, \mathbf g}}^\alpha=F_{f_b,g_b}^{\alpha_b}F_{f_{b-1},g_{b-1}}^{\alpha_{b-1}}\cdots F_{f_1,g_1}^{\alpha_1}, \forall \alpha\in \mathbb N^b.\end{equation}
  If $b=0$, we set $ F_{{\mathbf f, \mathbf g}}^\alpha=1$.   Fix a total order  $\prec$  on  $\mathcal B_{I_i}$.
Let $\mathcal M_{I_i}$ be the set of all $F_{{\mathbf f, \mathbf g}}^\alpha$, $\alpha\in \mathbb N^b, b\in \mathbb N$ such that $F_{f_{j},g_{j}}\prec F_{f_{j+1}, g_{j+1}}$,  $1\le j\le b-1$.

 \begin{Lemma}\label{tsmodule}  For $i\in \{1, 2\}$ and $r\in \mathbb N$, define $M_{I_i, r}=M^{\mathfrak p_{I_i}}(\lambda_{I_i, \mathbf c})\otimes V^{\otimes r}$.
Then $M_{I_i, r}$ has  basis  $   \mathcal S_{i,r}=\{ym_i\otimes v_{\mathbf i}\mid y\in \mathcal M_{I_i},  \mathbf i\in \underline N^r\}$, where $v_{\mathbf i}=v_{i_1}\otimes v_{i_2}\otimes \ldots\otimes v_{i_r}$.\end{Lemma}

We say that    $F_{\mathbf k,\mathbf l}^{\alpha}m_i\otimes v_{\bf i}$
    is  of degree    $\sum_{j} {\alpha_j}$. This is
the degree of $F_{\mathbf k,\mathbf l}^\alpha $   defined via the  usual  $\mathbb Z$-grading on $\U(\mathfrak g)$.
For $j\in\mathbb N$, let $$M _{I_i,r}^{\leq j}=\mathbb C\text{-span} \{v\in \mathcal S_{i,r}\mid
\text{deg} v\leq j\}.$$ This gives a $\mathbb Z$-graded filtration of  $M_{I_i,r}$.
\begin{Lemma}\label{eM filtered de }
    Suppose that   $h,l \in \underline N$ and $j\in\mathbb N$.
     Then
    \begin{enumerate}
        \item[(1)] $F_{ h,  l}(M^{\mathfrak p_{I_i}}(\lambda_{I_i, \mathbf c})^{\leq j} )\subseteq M^{\mathfrak p_{I_i}}(\lambda_{I_i, \mathbf c})^{\leq j+1} $   if $ F_{h,l}\in \mathfrak u_{I_i}^-$,
        \item[(2)] $F_{ h,  l}(M^{\mathfrak p_{I_i}}(\lambda_{I_i, \mathbf c})^{\leq j} )\subseteq M^{\mathfrak p_{I_i}}(\lambda_{I_i, \mathbf c})^{\leq j} $   if $ F_{h,l}\notin \mathfrak u_{I_i}^-$.
            \end{enumerate}
\end{Lemma}
\begin{proof} If  $F_{h,l}\notin \mathfrak u_{I_i}^- $, then   $F_{h,l}m_i=b m_i$   for some scalar $b$. Otherwise, $F_{h,l}m_i$ is of degree $1$.
Suppose    $y\in \mathcal M_{I_i}$ and $\text{deg}(y)=j$.  By \eqref{comc} and induction on $j$,
\begin{equation}\label{usefact}
F_{h,l}ym_i=yF_{h,l} m_i ~~~ (\text{mod } M^{\mathfrak p_{I_i}}(\lambda_{I_i, \mathbf c})^{\leq j}).
\end{equation}
This proves  (2).
Suppose   $ F_{h,l}\in \mathfrak u_{I_i}^-$.
By \eqref{comc} and induction on $j$ again, $yF_{h,l}m_i$ is of degree  $j+1$ up to a linear combination of some elements in $M^{\mathfrak p_{I_i}}(\lambda_{I_i, \mathbf c})^{\leq j} $, proving (1).
\end{proof}
Thanks to \eqref{usefact}, we can move $F_{j,l}$ to the right until it meets $m_i$ when we compute the term of $F_{j,l}ym_i$ with degree bigger than   $\text{deg}(y)$, $y\in \mathcal M_{I_i}$.
This  fact will be used frequently in the proof of Corollary~\ref{usecorollary}.

For $1\le h\le k$, define
\begin{equation}\label{bigh} \mathbf p_{\leq h}=\bigcup_{j=0}^{h }\mathbf p_j\text{ and } \mathbf p_{>h}=\bigcup_{j=h+1}^k \mathbf p_j, \end{equation}
where $\mathbf p_0=\emptyset$  and $\mathbf p_j=\{p_{j-1}+1, p_{j-1}+2,\ldots, p_j\}$ for $1\leq j\leq k$. The following results can be verified by straightforward computation.

\begin{Lemma} \label{xh1} Fix an $i\in \{1, 2\}$ and  let  $m_i$ be the highest weight vector of   $M^{\mathfrak p_{I_i}} (\lambda_{I_i, \mathbf c})$ (see \eqref{lmihighset}).
 \begin{enumerate}\item[(1)]If $\mathfrak g=\mathfrak {so}_{2n+1}$, then  $\Omega (m_i\otimes  v_0)= \sum _{1\leq j\leq p_{k-1}}F_{0,j}m_i\otimes v_{j}+
\delta_{i,1}\sum _{j\in\mathbf  p_{k}}F_{0,j}m_i\otimes v_{j}$.
\item[(2)]  If $l\in \mathbf p_h$ and  $h\geq 1$, then $\Omega (m_i\otimes  v_l)=c_h m_i\otimes v_l+\sum_{j\in \mathbf p_{\leq h-1}}F_{l,j}m_i \otimes  v_j$,
\item[(3)] If $-l\in \mathbf p_h$ and  $h\geq 1$, then $\Omega (m_i\otimes  v_l)=-c_h m_i\otimes v_l+\sum _{j\in \mathbf p_{>h}} F_{l,-j}m_i\otimes v_{-j}+A+B+C$,
where
\begin{itemize}
 \item  $A=\sum_{j\in\mathbf p_{\leq k}\setminus\{ -l\},\{-l,j\}\nsubseteq \mathbf p_k}F_{l,j}m_i\otimes v_j$ $+ \delta_{i,1}\sum _{j\in\mathbf p_{\leq k}\setminus\{ -l\},\{-l,j\}\subset \mathbf  p_{k}}F_{l,j}m_i\otimes v_j$,
\item $B=0$ (resp., $\delta_{\mathfrak g,\mathfrak {sp}_{2n}} F_{l,-l}m_i\otimes v_{-l}$) if $-l\in \mathbf p_k$ and $i=2$ (resp., otherwise),
 \item $ C=0$ (resp., $-\delta_{\mathfrak g,\mathfrak {so}_{2n+1}} F_{0,-l}m_i\otimes v_{0}$) if
$-l\in \mathbf p_k$ and $i=2$ (resp.,  otherwise).
  \end{itemize}
\end{enumerate}
\end{Lemma}
\begin{proof}
Since $m_i$ is a highest weight vector, $\mathfrak n^+m_i=0$. Thus $\Omega$ acts on $m_i\otimes v_l$ via $\sum_{i=1}^4 \gamma_i$, where
$\gamma_1=\sum_{1\leq j\leq n}F_{j,j}\otimes F_{j,j}$, $\gamma_2= \sum_{1\leq r<j\leq n} F_{-r,\pm j}\otimes F_{\pm j,-r}$, $\gamma_3=
\frac{1}{2}\delta_{\mathfrak g,\mathfrak {sp}_{2n}}\sum_{1\leq j\leq n}F_{-j,j}\otimes F_{j,-j}$ and $\gamma_4= \delta_{\mathfrak g,\mathfrak {so}_{2n+1}}\sum_{1\leq j\leq n} F_{0,j}\otimes F_{j,0}$. We have
\begin{itemize}\item $ \gamma_j m_i\otimes v_0=0$, if $j\in \{1,2\}$,
 \item $\gamma_1m_i\otimes v_l=\pm c_{h}m_i\otimes v_{l}$, if $\pm l\in \mathbf p_h$.\end{itemize}
Since  $F_{r,j}m_i=0$ for any  $F_{r,j}\in \mathfrak p_{I_i}\cap \mathfrak n^-$,  we have
\begin{itemize}\item $ \gamma_2  (m_i\otimes v_l)=  \sum _{j\in \mathbf p_{\leq h-1}} F_{l,j}m_i\otimes v_j$, if $l\in \mathbf p_h$,
\item  $\gamma_2  (m_i\otimes v_l)= \sum _{j\in \mathbf p_{>h}} F_{l,-j}m_i\otimes v_{-j}+A$, if $-l\in \mathbf p_h$,
 \item $\gamma_3 (m_i\otimes v_l) =    0$  if either $l\in \mathbf p_h$ or $-l\in \mathbf p_k$ and $i=2$,
 \item $\gamma_3 (m_i\otimes v_l)=\delta_{\mathfrak g,\mathfrak {sp}_{2n}} F_{l,-l}m_i\otimes v_{-l}$, otherwise,
 \item $\gamma_4(m_i\otimes v_0)=\sum _{1\leq j\leq p_{k-1}}F_{0,j}m_i\otimes v_{j}+
\delta_{i,1}\sum _{j\in\mathbf  p_{k}}F_{0,j}m_i\otimes v_{j}$, if $\mathfrak g=\mathfrak {so}_{2n+1}$,
\item $\gamma_4(m_i\otimes v_l)=0$ if either $l\in \mathbf p_h$ or $-l\in \mathbf p_k$ and $i=2$,
\item $\gamma_4(m_i\otimes v_l)=-\delta_{\mathfrak g,\mathfrak {s0}_{2n+1}} F_{0,-l}m_i\otimes v_{0}$, otherwise.
\end{itemize}
Combining the above equations yields (1)--(3).
\end{proof}

We also denote by $V$ the natural $\mathfrak{gl}_N$-module  with basis $\{v_j\mid 1\leq j\leq N\}$. Let  $V^*$ be the linear dual of $V$.
Let  $L(\lambda)$ be the irreducible $\mfg$-module with the highest weight $\lambda$.
\begin{Lemma}\label{tensorwithv} \cite[VIII, \S9~Prop.2]{Bou}
Suppose that  $\lambda$ is a dominant integral weight.
 Then
$$\begin{aligned}
L(\lambda)\otimes V&\cong\left\{
                           \begin{array}{ll}
                             \oplus_j L(\lambda+\varepsilon_j), & \hbox{if $\mfg=\mathfrak{gl}_N$,} \\
                             \oplus_j L(\lambda\pm\varepsilon_j ), & \hbox{if $\mfg\in\{\mathfrak{sp}_{2n},\mathfrak{so}_{2n}\}$ or $\mfg=\mathfrak{so}_{2n+1}, \lambda_n=0$}\\
                             L(\lambda)\oplus(\oplus_j L(\lambda\pm\varepsilon_j )), & \hbox{if $\mfg=\mathfrak{so}_{2n+1}, \lambda_n>0$}
                           \end{array}
                         \right. \\
                          L(\lambda)\otimes V^* &\cong\quad \oplus_j L(\lambda-\varepsilon_j) \quad\quad\quad\quad\quad\text{ if $\mfg=\mathfrak{gl}_N$,}
\end{aligned}$$
where the direct sum is over all $j$ such that    $\lambda\pm\varepsilon_j$ is dominant integral.
\end{Lemma}

  Suppose  that $V$ is the natural $\mfg$-module, where $\mfg\in \{\mathfrak{sp}_{N},\mathfrak{so}_{N}\}$. Recall   $\mathbf p_j$'s in \eqref{bigh}.   We consider   subspaces $\tilde V_j\subset V$  spanned by $  V_j$, with $1\leq j\leq 2k$ (resp., $1\leq j\leq 2k-1$) if $i=1$(resp., $i=2$), where
\begin{equation}\label{eevlll}
  V_j=\left\{
      \begin{array}{ll}
        &\{v_l\mid l\in\mathbf p_j \},\quad\quad \quad\quad\hbox{  if either $1\leq j\leq k$ and $i=1$ or $1\leq j\leq k-1$ and $i=2$;} \\
       &\{v_l\mid -l\in \mathbf p_{2k-j+1} \}, \quad  \hbox{  if $k+1\leq j\leq 2k$ and $i=1$;} \\
  &\{v_l\mid -l\in \mathbf p_{2k-j}\},\quad \quad \hbox{ if $k+1\leq j\leq 2k-1$ and $i=2$;} \\
 &\{v_l, \delta_{\mfg, \mathfrak{so}_{2n+1}}v_0\mid \pm l\in\mathbf p_k\},   \hbox{ if $j=k$ and $i=2$.} \\
      \end{array}
    \right.
 \end{equation}
Then $\tilde V_j$ is isomorphic to the natural $\mathfrak {gl}_{q_j}$-module  (resp., $\mathfrak g_i$-module) if $ 1\leq j\leq k-1$ (resp., $j=k$), where $\mathfrak g_i$ is given in \eqref{defofgi}.
If $k+1\leq j\leq 2k$ and $i=1$ (resp., $k+1\leq j\leq 2k-1$ and $i=2$), then $\tilde V_j$ is isomorphic to the dual of the natural $\mathfrak {gl}_{q_{2k-j+1}}$-module  (resp., $\mathfrak {gl}_{q_{2k-j}}$-module). Thanks to  Lemma~\ref{tensorwithv}, we have the  following result, immediately.
\begin{Cor}\label{teroofvl}
As $\mathfrak l_{I_i}$-modules,
\begin{equation*}
L_{I_i}(\lambda_{I_i, \mathbf c})\otimes V\cong \left\{
                                                 \begin{array}{ll}
                                                  \bigoplus_{j=1}^{k} L_{I_i}(\lambda_{I_i, \mathbf c}+ \varepsilon_{p_{j-1}+1})\oplus L_{I_i}(\lambda_{I_i, \mathbf c}- \varepsilon_{p_{j}})\oplus\delta_{\mfg,\mathfrak {so}_{2n+1}} L_{I_i}(\lambda_{I_i, \mathbf c})
 , & \hbox{ if $i=1$;} \\
   \bigoplus_{j=1}^{k-1} L_{I_i}(\lambda_{I_i, \mathbf c}+ \varepsilon_{p_{j-1}+1})\oplus L_{I_i}(\lambda_{I_i, \mathbf c}- \varepsilon_{p_{j}})\oplus L_{I_i}(\lambda_{I_i, \mathbf c}+ \varepsilon_{p_{k-1}+1})                                                , & \hbox{ if $i=2$.}
                                                 \end{array}
                                               \right.
\end{equation*}
Moreover,
\begin{itemize}
\item[(1)] $ L_{I_i}(\lambda_{I_i, \mathbf c}+ \varepsilon_{p_{j-1}+1})$ is spanned by $\{m_i\otimes v\mid v\in V_j\}$ if either $ 1\leq j\leq k$ and $i=1$ or $ 1\leq j\leq k-1$ and $i=2$;
\item [(2)] $L_{I_i}(\lambda_{I_i, \mathbf c}- \varepsilon_{p_{j}})$ is spanned by \begin{enumerate}\item  $\{m_i\otimes v\mid v\in V_{2k-j+1}\}$ if $ 1\leq j\leq k$ and $i=1$,
\item  $\{m_i\otimes v\mid v\in V_{2k-j}\}$ if  $ 1\leq j\leq k-1$ and $i=2$;\end{enumerate}
\item [(3)] $L_{I_i}(\lambda_{I_i, \mathbf c})$ is spanned by $\{m_i\otimes v_0\}$ if $i=1$ and $ \mfg=\mathfrak{so}_{2n+1}$;
\item [(4)] $L_{I_i}(\lambda_{I_i, \mathbf c}+ \varepsilon_{p_{k-1}+1})$ is spanned by $\{m_i\otimes v\mid v\in V_k\}$ if $i=2$.
\end{itemize}
\end{Cor}

  \begin{Defn}\label{polf12}  Let  $\mathbf c$ be  given in Definition~\ref{deltac}. \begin{enumerate} \item
   If $\mfg\in \{\mathfrak {sp}_{2n}, \mathfrak{so}_{2n}\}$, we  define  $ f_1(t)=f_{2}(t) (t-u_{k+1})$ and  $f_2(t)=\prod_{j\in J} (t-u_{j})$, where  $J=\{1, 2, \ldots, 2k\}\setminus \{ k+1\}$
 and
 \begin{equation} \label{uij1} u_j=\begin{cases} \epsilon_\mathfrak g (c_j-p_{j-1}+\frac{1}{2} (2n-\epsilon_{\mathfrak g})), &\text{if $1\le j\le k$},\\
    \epsilon_{\mathfrak g}(-c_{2k-j+1}+p_{2k-j+1}- \frac 12(2n-\epsilon_{\mathfrak g})),  & \text{$k+1\leq j\leq 2k$.}\\ \end{cases} \end{equation}

  \item If  $\mfg=\mathfrak{so}_{2n+1}$,  define $f_1(t)=f_2(t) (t-u_{k+1}) (t-u_{k+2})$, and $f_2(t)=\prod_{j\in J} (t-u_j)$, where $J=\{1,, 2, \ldots, 2k+1\}\setminus   \{k+1, k+2\}$ and
 \begin{equation}\label{uij2} u_j=\begin{cases} c_j-p_{j-1}+n, &\text{if $1\leq j\leq k$,}\\
  0,  &\text{if $j=k+1$,}\\
    -c_{2k-j+2}+p_{2k-j+2}-n, &\text{if $k+2\leq j\leq 2k+1$.}\\
    \end{cases}
    \end{equation}\end{enumerate}
  \end{Defn}

Recall the functor $\Psi: \AB\rightarrow \END(\U(\mfg)\text{-mod})$ in Theorem~\ref{affaction}.
Let $\Psi_M:\AB\to \U(\mfg)\text{-mod}$ be the functor obtained by  the composition of $\Psi$ followed by evaluation at $M\in \U(\mfg)\text{-mod}$.
By Theorem~\ref{affaction}, for any $b\in \Hom_{\AB}(\ob m,\ob s)$ and $b_1\in \Hom_{\AB}(\ob h,\ob t)$, $m,s,h,t\in \mathbb N$,  we have
\begin{equation}\label{ddnamphis}
\Psi_M(b\otimes b_1)=\Psi(b_1)_{M\otimes V^{\otimes s}}\circ (\Psi(b)_M\otimes \text{Id}_{V^{\otimes h}}).
\end{equation}
Thanks to \cite[Proposition~9.3]{Hum},  $\mathcal{O}^{\mathfrak p_{I_i}}$ is closed under  tensoring with finite dimensional $\U(\mfg)$-modules.
So, we have  a functor
\begin{equation}\label{fucotssjd}
 \Psi_{M^{\mathfrak p_{I_i}}(\lambda_{I_i, \mathbf c})}:\AB\to \mathcal O^{\mathfrak p_{I_i}}.
\end{equation}
To simplify the notation, we denote $\Psi_{M^{\mathfrak p_{I_i}}(\lambda_{I_i, \mathbf c})}(X)$  by $X$  in the following two results.
\begin{Lemma}\label{polyofx} Suppose  $\mathfrak g\in \{\mathfrak {so}_{2n}, \mathfrak {sp}_{2n}\}$ and $i\in\{1,2\}$.
Then    $M^{\mathfrak p_{I_i}}(\lambda_{I_i, \mathbf c})\otimes V$ has    a parabolic Verma flag
$0=N_0\subseteq N_1\subseteq  \ldots\subseteq N_{2k}=M^{\mathfrak p_{I_i}}(\lambda_{I_i, \mathbf c})\otimes V$
   such that
 \begin{equation} \label{pbvf1} N_j/N_{j-1}\cong \begin{cases}  M^{\mathfrak p_{I_i}}(\lambda_{I_i, \mathbf c}+\varepsilon_{p_{j-1}+1}), & \text{if $1\leq j\leq k$,} \\
 M^{\mathfrak p_i}(\lambda_{I_i, \mathbf c}-\varepsilon_{p_{k}}), & \text{if $i=1$ and $j=k+1$,} \\
 0, & \text{if $i=2$ and $j=k+1$,} \\
                               M^{\mathfrak p_{I_i}}(\lambda_{I_i, \mathbf c}-\varepsilon_{p_{2k+1-j}}) , & \text{if $k+2\leq j\leq 2k$.}
                               \end{cases}\end{equation}
 Moreover, $X$ preserves this filtration and  $f_i(X)$
 acts  on  $M^{\mathfrak p_{I_i}}(\lambda_{I_i, \mathbf c})\otimes V$  trivially,   where $f_i(t)$ is given in Definition~\ref{polf12}(a).
   \end{Lemma}
\begin{proof}
By the tensor identity,
 \begin{equation}\label{tensoridentity}
M^{\mathfrak p_{I_i}}(\lambda_{I_i, \mathbf c})\otimes V\cong\U(\mfg)  \otimes_{\U(\mathfrak p_{I_i})} (L_{I_i}(\lambda_{I_i, \mathbf c})\otimes V).
\end{equation}
 Define  $N_j$ to be     the left $\U(\mfg)$-module  generated by $\tilde N_j$, where
\begin{itemize}
\item $\tilde N_j:=\{m_i\otimes v\mid v\in \oplus _{l=1}^j V_l\}$ if either  $i=1$ or $i=2$ and $1\leq j\leq k$,
\item $ \tilde N_j:=\{m_i\otimes v\mid v\in \oplus _{l=1}^k V_l \} $  if $i=2$ and  $j=k+1$,
\item $\tilde N_j:=\{m_i\otimes v\mid v\in \oplus _{l=1}^{j-1} V_{l}\}$  if  $i=2$ and $k+2\leq j\leq 2k$,
\end{itemize}
and  $V_l$'s are given in \eqref{eevlll}.
Thanks to Corollary~\ref{teroofvl},   $0=N_0\subseteq N_1\subseteq  \ldots\subseteq N_{2k}=M^{\mathfrak p_{I_i}}(\lambda_{I_i, \mathbf c})\otimes V$ is the  required parabolic Verma flag.
By Lemma~\ref{xh1}, $\Omega(m_i\otimes v)\in N_j$ if $m_i\otimes v\in \tilde N_j$ for all admissible $j$.
So, the action of $\Omega$ preserves this filtration. So does $X$.

It is well known that the Casimir  $C$ acts  on a highest weight module with the highest weight $\lambda$ via the scalar   $(\lambda,\lambda+2\rho)$ (see e.g. \cite[Lemma~8.5.3]{Mu}), where $\rho$ is given in \eqref{rho123}.
Since $\Omega=\frac {1}{2} ( \Delta(C)-C\otimes 1-1\otimes C)$,   $\Omega$ acts on
$ M^{\mathfrak {p}_{I_i}}(\lambda)$ via  the scalar
\begin{equation}\label{actionfomega}
\frac{1}{2}((\lambda,\lambda+2\rho) -(\lambda_{I_i, \mathbf c},\lambda_{I_i, \mathbf c}+2\rho)-(\varepsilon_1,\varepsilon_1+2\rho)).
\end{equation}
 By \eqref{actionfomega} and straightforward computation,    $X$ acts on  $N_j/N_{j-1}$ via $u_j$ if $i=1$, where $u_j$'s are given in \eqref{uij1}.
Hence, $f_1(X)$ acts on $M^{\mathfrak p_{I_1}}(\lambda_{I_1, \mathbf c})\otimes V$  trivially. Finally, one can verify that  $f_2(X)$  acts trivially on $M^{\mathfrak p_{I_2}}(\lambda_{I_2, \mathbf c})\otimes V$  in a similar way.
\end{proof}
\begin{Lemma}\label{polyofxodd}  Suppose  $\mathfrak g=\mathfrak {so}_{2n+1}$ and $i\in\{1,2\}$. Then
  $M^{\mathfrak p_{I_i}}(\lambda_{I_i, \mathbf c})\otimes V$   has a   parabolic Verma flag
$0=N_0\subseteq N_1\subseteq \ldots\subseteq N_{2k+1}=M^{\mathfrak p_{I_i}}(\lambda_{I_i, \mathbf c})\otimes V$ such that
$$N_j/N_{j-1}\cong \begin{cases}
                                M^{\mathfrak p_{I_i}}(\lambda_{I_i, \mathbf c}+\varepsilon_{p_{j-1}+1}), & \text{if $1\leq j\leq k$,} \\
                                M^{\mathfrak p_{I_i}}(\lambda_{I_i, c}), & \text{if $i=1$ and $j= k+1$,}\\
                                 M^{\mathfrak p_{I_i}}(\lambda_{I_i, c}-\varepsilon_{p_{k}}),  & \text{if $i=1$ and $j= k+2$,}\\
                                 0,  & \text{if $i=2$ and $j=k+1, k+2 $,} \\
                               M^{\mathfrak p_{I_i}}(\lambda_{I_i, \mathbf c}-\varepsilon_{p_{2k+2-j}}), & \text{if $k+3\leq j\leq 2k+1$.}
                               \\
                               \end{cases}$$      Moreover, $X$ preserves this filtration and  $f_i(X)$
 acts  on  $M^{\mathfrak p_{I_i}}(\lambda_{I_i, \mathbf c})\otimes V$  trivially,   where $f_i(t)$ is given in Definition~\ref{polf12}(b).

\end{Lemma}

\begin{proof}  Recall  $V_l$'s   in \eqref{eevlll}. Define   $N_j$ to be     the left $\U(\mfg)$-module  generated by
\begin{itemize}
\item $\{m_i\otimes v\mid v\in \oplus _{l=1}^j V_l\}$ if  $1\leq j\leq k$,
\item $\{m_i\otimes v\mid v\in \oplus _{l=1}^k V_l\}$ if  $i=2$ and $k+1\leq j\leq k+2$,
\item $ \{m_i\otimes v\mid v\in \oplus _{l=1}^{j-2} V_l  \} $ if $i=2$ and  $k+3\leq j\leq 2k+1$,
\item $\{m_i\otimes v\mid v\in \oplus _{l=1}^k V_l\cup\{v_0\}\}$ if $i=1$ and $j=k+1$,
\item $ \{m_i\otimes v\mid v\in \oplus _{l=1}^{j-1} V_l\cup\{v_0\} \} $  if $i=1$ and $k+2\leq j\leq 2k+1$.
\end{itemize}
By \eqref{tensoridentity} and Corollary~\ref{teroofvl},   $0=N_0\subseteq N_1\subseteq \ldots\subseteq N_{2k+1}=M^{\mathfrak p_{I_i}}(\lambda_{I_i, \mathbf c})\otimes V$  is the required parabolic Verma flag.
One can check the remaining results by arguments similar to those in the proof of Lemma~\ref{polyofx}.
\end{proof}

%


 It is well known that  the natural transformations between the identity functor and itself  in $\U(\mfg)\text{-mod}$ forms  an algebra which can be identified with   the center $Z(\U(\mathfrak g) )$ of $\U(\mathfrak g )$.  In particular, $\Psi_{\U(\mathfrak g)}(\Delta_j)(1)\in Z(\U(\mathfrak g) )$, where $\Delta_j$ is given in \eqref{defofdelta}.
 Define
 \begin{equation} \label{defogzk}
 z_j:=\varepsilon_{\mathfrak g}^{j+1}(\text{Id}\otimes \text{tr}_V)((\frac{1}{2}(N-\epsilon_\mfg)+\Omega)^j), \text{ for all }j\in \mathbb N.
 \end{equation}
 Suppose $u$ is an indeterminate. In   \cite[page~646]{DRV}, Daugherty, Ram and Virk proved that \begin{equation}\label{zkco}
\sum_{j\in\mathbb Z_{\geq0}} \varepsilon_{\mathfrak g}^{j+1}z_ju^{-j}=u\sum_{j\in\mathbb Z_{\geq0}} \frac{(\text{Id}\otimes \text{tr}_V)(\Omega^j) }{ (u-\frac{1}{2}(N-\epsilon_\mfg))^{j+1}}.
\end{equation}

\begin{Lemma}\label{zk}   For all $j\in \mathbb N$, $z_j=\Psi_{\U(\mathfrak g)}(\Delta_j)(1)$.
  \end{Lemma}
\begin{proof} Suppose $\phi\in\End(V)$ such that  $ \phi(v_l)=\sum_{j\in \underline N}a_{l,j}v_j$. Recall the monoidal  functor $\Phi$ in Proposition~\ref{level1s}.
Then
\begin{equation}\label{phiii}
\Phi(A)\circ (\phi\otimes \text{Id}_V)\circ \Phi(U)(1)=\Phi(A)(\sum_{l\in \underline N}\phi(v_l)\otimes v_l^*)= \varepsilon_{\mathfrak g}\sum_{l\in \underline N}a_{l,l}= \varepsilon_{\mathfrak g}\text{tr}_V(\phi).\end{equation}
For any  $y\in \U(\mfg)$, we have the  $\mathbb C$-linear map in $\End( \U(\mfg)) $  sending  $g$ to $yg $, for any $g\in \U(\mfg)$.
Denote this map by $y$. Then
\begin{equation}\label{traceddd}
\begin{aligned}
\Psi_{\U(\mathfrak g)}( A)\circ (y\otimes \phi\otimes \text{Id}_V)\circ \Psi_{\U(\mathfrak g)}(U)(1)\overset{ \eqref{act111} } =&(\text{Id}\otimes \Phi(A))\circ (y\otimes \phi\otimes \text{Id}_V)\circ (\text{Id}\otimes \Phi(U))(1)\\
=~~~\ &y\otimes (\Phi(A)\circ (\phi\otimes \text{Id}_V)\circ \Phi(U)(1))\\
 \overset{\eqref{phiii}}=&\varepsilon_{\mathfrak g}\text{tr}_V(\phi) y=\varepsilon_{\mathfrak g}(\text{Id}\otimes \text{tr}_V)(y\otimes \phi).
\end{aligned}
\end{equation}
    By  \eqref{defofdelta}, we have
$$\begin{aligned}
\Psi_{\U(\mathfrak g)}(\Delta_j)(1)&=\Psi_{\U(\mathfrak g)}( A)\circ (\Psi_{\U(\mathfrak g)}(X^j)\otimes \text{Id}_V)\circ \Psi_{\U(\mathfrak g)}( U)(1)\\
&\overset{\eqref{actionofxxx}}=\varepsilon_{\mathfrak g}^j \Psi_{\U(\mathfrak g)}( A)\circ ((\frac{1}{2}(N-\epsilon_\mfg)+\Omega)^j\otimes \text{Id}_V)\circ \Psi_{\U(\mathfrak g)}( U)(1)\\
&\overset{\eqref{traceddd}}=\varepsilon_{\mathfrak g}^{j+1}(\text{Id}\otimes
\text{tr}_V)((\frac{1}{2}(N-\epsilon_\mfg)+\Omega)^j).\\
\end{aligned}$$
 \end{proof}

\begin{Lemma}\label{ghom123} $\Psi_{ M^{\mathfrak p_{I_i}} (\lambda_{I_i, \mathbf c})}(\Delta_j)=\omega_j\text{Id}_{M^{\mathfrak p_{I_i}}(\lambda_{I_i, \mathbf c})}$
 for all  $j\in \mathbb N$, where   $\omega_j$'s are some scalars in  $ \mathbb C$. Further,
 \begin{equation}\label{omegaa}
  u-\frac{1}{2}+\sum_{j=0}^{\infty}\frac{\omega_j}{u^{j}}=(u-\frac{1}{2}(-1)^a)\prod_{j\in J_i} \frac {u+u_j}{u- u_j },
\end{equation}
 where $a={\rm deg}f_i(t)$ (see Definition~\ref{polf12}),  $u$ is an indeterminate and \begin{enumerate}\item
  $J_1=\{1, 2, \ldots, 2k\}$   and $J_2=J_1\setminus \{k+1\}$ if   $\mathfrak g=\mathfrak {so}_{2n}, \mathfrak {sp}_{2n}$, \item
 $J_1=\{1, 2, \ldots, 2k+1\}$ and $J_2=J_1\setminus \{k+1, k+2\}$ if  $\mathfrak g=\mathfrak {so}_{2n+1}$.\end{enumerate}
\end{Lemma}
\begin{proof} By Lemma~\ref{zk},  $z_j$ is central. Thus,
 $ z_j|_{M^{\mathfrak p_{I_i}} (\lambda_{I_i, \mathbf c})}=\omega_j \text{Id}_{M^{\mathfrak p_{I_i}}(\lambda_{I_i, \mathbf c})}$ for some  $ \omega_j\in \mathbb C$.
Let $$\phi_{M^{\mathfrak p_{I_i}} (\lambda_{I_i, \mathbf c})}: \U(\mfg)\rightarrow M^{\mathfrak p_{I_i}} (\lambda_{I_i, \mathbf c})$$ be the $\U(\mfg)$-module homomorphism such that
$\phi_{M^{\mathfrak p_{I_i}} (\lambda_{I_i, \mathbf c})}(1)=m_i$.
Then
$$
\Psi_{ M^{\mathfrak p_{I_i}} (\lambda_{I_i, \mathbf c})}(\Delta_j) \circ \phi_{M^{\mathfrak p_{I_i}} (\lambda_{I_i, \mathbf c})}=\phi_{M^{\mathfrak p_{I_i}} (\lambda_{I_i, \mathbf c})}\circ  \Psi_{\U(\mfg)}(\Delta_j), $$  for all  $j\in\mathbb N$.
So,
\begin{equation}\label{wewww}
\Psi_{ M^{\mathfrak p_{I_i}} (\lambda_{I_i, \mathbf c})}(\Delta_j)(m_i)=\Psi_{\U(\mfg)}(\Delta_j)(1)(m_i)\overset{Lemma~\ref{zk}}=\omega_jm_i,
\end{equation}
 and   $\Psi_{ M^{\mathfrak p_{I_i}} (\lambda_{I_i, \mathbf c})}(\Delta_j)=\omega_j\text{Id}_{M^{\mathfrak p_{I_i}}(\lambda_{I_i, \mathbf c})}$, for $j\in\mathbb N$.

 Let $F$ be the $N\times N$ matrix such that the $(l,j)$th entry of $F$  is  $F_{l,j}$ in  \eqref{fij}.
By \cite[Corollary~7.1.4]{Mol},  the \textit{Gelfand invariant} ${\rm {tr}}(F^j)$ is  central in $\U(\mfg)$ for any $j\in \mathbb N$.
By \cite[(4.3)]{DRV},
\begin{equation}\label{tesfij}
(\text{Id}\otimes \text{tr}_V)(\Omega^j)=(-1)^j{\rm{tr}}F^j, j\in \mathbb N.
\end{equation}
 Let $\chi:Z(\U(\mfg))\rightarrow \U(\mathfrak h)$ be the Harish-Chandra homomorphism, where
 $\ell_j=h_j+\rho(h_j)$, $1\le  j\le n$ and $\rho$ is given in \eqref{rho123}.  Set $\ell_{-j}=-\ell_j$ and $\ell_0=0$. It follows from \cite[Theorem~4.1]{DRV} (see also \cite[Corollary~7.1.8]{Mol}) that
\begin{equation}\label{anothere}
1+\frac{u+\frac{1}{2}}{u+\frac{1}{2}-\frac{1}{2}\epsilon_{\mathfrak g} }\sum_{j=0}^\infty \frac{(-1)^j\chi({\rm{tr}}F^j)}{(u+  \frac{1}{2}-\frac{1}{2}(N-\epsilon_\mfg) )^{j+1}}=\prod_{j=-n}^n \frac{u+\ell_j+1}{u+\ell_j},
\end{equation}
 and the index $0$ is skipped in the product if $\mathfrak g\in \{\mathfrak {sp}_{2n}, \mathfrak {so}_{2n}\}$. Replacing $u+\frac12$ by $u$ in \eqref{anothere} yields
\begin{equation}\label{ser}1+\frac{u}{u-\frac{1}{2}\epsilon_{\mathfrak g} }\sum_{j=0}^\infty \frac{(-1)^j\chi({\rm{tr}}F^j)}{(u-\frac{1}{2}(N-\epsilon_\mfg))^{j+1}}=\prod_{j=-n}^n \frac{u+\ell_j+\frac{1}{2}}{u+\ell_j-\frac{1}{2}}.
\end{equation}
By \eqref{zkco} and \eqref{tesfij},
\begin{equation}\label{ser1}
\sum_{j\geq0} \epsilon_{\mathfrak g}^{j+1}\omega_ju^{-j}\text{Id}_{M^{\mathfrak p_{I_i}}(\lambda_{I_i, \mathbf c})}=u\sum_{j\geq 0}\frac{(-1)^j{\rm{tr}}F^j}{(u-\frac{1}{2}(N-\epsilon_\mfg) )^{j+1}}|_{M^{\mathfrak p_{I_i}}(\lambda_{I_i, \mathbf c})}.
\end{equation}
Note that $\text{tr}F^j$ acts on $M^{\mathfrak p_{I_i}}(\lambda_{I_i, \mathbf c})$ via a scalar,  and this  scalar can be computed  by  considering  the action of $\chi({\rm{tr}}F^j)$ on the highest weight vector of $M^{\mathfrak p_{I_i}}(\lambda_{I_i, \mathbf c})$.
By \eqref{ser}--\eqref{ser1},
\begin{equation} \label{comofw}
\sum_{j\geq0}\omega_j(\epsilon_{\mathfrak g}u)^{-j}+\epsilon_{\mathfrak g}u-\frac{1}{2}   =(\epsilon_{\mathfrak g}u-\frac{1}{2})\prod_{j=-n}^n \frac{u+\lambda_{I_i,\mathbf c}(\ell_j)+\frac{1}{2}}{u+\lambda_{I_i,\mathbf c}(\ell_j)-\frac{1}{2}}= (\epsilon_{\mathfrak g}u-\frac{1}{2})  \frac{\prod_{j=1}^b(u+\epsilon_{\mathfrak g}u_j) }{\prod_{j=1}^b(u-\epsilon_{\mathfrak g}u_j)} A_{\mathfrak g},
\end{equation}
where  $(A_{\mfg}, b) =( \frac{u+\frac{1}{2}}{u-\frac{1}{2}}, 2k+1)$ if $\mathfrak g=\mathfrak {so}_{2n+1}$, and $(A_{\mfg}, b)= (1, 2k)$,  otherwise.
Since  $a={\rm deg}f_i(t)$, we have
\begin{equation}\label{defofd}
a=\begin{cases}
      2k, & \text{if $ i=1$, $\mathfrak g\in \{\mathfrak {so}_{2n}, \mathfrak {sp}_{2n}\}$} \\
      2k+1, & \text{if $i=1$, $\mathfrak g=\mathfrak {so}_{2n+1}$} \\
2k-1, & \text{if  $ i=2$, $\mathfrak g\in \{\mathfrak {so}_{2n}, \mathfrak {sp}_{2n}, \mathfrak {so}_{2n+1}\}$.}
    \end{cases}
\end{equation}
Thanks to  Definition~\ref{polf12},
\begin{equation}\label{ukplus}
u_{k+1}=\left\{
          \begin{array}{ll}
            0, & \hbox{if $\mathfrak g=\mathfrak {so}_{2n+1}$;} \\
            \frac{1}{2}, & \hbox{ if $i=2$ and  $\mathfrak g\in \{\mathfrak {so}_{2n}, \mathfrak {sp}_{2n}\}$,}
          \end{array}
        \right. \quad u_{k+2}= 0 \text{ if } i=2 \text{ and $\mathfrak g=\mathfrak {so}_{2n+1}$ }
\end{equation}
So, \eqref{omegaa} follows from \eqref{comofw} and \eqref{ukplus} by replacing $\epsilon_{\mathfrak g}u $ with $u$.
\end{proof}
Fixing  an   $i\in \{1, 2\}$ and a  $\mathfrak g\in \{\mathfrak {so}_{2n+1},\mathfrak {so}_{2n}, \mathfrak {sp}_{2n}\}$,  we have a polynomial $f_i(t)$ in Definition ~\ref{polf12}
and a set of scalars $\omega$   in Lemma~\ref{ghom123}. In this case, $\omega$ is $\mathbf u$-admissible in the sense of \eqref{admc}.

\begin{Theorem}\label{fcycy} The  functor $\Psi_{M^{\mathfrak p_{I_i}}(\lambda_{I_i, \mathbf c}) } $ factors through both  $\CB^{f_i}$ and $\CB^{f_i}(\omega)$.
\end{Theorem}
\begin{proof}
Thanks to  Lemmas~\ref{polyofx}-\ref{polyofxodd} and \ref{ghom123}, we have
$$\Psi_{M^{\mathfrak p_{I_i}}(\lambda_{I_i, \mathbf c}) }(f_i(X))=0 \text{ and }\Psi_{M^{\mathfrak p_{I_i}}(\lambda_{I_i, \mathbf c}) }(\Delta_j-\omega_j)=0, \forall j\in\mathbb N.$$
 By \eqref{ddnamphis},
$\Psi_{M^{\mathfrak p_{I_i}}(\lambda_{I_i, \mathbf c})}(I(\ob m,\ob s))=0$ and $\Psi_{M^{\mathfrak p_{I_i}}(\lambda_{I_i, \mathbf c})}(J(\ob m,\ob s))=0$, for all  $ m,s\in\mathbb N$, where $I$ (resp., $J$) is the right tensor ideal of $\AB$ generated by $f_i(X)$ (resp., $f_i(X)$ and $ \Delta_j-\omega_j$, $j\in\mathbb N$).
 So, $\Psi_{M^{\mathfrak p_{I_i}}(\lambda_{I_i, \mathbf c})}$   factors through both  $\CB^{f_i}$ and $\CB^{f_i}(\omega)$.
\end{proof}
To simplify the notation, we also denote by $\Psi_{M^{\mathfrak p_{I_i}}(\lambda_{I_i, \mathbf c}) }$ the resulting functor from $\CB^{f_i}(\omega)$ to $ \mathcal O^{\mathfrak p_{I_i}}$  in Theorem~\ref{fcycy}.
%
%
%

\begin{Assumptions}\textsf{From here onwards until Theorem~\ref{calz},   we assume that  $2r\le \min\{q_1, q_2, \ldots, q_k\}$, where $r$ is a given  positive integer. Fix  $i\in \{1, 2\}$ and  $\mathfrak g\in \{\mathfrak {so}_{2n+1},\mathfrak {so}_{2n}, \mathfrak {sp}_{2n}\}$ such that
$\mathfrak g\neq \mathfrak {so}_{2n+1}$ if $i=1$.  Let $a$ be the positive integer  given in  \eqref{defofd}. }
\end{Assumptions}

By \eqref{defwrtofd},   $d= x_{1}^{\alpha(d)_{1}}\circ \ldots \circ x_{2r}^{\alpha(d)_{2r}}\circ \hat{d}$ for any   $d\in\bar{\mathbb {ND}}^a_{0,2r}/\sim$,
where $\hat d $ is obtained from $d$ by removing all dots on  $d$.
 So,   $0\leq \alpha(d)_h\leq a-1 $ for $1\leq h\leq 2r$.
 Recall the notations $i(\hat d)$ and $j(\hat d)$ in Definition~\ref{defofij}. Write $i(\hat d)=(i_1,i_2,\ldots,i_r)$ and $j(\hat d)=(j_1,j_2,\ldots,j_r)$. Then $\{(i_l,j_l)\mid 1\leq l \leq r\}$ is the set of all cups of $d$. Note that  $\alpha(d)_{i_l}=0$ for $1\leq l\leq r$. 
 \begin{Defn}\label{defofaandf}
 Keep the notations above. Suppose $d\in\bar{\mathbb {ND}}^a_{0,2r}/\sim$. For $1\leq l\leq r$, we define
 \begin{itemize}
 \item[(1)]  $ a_{i_l,\alpha(d)_{i_l}}=l$ and $ F(d)_{i_l,\alpha(d)_{i_l}}=1$,
\item[(2)] $-a_{j_l,j}=p_{j}+l\in \mathbf p_{j+1}$ if $0\leq j\leq k-1$,
 \item[(3)] $a_{j_l,j}=p_{2k-j}-l+1\in \mathbf p_{2k-j}$ if $i=1$ and $k\leq j\leq a-1$,
 \item[(4)]  $a_{j_l,j}=p_{2k-j-1}-l+1\in \mathbf p_{2k-j-1}$ if $i=2$ and  $k\leq j\leq a-1$,
\item[(5)]  $F(d)_{j_l,0}=1$, $F(d)_{j_l,j}= F_{a_{j_l,j-1}, a_{j_l,j}} F(d)_{j_l,j-1} $ for $1\leq j\leq a-1$,
\end{itemize}
 where  $\mathbf p_j$'s  are given in \eqref{bigh}.
 \end{Defn}
 \begin{example}
 Let $d=x_1^{\alpha(d)_1}\circ\ldots\circ x_{2r}^{\alpha(d)_{2r}} \circ \eta_{\ob r}\in \bar{\mathbb {ND}}^a_{0,2r}/\sim$, where $\eta_{\ob r}$ is given in \eqref{usuflelem}.
 Then $i(\eta_{\ob r})=(r, r-1, \ldots,1)$ and  $j(\eta_{\ob r})=(r+1, r+2,\ldots,2r)$.
 So, $i_l=r-l+1$ and $j_l=r+l$ for  $1\leq l\leq r$.
 Since $d$ is normally ordered,   $\alpha(d)_{h}=0$ for  $1\leq h\leq r$.
 In this case,  $a_{h,\alpha(d)_{h}}=r-h+1$, $ F(d)_{h,\alpha(d)_{h}}=1$ for $1\leq h\leq r$.
 Moreover, for $r+1\leq h\leq 2r$,
 \begin{enumerate}
 \item[(1)]  $a_{h,j}=-(p_{j}+h-r) $ if $0\leq j\leq k-1$,
  \item[(2)] $a_{h,j}=p_{2k-j}-h+r+1$ if $i=1$ and $k\leq j\leq a-1$,
 \item[(3)]  $a_{h,j}=p_{2k-j-1}-h+r+1$ if $i=2$ and  $k\leq j\leq a-1$.
 \end{enumerate}

 \end{example}

%
The following result follows from  Definition~\ref{defofaandf}.
\begin{Lemma}\label{finuiss}
Keep the notations in Definition~\ref{defofaandf}.
\begin{enumerate}
\item $F_{a_{h,j-1}, a_{h,j}}\in \U(\mathfrak u^-_{I_i})$ for all $1\leq h\leq 2r$ and $0\leq j\leq a-1$.
\item For all admissible $m,s,h,t$,
\begin{equation}\label{noteauql}
\begin{aligned}
a_{m,s}&= a_{h,t} \text{ only if   } (m,s)=(h,t), \\
  a_{m,s}&= -a_{h,t} \text{ only if } (m,s)=(i_l,0), (h,t)=(j_l,0) \text{ for  } 1\leq l\leq r.
\end{aligned}
\end{equation}
\end{enumerate}

\end{Lemma}
Recall $M_{I_i,l}$ has a basis $\mathcal S_{i,l}$ (see Lemma~\ref{tsmodule}), for  $l\in\mathbb N$.
Write any element $x\in M_{I_i,l}$ as a linear combination of  basis elements in $\mathcal S_{i,l}$.
  We say $ym_i\otimes v_{\bf j}$ is a term of $x$ if    $ym_i\otimes v_{\bf j}$ appears in the expression of $x$ with non-zero coefficient.
For any PBW monomials $y_1,y_2$ in $\U(\mathfrak u^-_{I_i})$, we write  $y_1 \approx y_2$ if $y_1$ is obtained  from $y_2$ by changing the orders of the factors in $y_2$.
Thanks to \eqref{usefact}, up to a linear combination of lower degree terms,
\begin{equation}\label{changingoreder}
y_1m_i=  y_2m_i \text{ if } y_1 \approx y_2.
\end{equation}
 Note that $F_{h,l}=-\theta_{h,l}F_{-l,-h}$ for $h,l\in \underline N$ (see \eqref{fij}). So, for any PBW monomial $y\in \U(\mathfrak u^-_{I_i})$, there is a unique $\tilde y\in \mathcal M_{I_i}$ (see \eqref{eijv}) such that $y\approx\pm \tilde y$. For example, $\widetilde F_{p_2, p_1}\approx-F_{-p_1,-p_2} $.

\begin{Cor}\label{keyfca}Keep the notations above. Suppose  $d\in\bar{\mathbb {ND}}^a_{0,2r}/\sim$ such that   $\alpha(d)_h>0$ for some $h$.     For $1\leq j\leq a-1$,
$\widetilde F_{a_{h,j-1},a_{h,j}}m_i\otimes v_{a_{h,j}}$  is a term of  $\Omega(m_i\otimes v_{l}) $ if and only if $l= a_{h,j-1}$. In this case, the coefficient of $\widetilde F_{a_{h,j-1},a_{h,j}}m_i\otimes v_{a_{h,j}}$ is $\pm1$.
\end{Cor}
\begin{proof} Since $\alpha(d)_h>0$ for some $h$, the  $h$th endpoint on the top row of $d$ is the right endpoint of some cup of $d$.  So,  $h=j(\hat d)_l$ for some $l$,  $1\leq l\leq r$. By  Lemma~\ref{xh1} and the definitions of $F_{a_{h,j-1},a_{h,j}}$ and $a_{h,j}$ in Definition~\ref{defofaandf}, we have the result.
\end{proof}

\begin{Defn}\label{fdtad}Keep the notations in Definition~\ref{defofaandf}.
Define $ v(d)=\widetilde F(d)m_i\otimes v_{\mathbf w(d)}$, where  $F(d)=F(d)_{1,\alpha(d)_1}F(d)_{2,\alpha(d)_2}\cdots F(d)_{2r,\alpha(d)_{2r}}$ and
    $\mathbf w(d)=(a_{1,\alpha(d)_1},\ldots, a_{2r,\alpha(d)_{2r} })$.
\end{Defn}
Thanks to Lemma~\ref{finuiss}(a), $F(d)\in \U(\mathfrak u^-_{I_i})$. Hence $v(d)\in  \mathcal S_{i,2r} $ (see Lemma~\ref{tsmodule}).
By Definition~\ref{fdtad},  $  \text{deg}(v(d))=\text{deg} F(d)=\text{deg}(d)$,
 where  $\text{deg}(d)=\sum_{j=1}^{2r}\alpha(d)_j$ (see \eqref{degd}).
\begin{example} Suppose that $r=2$ and $d=x_3^{2}\circ x_4^{a-1}\circ \eta_{\ob 2}$, where $\eta_{\ob 2}$ is given in \eqref{usuflelem}.
Then $\alpha(d)=(0,0,2,a-1)$, $i(\eta_{\ob 2})=(2,1)$ and $j(\eta_{\ob 2})=(3,4)$.
Since $\alpha_{i_l}=0$ for $l=1,2$, $d\in \bar{\mathbb {ND}}^a_{0,4}/\sim$. So,
\begin{enumerate}
\item $a_{1,\alpha(d)_1}=2$, $a_{2,\alpha(d)_1}=1$, $a_{3,\alpha(d)_3}=-(p_2+1)$, $a_{4, \alpha(d)_4}=p_1-1$, $v_{\mathbf w(d)}=v_2\otimes v_1\otimes v_{-(p_2+1)}\otimes v_{p_1-1}$,
\item $F(d)_{1,\alpha(d)_1}=F(d)_{2,\alpha(d)_2}=1$, $F_{3,\alpha(d)_3}=F_{-(p_1+1), -(p_2+1)} F_{-1,  -(p_1+1)}$,
\item $F(d)_{4,j}=F_{-(p_{j-1}+2),-(p_j+2)} \cdots  F_{-(p_1+2),-(p_2+2)} F_{-2,-(p_1+2)}$, for $1\leq j\leq k-1$.
\item $F(d)_{4,a-1}=y F_{p_{k}-1, p_{k-1}-1}F_{-(p_{k-1}+2), p_{k}-1}F(d)_{4,k-1}$ if $i=1$,
\item $F(d)_{4,a-1}=y  F_{-(p_{k-1}+2), p_{k-1}-1}F(d)_{4,k-1}$ if $i=2$,

\end{enumerate}where $y=F_{p_2-1,p_1-1} F_{p_3-1,p_2-1}\cdots F_{p_{k-1}-1, p_{k-2}-1}$.
\end{example}

For all $0\le h<l\le 2r$, let
$\phi_{h,l}:\U(\mfg)^{\otimes2}\to \U(\mfg)^{\otimes(2r+1)}$  be the linear map  such that
\begin{equation}\label{pi-ab}
\phi_{h,l}(x\otimes y)=1\otimes\ldots\otimes 1\otimes \overset{h+1\text{th}} {x}\otimes 1\otimes\ldots\otimes1\otimes\overset{l+1\text{th}} y\otimes 1\otimes\ldots\otimes1, \forall x\otimes y\in \U(\mfg)^{\otimes2} . \end{equation}

\begin{Cor} \label{usecorollary}
 Suppose that  $d\in\bar{\mathbb {ND}}^a_{0,2r}/\sim$ such that $\alpha(d)_h>0$ for some $h$. If \begin{enumerate}\item  $\bar d\in\bar{\mathbb {ND}}^a_{0,2r}/\sim$  such that $\hat d=\widehat{\bar d}$ and $\alpha(\bar d)=\alpha(d)-(0^{h-1}, 1,0^{2r-h})$,
  \item  $ym_i\otimes v_{\mathbf i}\in \mathcal S_{i,2r} $ (see Lemma~\ref{tsmodule}) with degree $\text{deg}(d)-1$,
   \end{enumerate}
   then
  $v(d)$ is a term of $\phi_{0,h}(\Omega)(ym_i\otimes v_{\mathbf i})$  if and only if $\mathbf i=\mathbf w(\bar d)$ and  $ y\approx\pm F(\bar d)$. In this case,  the  coefficient is  $\pm1$.
\end{Cor}
 \begin{proof}  ``$\Leftarrow:$'' The result follows from \eqref{usefact} and Corollary~\ref{keyfca}.
 \medskip

``$\Rightarrow:$''
By  \eqref{usefact} and Lemma~\ref{xh1},  any   term of $\phi_{0,h}(\Omega)(ym_i\otimes v_{\mathbf i})$ with  degree $\text{deg}(d)$ is \textbf{of the  form}
$\bar y  m_i\otimes v_{\mathbf j}$,
 where\begin{itemize}\item  $\bar y\approx \pm y F_{i_h,l}$  for some  $ F_{i_h,l}\in\mathfrak u^-_{I_i}$, \item  $ \mathbf j$ is obtained from $\mathbf i$   by replacing $i_h$ with $l$.\end{itemize}
 If $v(d)$ is one of such terms,  then
$F(d) \approx \pm y F_{i_h, a_{h,\alpha(d)_h}}$ and  $i_t=\mathbf w(d)_t$ whenever $t\neq h$.
 By \eqref{noteauql}, $y\approx\pm  F(\bar d)$ and
 $i_h=\mathbf w(\bar d)_h$. So, $\mathbf i=\mathbf w(\bar d)$.
 \end{proof}

Suppose  $d, d'\in\bar{\mathbb {ND}}^a_{0,2r}/\sim$. We write $\Psi_{M^{\mathfrak p_{I_i}}(\lambda_{I_i, \mathbf c})}(d')(m_i)$ as a linear combination of  basis elements  of  $ M_{I_i, 2r}$ in Lemma~\ref{tsmodule}.
Let $c_{d,d'}$ be the coefficient of $v(d)$ in this expression.
\begin{Lemma}\label{L:coefficients technical lemma}
Suppose  $d, d'\in\bar{\mathbb {ND}}^a_{0,2r}/\sim $.   \begin{enumerate} \item   $c_{d,d'}=\pm 1$ if $d=d'$. \item $c_{d,d'}=0$ if either $\text{deg}(d')<\text{deg}(d)$ or $\text{deg}(d')=\text{deg}(d)$
and $  d\neq  d'$.
\end{enumerate} \end{Lemma}

\begin{proof} Following \eqref{defwrtofd}, we have  $d'=x_{1}^{\alpha(d')_{1}}\circ \ldots \circ x_{2r}^{\alpha(d')_{2r}}\circ \widehat{d'}$ and   $\widehat{d'}\in \bar {\mathbb B}_{0,2r}/\sim$.  Then
  \begin{equation}\label{copwidfhfgff}\begin{aligned}\Psi_{M^{\mathfrak p_{I_i}}(\lambda_{I_i, \mathbf c})}(d')(m_i)\overset{\eqref{defofzd1}} =&\pm\Psi_{M^{\mathfrak p_{I_i}}(\lambda_{I_i, \mathbf c})}(x_{1}^{\alpha(d')_{1}}\circ \ldots \circ x_{2r}^{\alpha(d')_{2r}})(m_i\otimes z(\widehat {d'}))\\
  \overset{\eqref{actionofxxx}}=& \pm Y_{1}^{\alpha(d')_{1}} Y_{2}^{\alpha(d')_{2}} \cdots   Y_{2r}^{\alpha(d')_{2r}}(m_i\otimes z(\widehat {d'})),
  \end{aligned} \end{equation}
  where $z(\widehat {d'})$ is given in \eqref{defofzd}, $Y_j=X_j\otimes \text{Id}_{V^{\otimes 2r-j}}$ and $X_j$'s are given in \eqref{defofxk} under the assumption that  $M=M^{\mathfrak p_{I_i}}(\lambda_{I_i, \mathbf c})$.
Thanks to  \eqref{defofxk} and Lemma~\ref{eM filtered de },  the degree of  any   term    of $\Psi_{M^{\mathfrak p_{I_i}}(\lambda_{I_i, \mathbf c})}(d')(m_i)$
is less than or equal to $\text{deg}(d')$.  We can use $\phi_{0, j}(\Omega)$ to replace
$Y_j$ in \eqref{copwidfhfgff} when we  compute the terms
of  $\Psi_{M^{\mathfrak p_{I_i}}(\lambda_{I_i, \mathbf c})}(d')(m_i)$ with the highest degree $\text{deg}(d')$ (see \eqref{pi-ab} for  $\phi_{0, j}$).
Since $\text{deg}(d)=\text{deg} (v(d))$, we have   $c_{d,d'}=0$ if  $\text{deg}(d')<\text{deg}(d)$.

Suppose that $\text{deg}(d)=\text{deg}(d')=0$.   Then  $F(d)=1$ and $v(d)=m_i\otimes v_{\theta(d)}$, where $\theta(d)$ is given in \eqref{defofwdbed}. By  \eqref{coeffi},
we have $ c_{d,d'}=\pm\delta_{d,d'}$.
 Using   Corollary~\ref{usecorollary} and induction on $\text{deg}(d)$ yields (a). If $r=1$, then
 $d=d'$ if and only if $\text{deg}(d)=\text{deg}(d')$, for any  $d, d'\in  \bar{\mathbb {ND}}^a_{0, 2}/\sim$. So, the result is proven for the case $r=1$.
In the following, we assume that     $r>1$, $\text{deg}(d')=\text{deg}(d)>0$
and $  d\neq  d'$.

Since $ 2r $ is the rightmost endpoint of $d'$, there is a unique $j$ such that  $(j,2r)$ is a cup of $d'$. In this case,    $j$  is the left endpoint of $(j, 2r)$. Suppose  $ um_i \otimes v_\mathbf l\in \mathcal S_{i,2r} $ (see Lemma~\ref{tsmodule}) for some PBW monomial $u\in \U(\mathfrak u_{I_i}^-)$ and $\mathbf l=(l_1,l_2,\ldots,l_{2r})\in\underline N^{2r}$, which   appears as a term of $\Psi_{M^{\mathfrak p_{I_i}}(\lambda_{I_i, \mathbf c})}(d')(m_i)$
 with the  highest degree $\text{deg}(d')$.  There are four cases  we have to deal with as follows.

 {\textbf{Case 1:} $\alpha(d')_{2r}=0$ and   either $\alpha(d)_{2r}>0$ or  $\alpha(d)_{2r}=0$ and   $(j,2r)$ is not a cup of $d$.}
  In this case,  we  have  $l_j=-l_{2r}$ by \eqref{defofzd}. Thanks to  \eqref{noteauql}, $a_{j,\alpha(d)_j}\neq - a_{2r,\alpha(d)_{2r}}$. Thus  $v(d)$  can not be a term   of $\Psi_{M^{\mathfrak p_{I_i}}(\lambda_{I_i, \mathbf c})}(d')(m_i)$, forcing     $c_{d,d'}=0$.

{\textbf{Case 2:} $\alpha(d)_{2r}>0$ and $\alpha(d')_{2r}>0$.} In this case, we define $\bar d$ such that $\hat{d}=\widehat{\bar d }$ and $ \alpha(\bar d)= \alpha(d)-(0^{2r-1},1) $. Similarly, we define  $\bar {d'}$.  So, $d=x_{2r}\circ \bar d$, $d'=x_{2r}\circ \bar {d'}$ and $\bar {d'} \neq \bar d$.
 By   Corollary~\ref{usecorollary} and \eqref{changingoreder},
 $ v(\bar d)$ is  a term   of  $\Psi_{M^{\mathfrak p_{I_i}}(\lambda_{I_i, \mathbf c})}(\bar {d'})(m_i)$   if  $c_{ d, {d'}}\neq0$. However, by inductive assumption,
 $c_{\bar d,\bar {d'}}=0$ since  $\text{deg} (\bar d)=\text{deg} (\bar {d'})=\text{deg}(d)-1$  and $\bar {d'} \neq \bar d$.  So,  $ v(\bar d)$ can not be  a term   of  $\Psi_{M^{\mathfrak p_{I_i}}(\lambda_{I_i, \mathbf c})}(\bar {d'})(m_i)$.

 {\textbf{Case 3:} $\alpha(d')_{2r}>0$ and $\alpha(d)_{2r}=0$.}   By Lemma~\ref{xh1} and \eqref{usefact},  $-l_{2r}\notin \mathbf p_1$. Since $\alpha(d)_{2r}=0$,  we have $-a_{2r,\alpha(d)_{2r}}=r\in\mathbf p_1$. Thus
 $v(d)$  can not be a term    of $\Psi_{M^{\mathfrak p_{I_i}}(\lambda_{I_i, \mathbf c})}(d')(m_i)$ and   $c_{d,d'}=0$.

 {\textbf{Case 4:} $\alpha(d')_{2r}=\alpha(d)_{2r}=0$ and $(j,2r)$ is    a cup of $d$.} In this case,  $(a_{j, \alpha(d)_j},a_{2r,\alpha(d)_{2r}})=(r,-r)$  and there is no $\bullet$'s on the cup $(j,2r)$ of $d'$ and $d$.  So,  $\alpha(y)_j=\alpha(y)_{2r}=0$ for $y=d,d'$.
 Recall the notations $i(\hat d)$ and $j(\hat d)$ in Definition~\ref{defofij}.
Let $d_1,  d'_1\in \bar{\mathbb{ND}}^a_{0,2(r-1)}/\sim$ such that
\begin{equation}\label{rfrgrg}
 \alpha(y_1)_{l}=\left\{
                   \begin{array}{ll}
                    \alpha(y)_{l}, & \hbox{ for $1\leq l\leq j-1$;} \\
                   \alpha(y)_{l+1}  , & \hbox{ for $j\leq l\leq 2(r-1)$,}
                   \end{array}
                 \right.
 \end{equation} and for $1\leq l\leq r-1$,
\begin{equation}\label{eudheudh}
\begin{aligned}
i(\widehat y_1)_{l}&= i(\hat y)_{l},  \text{ if $1\leq i(\hat y)_{l}<j$},~~\quad  j(\widehat y_1)_{l}= j(\hat y)_{l}, \text{ if $1\leq j(\hat y)_{l}<j$},\\
i(\widehat y_1)_{l}&= i(\hat y)_{l}-1, \text{ if $j<i(\hat y)_{l}<2r$},~~\quad j(\widehat y_1)_{l}= j(\hat y)_{l}-1, \text{ if $j<j(\hat y)_{l} <2r$}.
\end{aligned}\end{equation}
Then $d_1$ and $d_1'$ can be viewed  as the diagrams obtained from $d$ and $d'$ by removing the cup $(j,2r)$, respectively (see Example~\ref{eaxmdkfofddll}).
  By \eqref{rfrgrg} and \eqref{eudheudh},  $d_1\neq d_1'$,
 $\text{deg}(d_1)=\text{deg}(d_1')=\text{deg}(d)$.
Moreover,  $F(d)=F(d_1)$ and $\mathbf w(d_1)$ is obtained from $\mathbf w(d)$ by deleting $a_{j, \alpha(d)_j}$ and $a_{2r,\alpha(d)_{2r}}$.
 By \eqref{defofzd1} and \eqref{eudheudh}, $z(\widehat d_1')$ is obtained from $z(\widehat d')$ by deleting the tensor factors $z(\widehat d')_j$ and $z(\widehat d')_{2r}$.
 Note that  $\alpha(d')_j=\alpha(d')_{2r}=0$. By \eqref{copwidfhfgff} and \eqref{rfrgrg}, $ um_i \otimes v_\mathbf l$ is a term of  $\Psi_{M^{\mathfrak p_{I_i}}(\lambda_{I_i, \mathbf c})}(d')(m_i)$ with the highest degree  $\text{deg}(d')$ only if
   $um_i\otimes v_{ {\mathbf j}}$ is also a term of $\Psi_{M^{\mathfrak p_{I_i}}(\lambda_{I_i, \mathbf c})}(d_1')(m_i)$
 with the  highest degree $\text{deg}(d')$, where ${\mathbf j}$ is obtained from $\mathbf l$ by deleting $l_j$ and $l_{2r}$.
  Therefore,  $c_{d_1,d_1'}\neq0$ if $c_{d,d'}\neq0$. However,
by induction assumption on $r-1$, we have $c_{d_1,d_1'}=0$, forcing  $c_{d,d'}=0$.
\end{proof}
\begin{example}\label{eaxmdkfofddll}Suppose that $k\geq2$.
Define
$$ d=\begin{tikzpicture}[baseline = 25pt, scale=0.35, color=\clr]
        \draw[-,thick] (0,4) to[out=down,in=left] (1.5,2) to[out=right,in=down] (3,4);
        \draw[-,thick] (1.5,4) to[out=down,in=left] (3,2) to[out=right,in=down] (4.5,4);
         \draw[-,thick] (3.5,4) to[out=down,in=left] (5,2) to[out=right,in=down] (6.5,4); \draw(3,3.6)\bdot;\draw(4.5,3.6)\bdot;
           \end{tikzpicture}~,\quad d_1=
           \begin{tikzpicture}[baseline = 25pt, scale=0.35, color=\clr]
        \draw[-,thick] (0,4) to[out=down,in=left] (1.5,2) to[out=right,in=down] (3,4);
        \draw[-,thick] (1.5,4) to[out=down,in=left] (3,2) to[out=right,in=down] (4.5,4); \draw(3,3.6)\bdot;\draw(4.5,3.6)\bdot;
           \end{tikzpicture}.
           $$
     Then $d_1$ can be viewed as the diagram  obtained from $d$ by deleting the cup $(4,6)$.  We have
\begin{itemize}
\item $\alpha(d)=(0,0,1,0,1,0)$, $\alpha(d_1)=(0,0,1,1)$;
\item  $i(\hat d)=(1,2,4)$, $ j(\hat d)=(3,5,6)$, $i(\widehat d_1)=(1,2)$, $ j(\widehat d_1)=(3,4)$;
\item   $\mathbf w(d)=(1,2,-(p_1+1),3,-(p_1+2),-3)$,  $\mathbf w( d_1)=(1,2,-(p_1+1),-(p_1+2))$,
\item   $ F(d)=F(d_1)=F_{-1,-p_1-1}F_{-2,-p_1-2}$.
 \end{itemize}     \end{example}
It follows from Lemma~\ref{numberss} and Definition~\ref {D:N.O. dotted  OBC diagram} that the cardinality of $\bar{\mathbb {ND}}^a_{m,s}/\sim$ is  $a^r(2r-1)!!$ (resp., 0) if $m+s=2r$ (resp., $m+s$ is odd).
\begin{Prop}\label{keytheorem}
 $\Hom_{\CB^{f_i}(\omega)}(\ob 0,\ob {2r})$ has  $\mathbb C$-basis given by $\bar{\mathbb {ND}}^a_{0,2r}/\sim$.
\end{Prop}

\begin{proof}
  By Proposition~\ref{regularspan}(c), it is enough to prove that   $\bar{\mathbb {ND}}^a_{0,2r}/\sim$ is linearly independent. Suppose
   \begin{equation}\label{combesjd}
   \sum_{d\in \bar{\mathbb {ND}}^a_{0,2r}/\sim}  g_{d} d=0,
   \end{equation}
where  $g_d$'s $\in \mathbb C$.  If $g_d\neq 0$ for some $d\in \bar{\mathbb {ND}}^a_{0,2r}/\sim $, we choose    $d_0$ among  such elements   such that   $ \text{deg}(d_0)$ is maximal.
      Thanks to Lemma~\ref{L:coefficients technical lemma},
  the coefficient  of $v(d_0)$ in $\Psi_{M^{\mathfrak p_{I_i}}(\lambda_{I_i, \mathbf c})}(\sum_{d\in \bar{\mathbb {ND}}^a_{0,2r}/\sim}g_{d} d)(m_i)$ is the non-zero scalar $\pm g_{d_0} $. This contradicts \eqref{combesjd}. \end{proof}

Recall that  $\mathbb Z_{(2)}$ is the localization of $\mathbb Z$ by  the set $\{2^m\mid m\in\mathbb N\}$.
\begin{Theorem}\label{calz}Let $\hat f(t)=\prod_{j=1}^a(t-\hat u_j)\in \mathcal Z[t]$, where $\mathcal Z=\mathbb Z_{(2)}[\hat u_1, \ldots, \hat u_a]$ is the  ring of polynomials in variables $\hat u_1, \ldots, \hat u_a$. Suppose  $\tilde \omega=\{\tilde\omega_j\mid j\in \mathbb N \}$. If $\tilde \omega$ is      $\hat{\mathbf u}$-admissible in the sense of \eqref{admc},
then the $\mathcal Z$-module  $\Hom_{\CB^{\hat f}(\tilde\omega)}(\ob 0,\ob {2r})$ has basis given by $\bar{\mathbb {ND}}^a_{0,2r}/\sim$.
\end{Theorem}
\begin{proof}
 We claim that
 \begin{equation}\label{cliamhdg}
 \sum_{d\in \bar{\mathbb {ND}}^a_{0,2r}/\sim}  g_{d}(\hat u_1,\ldots, \hat u_a) d=0 \ \ \text{  if and only if  $g_d(\hat u_1,\ldots, \hat u_a)=0$,}
 \end{equation}
where  $ g_d(\hat u_1,\ldots, \hat u_a)\in \mathcal Z$.
In order to prove the claim, we need to  specialize $\hat u_j$'s  at some scalars $u_j$'s in $\mathbb C$.
If $a=2k$, we assume $i=1$. If $a=2k-1$, we assume $i=2$.
 Define
$$\mathbf u_1= (u_1,\ldots, u_{2k}) \text { and } \mathbf u_2=(u_1,u_2,\ldots, u_k,u_{k+2}, u_{k+3},\ldots,u_{2k}), $$
  where $u_j$'s are  given in \eqref{uij1} such that   $n=\sum_{j=1}^k q_j$  and     $\mfg=\mathfrak{so}_{2n}$.
   If $a=2k$  and $(\hat u_1,\hat u_2,\ldots, \hat u_a)$ is specialized  at $\mathbf u_1$, then $\hat f(t)$ is specialized to $f_1(t)$ in Definition~\ref{polf12}(a).
If $a=2k-1$ and  $(\hat u_1,\hat u_2,\ldots, \hat u_a)$ is specialized  at $\mathbf u_2$, then $\hat f(t)$ is specialized to $f_2(t)$ in Definition~\ref{polf12}(a).
 Since $\tilde \omega_j$'s are uniquely  determined by $\hat u_j$'s if it is $\hat {\mathbf u}$-admissible (see \eqref{wdeterminedbyu}), $\tilde \omega$ is  specialized to $\omega$ in Lemma~\ref{ghom123} for $\mathfrak g=\mathfrak{so}_{2n}$ and $i=1,2$.

Denote by $\mathcal C_\kappa$  the category $\mathcal C$ over $\kappa$. Considering $\mathbb C$ as the $\mathcal Z$-module such that $(\hat u_1,\hat u_2,\ldots, \hat u_a)$ acts on $\mathbb C$ via   $\mathbf u_i$, we have  a $ \mathcal Z$-linear monoidal functor $\mathcal F_i: \AB_{\mathcal Z}\rightarrow \AB_{\mathbb C}$ which sends the generators to generators with the same names.
Let   $I_{\mathbb C}^i$  be the right tensor ideal of   $\AB_{\mathbb C}$ generated by
 $f_i(X)$, $\Delta_j-\omega_j$, $j\in\mathbb N$, where  $f_i(t)$ is given in Definition~\ref{polf12}(a) and $\omega_j$'s are given in Lemma~\ref{ghom123} for $\mfg=\mathfrak{so}_{2n}$.
Then $\mathcal F_i(\hat f(X))=f_i(X)$ and $\mathcal F_i(\Delta_j-\tilde \omega_j)= \Delta_j-\omega_j$, $j\in\mathbb N$.
So, $\mathcal F_i$ induces a $\mathcal Z$-linear functor $\tilde {\mathcal F}_i: \CB^{\hat f}(\tilde\omega)\rightarrow \CB^{  f_i}( \omega)$, where $\CB^{  f_i}( \omega)=\AB_{\mathbb C}/I_{\mathbb C}^i$.

Suppose that  $\sum_{d\in \bar{\mathbb {ND}}^a_{0,2r}/\sim}  g_{d}(\hat u_1,\ldots, \hat u_a) d=0$ in $\CB^{\hat f}(\tilde\omega)$. Then
$$\sum_{d\in \bar{\mathbb {ND}}^a_{0,2r}/\sim}  \tilde {\mathcal F}_i (g_{d}(\hat {\mathbf u}) d)=\sum_{d\in \bar{\mathbb {ND}}^a_{0,2r}/\sim}  g_{d}( \mathbf u_i) d=0.$$
Therefore,
$$\Psi_{M^{\mathfrak p_{I_i}}(\lambda_{I_i, \mathbf c})}\left( \sum_{d\in \bar{\mathbb {ND}}^a_{0,2r}/\sim}  g_{d}( \mathbf u_i) d\right)(m_i)=0, \quad i=1,2,$$
where $\Psi_{M^{\mathfrak p_{I_i}}(\lambda_{I_i, \mathbf c})}$ is given in Theorem~\ref{fcycy}.
Suppose  that  $2r\leq \min\{q_j\mid 1\leq j\leq k\}$. By Proposition~\ref{keytheorem},
 we have
 \begin{equation}\label{equalyozr}
 g_{d}( \mathbf u_i)=0, \text{ for all $d\in\bar{\mathbb {ND}}^a_{0,2r}/\sim$}.
 \end{equation}
 Thanks to \eqref{uij1},
 \begin{equation}\label{ujkshd}
 u_j=\left\{
       \begin{array}{ll}
         c_j+\sum_{j\leq l\leq k}q_l-1/2, & \hbox{$1\leq j\leq k$;} \\
         -c_{2k-j+1}-\sum_{k+2\leq l\leq j}q_{2k-l+2} +1/2, & \hbox{$k+1\leq j\leq 2k$.}
       \end{array}
     \right.
 \end{equation}
So,
  we can view $ u_j$'s as polynomials of $c_l$'s and $q_j$'s, where $1\leq j\leq k$ and  $ 1\leq l\leq k$ (resp., $1\leq l\leq k-1 $) if $a=2k$ (resp., $a=2k-1$).
If $a=2k$, define $\phi: \mathbb C^a\rightarrow \mathbb C^a$ such that
 $$\phi(q_1, \ldots,q_k, c_1,\ldots, c_k)=(u_1,\ldots,u_k,u_{2k}, \ldots, u_{k+1}).$$
If $a=2k-1$, define $\psi: \mathbb C^a\rightarrow \mathbb C^a$ such that  $$\psi(q_1, \ldots,q_k, c_1,\ldots, c_{k-1})=(u_1,\ldots,u_k,u_{2k}, \ldots, u_{k+2}).$$
Thanks to \eqref{ujkshd},  the Jacobi matrix of $\phi$ is
$$J_\phi=\left(
  \begin{array}{cc}
    U_{k} & I_k \\
    -U_k+I_k & -I_k \\
  \end{array}
\right)$$
where $I_k$ is the identity matrix of rank $k$ and $U_k$ is the upper-triangular matrix such that the $(l,j)$th entry is $1$ for $l\leq j$. The  Jacobi matrix $J_\psi$ is obtained from $J_\phi$ by deleting the last column and last row.
It is routine to check that   $\text{det}J_\phi=(-1)^{k}$ and $\text{det}J_\psi=(-1)^{k-1}$. So, both $\phi$ and $\psi$  are dominant.
 Define
   $$\begin{aligned} & U_1=\{(q_1,q_2,\ldots, q_k,c_1,c_2,\ldots,c_k)\mid q_j\in \mathbb Z \text{ and }q_j\geq 2r, 1\leq j\leq k\},\\
   & U_2=\{(q_1,q_2,\ldots, q_k,c_1,c_2,\ldots,c_{k-1})\mid q_j\in \mathbb Z \text{ and }q_j\geq 2r, 1\leq j\leq k\}.\\
   \end{aligned} $$
 Then  $U_i$      is Zariski dense in $\mathbb C^a$, $i=1,2$.
 Since $\phi$ and $\psi$  are dominant, $\phi(U_1)$ and $\psi(U_2)$ are Zariski dense in $\mathbb C^a$.
This observation together with \eqref{equalyozr} yield  the claim in \eqref{cliamhdg}.
 Finally,  the result follows from the claim and Proposition~\ref{regularspan}(c).

\end{proof}
\begin{Cor}\label{basiskf}
Let $\kappa$ be a commutative ring containing the multiplicative identity $1$ and invertible element $2$.  If $f(t)$ is  the polynomial  in Definition~\ref{COBC defn} and $\omega$ is $\mathbf u$-admissible in the sense of \eqref{admc}, then
  $\Hom_{\CB^f(\omega)}(\ob m,\ob {s})$ has  $\kappa$-basis given by  $\bar{\mathbb {ND}}^a_{m,s}/\sim$,
for all   $m,s \in \mathbb N$.
\end{Cor}
\begin{proof} If $m+s $  is odd, then $\bar{\mathbb {ND}}^a_{m,s}=\emptyset $  and  $\Hom_{\CB^f(\omega)}(\ob m,\ob {s})=0$.
  Suppose that $m+s=2r$ for some $r\in\mathbb N$. Recall $\mathcal Z$ and
$\CB^{\hat f}(\tilde \omega)$ in Theorem~\ref{calz}.
  Define  $\mathcal C:= \kappa\otimes _{\mathcal Z}\CB^{\hat f}(\tilde \omega)$   such that $\hat u_j$ acts on $\kappa$ as $u_j$ for $1\leq j\leq a$.
  By Theorem~\ref{calz} and standard arguments on the base change property, $\Hom_{\mathcal C}(\ob 0,\ob {2r})$ has  $\kappa$-basis given by  $\bar{\mathbb {ND}}^a_{0,2r}/\sim$.
There is an obvious $\kappa$-linear functor $F: \AB\rightarrow \mathcal C$,  which sends  elements in  ${\mathbb {D}}_{m,s}$ to elements in  $\mathbb {D}_{m,s}$ with the same names for all $m, s\in \mathbb N$.
Then $F(f(X))=0$.   By \eqref{wdeterminedbyu}, $\tilde \omega_j$'s are uniquely  determined by $\hat u_j$'s if it is $\hat {\mathbf u}$-admissible. So, $\tilde \omega_j$ acts on $\kappa$ as $\omega_j$  and  $F(\Delta_j-\omega_j)=0$, $j\in\mathbb N$. Moreover,  $F(J(\ob m,\ob s))=0$ for all $\ob m,\ob s$, where $J$ is the right tensor ideal
of $\AB$
generated by $f(X)$ and $\Delta_j-\omega_j$, $j\in\mathbb N$. Therefore, $  F$ factors through $ \CB^f(\omega)$. It   sends the required basis of   $\Hom_{\CB^f(\omega)}(\ob 0,\ob {2r})$ to the corresponding basis of
 $\Hom_{\mathcal C}(\ob 0,\ob {2r})$.
 Then the result for $\Hom_{\CB^f(\omega)}(\ob 0,\ob {2r})$ follows from the result for $\Hom_{\mathcal C}(\ob 0,\ob {2r})$.


Suppose  $m>0$.  Thanks to   Lemma~\ref{key1}, there is a $\kappa$-module isomorphism
  \begin{equation}\label{isommnor}
   \Hom_{\CB^{f}(\omega)}(\ob m,\ob s) \cong \Hom_{\CB^{f}(\omega)}(\ob 0,\ob {2r}).
  \end{equation}
 Now, the result follows from  Proposition~\ref{regularspan}(c), immediately.   \end{proof}

\begin{proof}[\textbf{Proof of Theorem~\ref{basis of cyc thm2}}]

``$\Leftarrow$": Thanks to Corollary~\ref{basiskf}, the result is true.

``$\Rightarrow$": In particular, $ \End_{\CB^f(\omega)}(\ob r)$ has a $\kappa$-basis given by $\bar{\mathbb {ND}}^a_{r,r}/\sim$.
Since $\CB^f(\omega)$ is the quotient category of $\AB$, we have the quotient functor  from  $\AB$ to $ \CB^f(\omega)$.
 Thus, we have an epimorphism $\pi:\End_{\AB}(\ob r)\rightarrow  \End_{\CB^f(\omega)}(\ob r)$. Moreover,
 \begin{equation}\label{piepim}
 \pi(f(x_1))=f(X)\otimes 1_{\ob{r-1}}=0,  \pi(\Delta_j1_{\ob r})=\omega_j1_{\ob r}  \text{ for $j\in\mathbb N$,}
 \end{equation}
 where $x_1$ is given in \eqref{defofxy}.
 Recall the affine Nazarov-Wenzl algebra $\mathcal W_r^{\rm aff}$ in Definition~\ref{definition of Na} and the epimorphism   $\gamma: \mathcal W_r^{\rm aff}\rightarrow \End_{\AB}(\ob r)$  in the proof of  Theorem~\ref{affiso}. So,
 \begin{equation}\label{gammapedjh}
 \gamma(f(\ob x_1))=f(x_1), \gamma(\hat\omega_j)=\Delta_j1_{\ob r}, \text{ for $j\in\mathbb N$,}
  \end{equation}
  where $\ob x_1$ is the generator of $\mathcal W_r^{\rm aff}$ in Definition~\ref{definition of Na}.
  Let  $ \gamma_1:=\pi\circ\gamma:  \mathcal W_r^{\rm aff} \rightarrow \End_{\CB^f(\omega)}(\ob r)$. Then  $  \gamma_1$ is also surjective.
 By \eqref{piepim}--\eqref{gammapedjh},
 \begin{equation}\label{ssssff}
 \gamma_1(f(\ob x_1))=0, \gamma_1(\hat\omega_j)=\omega_j1_{\ob r}, ~j\in\mathbb N.
 \end{equation}
Recall the algebra $\mathcal W_{r}^{\rm aff }(\omega)$   is  obtained from  $\mathcal W_r^{\rm aff}$ by specializing the central elements  $\hat\omega_j$  at $\omega_j$ in $\kappa$
 for all admissible $j$.
  Let $J=\langle f(\ob x_1)\rangle$ be the two sided ideal of $\mathcal W_{r}^{\rm aff }(\omega) $ generated by $f(\ob x_1)=\prod_{i=1}^a (\ob x_1-u_i)$.
 By \eqref{ssssff},   $ \gamma_1$   factors through  the cyclotomic Nazarov-Wenzl algebra $\mathcal  W_{a, r}(\omega)$, where
  \begin{equation} \label{liewall} \mathcal  W_{a, r}(\omega)=\mathcal W_{r}^{\rm aff }(\omega)/J.\end{equation}
We   denote the induced epimorphism by $\bar\gamma: \mathcal  W_{a, r}(\omega) \rightarrow\End_{\CB^f(\omega)}(\ob r)$.

  Suppose  $d\in\bar{\mathbb {ND}}^a_{r,r}/\sim$. By \eqref{degdregu},
  \begin{equation}\label{cycregular}
  d=  x_{r}^{\alpha(d)_{r}}\circ\ldots\circ x_1^{\alpha(d)_1} \circ\hat d\circ x_{r}^{\beta(d)_{r}}\circ\ldots\circ x_1^{\beta(d)_1},
  \end{equation}
  where $0\leq \alpha(d)_j, \beta(d)_h\leq a-1$ and   $\alpha(d)_j=0$ (resp., $\beta(d)_h=0$) if the $j$th (resp., $h$th) endpoint at the top (resp., bottom) row  is the left  endpoint of a cup (resp., the right endpoint of a cap or the endpoint on a vertical strand).
Similar to \eqref{regulainw},  if we use $\ob x_j$ (resp., the multiplication of $\mathcal W_{a,r}(\omega)$) to replace $x_j$ (resp., each $\circ$) in \eqref{cycregular}, $1\le j\le r$,
we obtain an element, say $\tilde d\in \mathcal W_{a, r}(\omega)$.
  Such an element is called a  regular  monomial in $\mathcal W_{a, r}(\omega)$ (see, e.g., \cite[Definition~2.14]{AMR}).
   Further, $\bar \gamma(\tilde d)=d$.
   This shows   that the set of all  regular monomials $\tilde d$ is $\kappa$-linear independent in  $\mathcal W_{a, r}(\omega)$. By \cite[Theorem~5.5]{AMR},  the number of all such regular monomials    is $a^r(2r-1)!!$ . In \cite{G09},  Goodman proved that
$\mathcal  W_{a, r}(\omega)$ is always  free over $\kappa$. However,  its   rank is less than or equal to  $a^r(2r-1)!!$. Thanks to  \cite[Theorem~5.2]{G09}, the equality holds if and only if $\omega$  is  $\mathbf u$-admissible.  This completes the proof.
\end{proof}

In the proof   Theorem~\ref{basis of cyc thm2}, we have the epimorphism  $\bar \gamma :\mathcal  W_{a, r}(\omega)\rightarrow\End_{\CB^f(\omega)}(\ob r)$.
By  \cite[Theorem~5.5]{AMR},  $\mathcal  W_{a, r}(\omega)$ is  free over $\kappa$ with rank $(2r-1)!! a^r$ if $\omega$ is $\mathbf u$-admissible in the sense of \eqref{admc}.
So, \begin{equation} \label{catec}\bar \gamma:\mathcal  W_{a, r}(\omega)\rightarrow   \End_{\CB^f(\omega)}(\ob r) \end{equation} is actually an isomorphism if $\omega$ is $\mathbf u$-admissible.
\section{Decomposition numbers of cyclotomic Nazarov-Wenzl algebras}
 The aim of this section is  to  establish isomorphisms  between   cyclotomic Nazarov-Wenzl algebras and certain endomorphism algebras in the BGG parabolic categories  $\mathcal O$
associated to  orthogonal and symplectic Lie algebras.
Using standard arguments in \cite{AST, RS1, RS2}, we obtain formulae to compute  decomposition numbers of  (level $a$) cyclotomic Nazarov-Wenzl algebras.
Throughout, we fix $\kappa=\mathbb C$,   $i\in\{1,2\}$  and $\mfg\in\{ \mathfrak {so}_N, \mathfrak {sp}_{2n} \}$. Let
$$M_{I_i,r}:= M^{\mathfrak p_{I_i}}(\lambda_{I_i, \mathbf c})\otimes V^{\otimes r} \text{ and $ \mathcal W_{a, r}(\omega)=\mathcal W_{r}^{\rm aff} (\omega)/J$,}$$
where
\begin{itemize}
 \item  $a=\text{deg}f_i(t)$, and $f_i(t)$ is given  in Definition~\ref{polf12},
 \item $J$ is the two sided ideal generated by $f_i(\ob x_1)$,
\item $\omega$ is given  in Lemma~\ref{ghom123}.
   \end{itemize}

Recall the functor $\Psi_{M^{\mathfrak p_{I_i}}(\lambda_{I_i, \mathbf c})}:\AB\to \mathcal O^{\mathfrak p_{I_i}}$ in  \eqref{fucotssjd}. By Theorem~\ref{fcycy}, it factors through  $\CB^{f_i}(\omega)$.

\begin{Defn}\label{tileta}
 For any $m,s,r\in\mathbb N$ such that $m+s=2r$ and $m>0$, define
 two  $\mathbb C$-linear maps
\begin{enumerate}\item
 $\tilde\eta_{\ob m}=(-\otimes \text{Id}_{V^{\otimes  m}}) \circ \Psi_{M^{\mathfrak p_{I_i}}(\lambda_{I_i, \mathbf c})}(\eta_{\ob m}): \Hom_{\mathcal O}( M_{I_i, m}, M_{I_i, s})\rightarrow\Hom_{\mathcal O}(M_{I_i, 0}, M_{I_i, 2r} )$,
 \item $\tilde\varepsilon _{\ob m}= (\text{Id}_{M_{I_i, s} }\otimes \Phi(\varepsilon _{\ob m})) \circ(-\otimes \text{Id}_{V^{\otimes  m}}): \Hom_{\mathcal O}(M_{I_i, 0},M_{I_i, 2r} )\rightarrow\Hom_{\mathcal O}(M_{I_i, m}, M_{I_i,  s}) $,\end{enumerate}
where $\eta_{\ob m}$, $\varepsilon _{\ob m}$ are given in \eqref{usuflelem} and $\Phi$ is given in Proposition~\ref{level1s}.\end{Defn}

\begin{Lemma}\label{conmuf} For any $m,s,r\in\mathbb N$ such that $m+s=2r$ and $m>0$,
$\tilde\eta_{\ob m}$ and $\tilde\varepsilon _{\ob m}$ are mutually inverse to each other. Moreover, the following diagrams are commutative:
\begin{equation}\label{conmuf1}
\begin{aligned}
\Hom_{\CB^{f_i}(\omega)}(\ob 0, \ob {2r})& \overset{\Psi_{M^{\mathfrak p_{I_i}}(\lambda_{I_i, \mathbf c})}}\longrightarrow \Hom_{\mathcal O}( M_{I_i, 0},M_{I_i, 2r})\\
\bar \varepsilon _{\ob m}\downarrow\quad\quad & \quad\quad\quad\quad\quad \quad\downarrow \tilde\varepsilon _{\ob m}\\
\Hom_{\CB^{f_i}(\omega)}(\ob m, \ob {s})& \overset{\Psi_{M^{\mathfrak p_{I_i}}(\lambda_{I_i, \mathbf c})}}\longrightarrow \Hom_{\mathcal O}( M_{I_i, m},M_{I_i, s}),\\
\end{aligned}
\end{equation}
\begin{equation}\label{conmuf2}
\begin{aligned}
\Hom_{\CB^{f_i}(\omega)}(\ob 0, \ob {2r})& \overset{\Psi_{M^{\mathfrak p_{I_i}}(\lambda_{I_i, \mathbf c})}}\longrightarrow \Hom_{\mathcal O}( M_{I_i, 0},M_{I_i, 2r})\\
\bar \eta_{\ob m}\uparrow\quad& \quad\quad \quad\quad\quad\quad\uparrow \tilde\eta_{\ob m}\\
\Hom_{\CB^{f_i}(\omega)}(\ob m, \ob {s})& \overset{\Psi_{M^{\mathfrak p_{I_i}}(\lambda_{I_i, \mathbf c})}}\longrightarrow  \Hom_{\mathcal O}( M_{I_i, m},M_{I_i, s}),\\
\end{aligned}
\end{equation}
where $ \bar \varepsilon _{\ob m}$ and  $\bar\eta_{\ob m}$ are given  in Definition~ \ref{etae}.
\end{Lemma}
\begin{proof}
By Theorem~\ref{affaction} and \eqref{ddnamphis},
\begin{equation}\label{egama}
\Psi_{M^{\mathfrak p_{I_i}}(\lambda_{I_i, \mathbf c})}(\eta_{\ob m})=\text{Id}_{M^{\mathfrak p_{I_i}}(\lambda_{I_i, \mathbf c})}\otimes \Phi(\eta_{\ob m}),\ \ \Psi_{M^{\mathfrak p_{I_i}}(\lambda_{I_i, \mathbf c})}(1_{\ob s}\otimes \varepsilon _{\ob m})=\text{Id}_{M_{I_i,s}}\otimes  \Phi(\varepsilon _{\ob m}).
\end{equation}
For any $b\in \Hom_{\mathcal O}( M_{I_i, 0},M_{I_i, 2r})$,
$$\begin{aligned}
\tilde\eta_{\ob m}\circ \tilde\varepsilon _{\ob m}(b)&=(\text{Id}_{M_{I_i, s} }\otimes \Phi(\varepsilon _{\ob m})\otimes \text{Id}_{V^{\otimes  m}}) \circ(b\otimes \text{Id}_{V^{\otimes  2m}})\circ (\text{Id}_{M^{\mathfrak p_{I_i}}(\lambda_{I_i, \mathbf c})}\otimes \Phi(\eta_{\ob m})),~ \text{by \eqref{egama}}\\
&=(\text{Id}_{M_{I_i, s} }\otimes \Phi(\varepsilon _{\ob m})\otimes \text{Id}_{V^{\otimes  m}})\circ (\text{Id}_{M_{I_i,2r}}\otimes \Phi(\eta_{\ob m}))\circ b\\
&=\left(\text{Id}_{M_{I_i, s} }\otimes\left(( \Phi(\varepsilon _{\ob m})\otimes \text{Id}_{V^{\otimes  m}})\circ (\text{Id}_{V^{\otimes  m}}\otimes \Phi(\eta_{\ob m}))\right)\right)\circ b\\
&= (\text{Id}_{M_{I_i, s}}\otimes \Phi(1_{\ob m}))\circ b, ~ \text{ by Lemma~\ref{eviso}}\\
&=\text{Id}_{M_{I_i, 2r}}\circ b=b.\\
\end{aligned}
$$
Similarly, we have $\tilde\varepsilon _{\ob m}\circ\tilde\eta_{\ob m}(b) =b$ for any $b\in \Hom_{\mathcal O}( M_{I_i, m},M_{I_i, s})$. So,
 $\tilde\eta_{\ob m}$ and $\tilde\varepsilon _{\ob m}$ are mutually inverse to each other.
 For any $d\in \Hom_{\CB^{f_i}(\omega)}(\ob 0, \ob {2r})$,
 we have
 $$\begin{aligned}\Psi_{M^{\mathfrak p_{I_i}}(\lambda_{I_i, \mathbf c})}(\bar \varepsilon _{\ob m}(d))=&\Psi_{M^{\mathfrak p_{I_i}}(\lambda_{I_i, \mathbf c})}(1_{\ob s}\otimes \varepsilon _{\ob m})\circ \Psi_{M^{\mathfrak p_{I_i}}(\lambda_{I_i, \mathbf c})}(d\otimes 1_{\ob m}), \quad \text{ by Definition~\ref{etae},}\\
  =&\Psi_{M^{\mathfrak p_{I_i}}(\lambda_{I_i, \mathbf c})}(1_{\ob s}\otimes \varepsilon _{\ob m})\circ(\Psi_{M^{\mathfrak p_{I_i}}(\lambda_{I_i, \mathbf c})} (d)\otimes \text{Id}_{V^{\otimes m}}), \quad \text{ by  \eqref{ddnamphis},}\\
 =&( \text{Id}_{M_{I_i, s}}\otimes \Phi( \varepsilon _{\ob m}))\circ (\Psi_{M^{\mathfrak p_{I_i}}(\lambda_{I_i, \mathbf c})}(d)\otimes \text{Id}_{V^{\otimes m}}),\quad \text{by the second equation in \eqref{egama}},\\
 =&\tilde\varepsilon _{\ob m}( \Psi_{M^{\mathfrak p_{I_i}}(\lambda_{I_i, \mathbf c})}(d)).
 \end{aligned}$$
 Hence, the diagram in  \eqref{conmuf1} is commutative.  One can check that the diagram in \eqref{conmuf2} is commutative  in a similar way. In this case, we need the first equation in \eqref{egama}.
 \end{proof}

\begin{Cor}\label{injective} Suppose $2r\le \min\{q_1, q_2, \ldots, q_k\}$ and $\mathfrak g\neq \mathfrak {so}_{2n+1}$ if $i=1$.
\begin{enumerate}
\item The homomorphism $\Psi_{M^{\mathfrak p_{I_i}}(\lambda_{I_i, \mathbf c})}: \Hom_{\CB^{f_i}(\omega)}(\ob 0,\ob {2r}) \rightarrow \Hom_{\mathcal O}(M_{I_i, 0},M_{I_i, 2r} ) $ is injective.
    \item The homomorphism $\Psi_{M^{\mathfrak p_{I_i}}(\lambda_{I_i, \mathbf c})}: \End_{\CB^{f_i}(\omega)}(\ob r) \rightarrow \End_{\mathcal O}(M_{I_i, r}) $ is injective.
\end{enumerate}
\end{Cor}
\begin{proof}In Proposition~\ref{keytheorem}, we have proven that $\Hom_{\CB^{f_i}(\omega)}(\ob 0,\ob {2r})$ has a basis given by $\bar{\mathbb {ND}}^a_{0,2r}/\sim$.
Suppose that
$\Psi_{M^{\mathfrak p_{I_i}}(\lambda_{I_i, \mathbf c})}(\sum_{d\in \bar{\mathbb {ND}}^a_{0,2r}/\sim}g_{d} d)=0$,
  where  $g_d$'s $\in \mathbb C$. In particular,
   \begin{equation}\label{miphiw}
   \Psi_{M^{\mathfrak p_{I_i}}(\lambda_{I_i, \mathbf c})}(\sum_{d\in \bar{\mathbb {ND}}^a_{0,2r}/\sim}g_{d} d)(m_i)=0.
   \end{equation}
   By  the proof of Proposition~\ref{keytheorem}, $g_d=0$ for all $d\in \bar{\mathbb {ND}}^a_{0,2r}/\sim$. So,  (a) is proven.
(b) follows from (a) and Lemma~\ref{conmuf}.
\end{proof}



A partition $\lambda=(\lambda_1,\lambda_2,\ldots)$    is a sequence  of non-negative integers such that $\lambda_1\geq \lambda_2\geq \ldots$.
Write $|\lambda|=\sum_{j}\lambda_j$.
 An $a$-multipartition of $m\in\mathbb N$ is a sequence of partitions $(\mu^{(1)},\mu^{(2)},\ldots,\mu^{(a)})$
such that  $\sum_{j=1}^a|\mu^{(j)}|=m$.
Define  $$\Lambda_{a,r}:=\{\mu \mid \mu \text{ is an $a$-multipartition of $r-2l$} , 0\leq l\leq \lfloor r/2 \rfloor\}.$$
Let   $\Lambda_{I_i,r}$ be the set of all
 $\mathfrak p_{I_i}$-dominant integral weights $\lambda$ such that $ M^{\mathfrak p_{I_i}}(\lambda)$   appears in some parabolic Verma flag  of $M_{I_i, r}$.

\begin{Theorem}\label{main123}
Suppose  $M^{\mathfrak p_{I_i}}(\lambda_{I_i, \mathbf c})$ is tilting and $\mathfrak g\neq \mathfrak {so}_{2n+1}$ if $i=1$. If $2r\leq \min\{q_1, q_2, \ldots, q_k\}$, then
$ \mathcal  W_{a, r}(\omega) \cong \text{End}_{\mathcal O}(M_{I_i, r})$.
\end{Theorem}

\begin{proof}By  Lemma~\ref{ghom123},
 $\omega$ is $\mathbf u$-admissible. So, we have the isomorphism  $\bar\gamma$ in \eqref{catec}  and     $ \dim \mathcal W_{a, r}(\omega)=a^r(2r-1)!!$.
   By Corollary~\ref{injective}(b), there is a  monomorphism
$$ \Psi_{M^{\mathfrak p_{I_i}}(\lambda_{I_i, \mathbf c})}\circ \bar\gamma:  \mathcal  W_{a, r}(\omega) \hookrightarrow \text{End}_{\mathcal O}(M_{I_i, r}).$$
 It is enough to show that  the dimension of    $\text{End}_{\mathcal O}(M_{I_i, r})$  is $a^r(2r-1)!!$.

By the tensor identity (c.f. \eqref{tensoridentity}),
\begin{equation}\label{teoncushd}
M_{I_i,j}=M^{\mathfrak p_{I_i}}( \lambda_{I_i,\bf c})\otimes V^{\otimes j}\cong \U(\mathfrak g)\otimes _{\U(\mathfrak p_{I_i})}\otimes (L_{I_i}( \lambda_{I_i,\bf c})\otimes V^{\otimes j}).
\end{equation}
Recall that  $\lambda_{I_i,\bf c}$ is given  in Definition ~\ref{deltac} and $c_k=0$ if $i=2$.
For $0\leq j\leq r$ and  any  $ \lambda\in\Lambda_{I_i,j}$, we write $$\hat\lambda:=\lambda-\lambda_{I_i,\bf c}.$$ In particular, $\hat\lambda_{I_i,\bf c}=0$.
Then $\hat\lambda$ is $\mathfrak p_{I_i}$-dominant integral. By \eqref{teoncushd}, Lemma~\ref{tensorwithv} and \eqref{pdwt}, we have
\begin{equation}\label{pdomedh}
\begin{aligned}&
\text{$\hat \lambda_{l}\in \mathbb Z$, $1\leq l\leq n$, $\sum_{l=1}^{n} |\hat \lambda_{l}|\leq j$,}\\
 &\text{$\hat \lambda_{p_{l-1}+1}\geq\hat \lambda_{p_{l-1}+2}\geq \ldots\geq\hat \lambda_{p_{l}},$  if  either $i=1$ and $1\leq l\leq k$ or $i=2$ and $1\leq l\leq k-1$,}\\
& \text{$ \hat\lambda_{p_{k-1}+1}\geq \hat \lambda_{p_{k-1}+2} \geq \ldots\geq \hat\lambda_n\geq 0$, if $i=2$ and $\mathfrak g\in\{\mathfrak {so}_{2n+1}, \mathfrak {sp}_{2n}\}$,}\\
& \text{$ \hat\lambda_{p_{k-1}+1}\geq \hat \lambda_{p_{k-1}+2} \geq \ldots \geq|\hat\lambda_n|\geq 0$, if $i=2$ and $\mathfrak g=\mathfrak {so}_{2n}$,}
\end{aligned}\end{equation}
 where $p_0=0$, $p_j$'s are given in Definition~\ref{defofpi}. Since we are assuming that  $2r\leq \min\{q_1, q_2, \ldots, q_k\}$,
 \begin{equation}\label{ispart}
\begin{aligned}
& \hat \lambda_{p_{l-1}+r+1}=\ldots=\hat\lambda_{p_{l}-r}=0, \text{ if either } i=1 \text{ and } 1\leq l\leq k\text{ or }i=2 \text{ and } 1\leq l\leq k-1, \\
& \hat \lambda_{p_{k-1}+r+1}=\ldots=\hat\lambda_{n}=0, \text{ if } i=2.
\end{aligned}
\end{equation}
So,  $\lambda_n=0$ if $i=2$ and  $\lambda\in \Lambda_{I_i,j}$ for $ 0\leq j\leq r$.
 Since  $\mathfrak g\neq \mathfrak {so}_{2n+1}$ if $i=1$, by  Lemma~\ref{tensorwithv} and \eqref{teoncushd},
 \begin{equation}\label{inducpartion}
 \lambda \in \Lambda_{I_i,j} \text{ if and only if } \lambda=\mu\pm \epsilon_l \text{  for some } \mu\in \Lambda_{I_i,j-1}, 1\leq j\leq r,
 \end{equation}
 where $l\in \{1, 2, \ldots, n\}$ such that  $\mu\pm \epsilon_l$ is    $\mathfrak p_{I_i}$-dominant integral.
  Moreover,
 the multiplicity of $M^{\mathfrak p_{I_i}}(\lambda)$ in $M_{I_i,r}$ is equal to the number of
sequence of $\mathfrak p_{I_i}$-dominant integral  weights
\begin{equation}\label{sequnmuiltp}
(\gamma_0, \gamma_1, \ldots, \gamma_r)
\end{equation}
 such that $\gamma_0=\lambda_{I_i, \mathbf c}$, and $\gamma_r=\lambda$ and  $\gamma_j=\gamma_{j-1}\pm \epsilon_l\in\Lambda_{I_i,j} $ for some $l\in \{1, 2, \ldots, n\}$, $1\leq j\leq r$.

For  $0\leq j\leq r$ and  any  $ \lambda\in\Lambda_{I_i,j}$, we
 define
    $\tilde{\lambda}=( \nu^{(1)},\nu^{(2)}, \ldots, \nu^{(a)} )$   such that
  \begin{enumerate}
\item $\nu^{(l)}=(\hat \lambda_{p_{l-1}+1},\hat\lambda_{p_{l-1}+2}, \ldots,\hat \lambda_{p_{l-1}+r}) $ for $1\leq l\leq k$,
\item $\nu^{(l)}=(-\hat \lambda_{p_{2k-l+1}},-\hat\lambda_{p_{2k-l+1}-1}, \ldots,-\hat \lambda_{p_{2k-l+1}-r+1}) $ for $k+1\leq l\leq a$, $i=1$,
\item $\nu^{(l)}=(-\hat \lambda_{p_{2k-l}},-\hat \lambda_{p_{2k-l}-1}, \ldots,-\hat \lambda_{p_{2k-l}-r+1}) $ for $k+1\leq l\leq a$, $i=2$.
\end{enumerate}
\noindent Thanks to \eqref{ispart}, $\tilde{\lambda}$ is an $a$-multipartition.

It's easy to see that $\tilde \lambda \in\Lambda_{a,j}$ if  $\tilde \mu\in\Lambda_{a,j-1}$, where  $\lambda=\mu\pm \epsilon_l $  for some $ \mu\in\Lambda_{I_i,j-1}$, $1\leq j\leq r$.
By \eqref{inducpartion} and induction on $j$,  $\tilde \lambda \in\Lambda_{a,j}$  if $\lambda\in \Lambda_{I_i,j}$, $0\leq j\leq r$.
So,  we have an injective map
$$\iota_j: \Lambda_{I_i,j}\rightarrow  \Lambda_{a,j},~~ \lambda\mapsto \tilde \lambda,\ \   \text{for any $j, 0\leq j\leq r$.}$$
If $1\leq j\leq r$ and   $\nu\in \Lambda_{a,j}$, then   there is  an $a$-multipartition  $\alpha\in \Lambda_{a,j-1}$ such that
 \begin{equation}\label{ejdhhdf}
    \nu ^{(l)}=\alpha^{(l)}\pm (0^{t-1},1,0^{r-t}), \text{ for some $t$ and $l$}, \text{ and }  \nu ^{(h)}=\alpha^{(h)} \text{ for $h\neq l$.}
   \end{equation}
Suppose that $\tilde \mu=\alpha$ for some $\mu\in\Lambda_{I_i,j-1}$. Define
  $$  \begin{aligned}
      \lambda:=\left\{
                \begin{array}{ll}
                  \mu \pm\epsilon_{p_{l-1}+t}, & \hbox{ if $1\leq l\leq k$;} \\
                  \mu \mp\epsilon_{p_{2k-l+1}-t+1}, & \hbox{ if $k+1\leq l\leq a$ and $i=1$;} \\
 \mu \mp\epsilon_{p_{2k-l}-t+1}, & \hbox{ if $k+1\leq l\leq a$ and $i=2$.}
                \end{array}
              \right.
  \end{aligned}
   $$
  By \eqref{ejdhhdf},   $ \lambda-\lambda_{I_i,\mathbf c}$ satisfies the conditions in \eqref{pdomedh}. So, both  $ \lambda-\lambda_{I_i,\mathbf c}$ and  $ \lambda$ are  $\mathfrak p_{I_i}$-dominant integral.
   By \eqref{inducpartion}, $\lambda \in \Lambda_{I_i,j}$ and  $ \tilde \lambda=\nu$.
   By induction on $j$, we can verify that for any $\nu\in \Lambda_{a,j}$, there is a   $\lambda\in \Lambda_{I_i,j}$ such that $\tilde \lambda=\nu$. Thus, the map $\iota_j$ is surjective. Since   $\iota_j$ is  bijective,
  each $\gamma_j$ in \eqref{sequnmuiltp} corresponds to a unique multipartition $\tilde \gamma_j\in\Lambda_{a,j} $, $0\leq j\leq r$. So,   each  $(\gamma_0, \gamma_1, \ldots, \gamma_r)$ in \eqref{sequnmuiltp} corresponds to a unique up-down $\tilde \lambda$-tableau in the sense of \cite[\S4]{AMR} (e.g., Example~\ref{tableux}). This shows that  the
multiplicity of $M^{\mathfrak p_{I_i}}(\lambda)$ in a   parabolic Verma flag of $M_{I_i, r}$  is equal to the number of all up-down  $\tilde\lambda$-tableaux.
 Thanks to \cite[Proposition~11.1]{Hum}, the tensor product of any tilting module  and a finite dimensional module is also tilting. This implies that  $M_{I_i, r}$ is a tilting module in
$\mathcal O^{\mathfrak p_{I_i}}$.
 A standard arguments on endomorphism algebra of a tilting module shows that
the dimension of $\text{End}_{\mathcal O}(M_{I_i, r})$ is equal to the summation of squares of numbers of all up-down $\mu$-tableaux, where $\mu$ ranges over all elements in  $\Lambda_{a,r}$. Thanks to \cite[Lemma~5.1]{AMR}, this  summation is equal to $a^r (2r-1)!!$. So,  the dimension of    $\text{End}_{\mathcal O}(M_{I_i, r})$  is $a^r(2r-1)!!$.
\end{proof}
\begin{rem}\label{remasimple}
It follows from \cite[Theorem~9.12]{Hum} that $M^{\mathfrak p_{I_i}}(\lambda)$ is simple (hence tilting) if
$\langle \lambda+\rho, \beta^\vee \rangle \not\in \mathbb Z^{>0}$  for all $\beta\in R^+\setminus R_{I_i}^+$.
 So,  $M^{\mathfrak p_{I_i}}(\lambda_{I_i, \mathbf c})$   is  tilting in $\mathcal{O}^{\mathfrak p_{I_i}}$
if  the following conditions hold:
 \begin{enumerate} \item  $c_j-c_{l}+p_l-p_{j-1}-1\not\in \mathbb Z^{>0}$,  $\forall\  1\leq j<l\leq k$,
 \item $c_j-p_{j-1}+n\not\in \mathbb Z^{>0}$, and $2(c_j-p_{j-1}+n)-1\not\in \mathbb Z^{>0}$,  $\forall \ 1\leq j\leq k_i$ if $\mathfrak g= \mathfrak {sp}_{2n}$,
\item $2(c_j-p_{j-1}+n)-1\not\in \mathbb Z^{>0}$,  $\forall \ 1\leq j\leq k_i$ if $\mathfrak g=\mathfrak{so}_{2n+1}$,
\item $2(c_j-p_{j-1}+n)-3\not\in \mathbb Z^{>0}$,  $\forall \ 1\leq j\leq k_i$ if $\mathfrak g=\mathfrak{so}_{2n}$,

 \item  $c_j+c_l-p_{j-1}-p_{l-1}+N-2\not\in \mathbb Z^{>0}$,    $\forall 1\leq j< l\leq k$ if  $\mathfrak g=\mathfrak {so}_{N}$,
\item $c_j+c_l-p_{j-1}-p_{l-1}+N\not\in \mathbb Z^{>0}$,   $\forall \ 1\leq j< l\leq k$ if   $\mathfrak g=\mathfrak {sp}_{2n}$,
\end{enumerate}
where  $k_1=k$  and  $k_2= k-1$.
The Jantzen's criterion for simplicity of a parabolic Verma module can be found in \cite[Theorem~9.13]{Hum}.
\end{rem}

 Suppose that $2r\leq \min\{q_1,q_2,\cdots,q_k\}$ and $\mathfrak g\neq \mathfrak {so}_{2n+1}$ if $i=1$. Given any $\lambda_{I_i,\mathbf c}$ in Definition~\ref{deltac} such that $M^{\mathfrak p_{I_i}}(\lambda)$ is simple (see Remark~\ref{remasimple}), we have polynomials $f_i(t)$ given in Definition~\ref{polf12} and $\mathbf u$-admissible parameter $\omega$ given in Lemma~\ref{ghom123}. Hence we have the cyclotomic Nazarov-Wenzl algebra $ \mathcal W_{a, r}(\omega)$ associated to $f_i(t)$ and $\omega$, where $a=\text{deg}f_i(t)$ (see \eqref{defofd}). Thanks to Theorem~\ref{main123}, such  $ \mathcal W_{a, r}(\omega)$ can be realized as the endomorphism algebra $\text{End}_{\mathcal O}(M_{I_i, r})$.

We keep the assumption in Theorem~\ref{main123} in the remaining part of the paper.
There is a well-known contravariant character preserving duality functor $\vee: \mathcal O^{\mathfrak p_{I_i}}\rightarrow\mathcal O^{\mathfrak p_{I_i}}$ \cite[\S3.2]{Hum}. Since $M_{I_i, r}$ is a tilting module in   $\mathcal O^{\mathfrak p_{I_i}}$, it is   self-dual under the functor $\vee$.  There is a partial order $\leq$ on $\Lambda_{I_i,r}$
 such that $\lambda\leq\mu$ if and only if $\mu-\lambda\in \mathbb N\Pi$.
 By standard arguments in \cite[\S4]{AST}, $\text{End}_{\mathcal O}(M_{I_i, r})$ is a cellular algebra with respect to the poset $(\Lambda_{I_i,r},\leq)$ in the sense of \cite{GL}. The corresponding anti-involution is given by
 \begin{equation}\label{antiinvolution}
 \sigma:\text{End}_{\mathcal O}(M_{I_i, r})\rightarrow\text{End}_{\mathcal O}(M_{I_i, r}), \phi\mapsto \phi^\vee.
 \end{equation}
 By  \cite[Definition~4.1]{AST}, define  $$S(\lambda):=\Hom_{\mathcal O}(M^{\mathfrak p_{I_i}}(\lambda),M_{I_i, r})  \text{ and }  S(\lambda)^*:=\Hom_{\mathcal O}(M_{I_i, r},M^{\mathfrak p_{I_i}}(\lambda)^\vee), \forall \lambda\in\Lambda_{I_i,r}.$$
 Then $S(\lambda)$ is the left cell module of $\text{End}_{\mathcal O}(M_{I_i, r})$ and  $S(\lambda)^*$ is the right dual cell  module  of $\text{End}_{\mathcal O}(M_{I_i, r})$.  Let $T^{\mathfrak p_{I_i}}(\lambda)$ be the indecomposable tilting module in $\mathcal O^{\mathfrak p_{I_i}}$ with the highest weight $\lambda\in \Lambda_{I_i,r}$.
 Define $$\bar\Lambda_{I_i,r}:=\{\lambda\in\Lambda_{I_i,r}\mid  T^{\mathfrak p_{I_i}}(\lambda) \text{ is a direct summand of }M_{I_i, r}\}.$$
 By \cite[Theorem~4.11]{AST}, $S(\lambda)$ has the simple head, say  $ D(\lambda)$ for any $\lambda\in \bar\Lambda_{I_i,r}$. Moreover, $ \{D(\lambda)\mid\lambda  \in \bar\Lambda_{I_i,r}\}$ is the complete set of pairwise non-isomorphic simple $\text{End}_{\mathcal O}(M_{I_i, r})$-modules.

Since  $M_{I_i,r}$ is  a left  $\text{End}_{\mathcal O}(M_{I_i, r})$-module, $ \Hom_{\mathcal O}(M_{I_i, r},M)$ is  a right $\text{End}_{\mathcal O}(M_{I_i, r})$-module for any  $M\in \mathcal O^{\mathfrak p_{I_i}}$.
Via the anti-involution  $\sigma$ in  \eqref{antiinvolution}, any left $\text{End}_{\mathcal O}(M_{I_i, r})$-module can be viewed as a right  $\text{End}_{\mathcal O}(M_{I_i, r})$-module and vice versa.
\begin{Defn}Let $\text{End}_{\mathcal O}(M_{I_i, r})\text{-mod}$ be the category of left $\text{End}_{\mathcal O}(M_{I_i, r})$-modules.
Define two functors
$$  \begin{aligned}
\mathbf f:=&\Hom_{\mathcal O}(M_{I_i, r},-): \mathcal O^{\mathfrak p_{I_i}}\rightarrow \text{End}_{\mathcal O}(M_{I_i, r})\text{-mod},\\
 \mathbf g:=&M_{I_i, r}\otimes_{ \text{End}_{\mathcal O}(M_{I_i, r})}-: \text{End}_{\mathcal O}(M_{I_i, r})\text{-mod}\rightarrow\mathcal O^{\mathfrak p_{I_i}}.
\end{aligned}$$
\end{Defn}
\begin{Lemma}\label{isomorcell}Keep the assumption in Theorem~\ref{main123}. For any $\lambda\in\Lambda_{I_i,r}$,
  $\mathbf f( M^{\mathfrak p_{I_i}}(\lambda)^\vee)\cong S(\lambda)$ as left $\text{End}_{\mathcal O}(M_{I_i, r})$-modules.
\end{Lemma}
\begin{proof}
The duality functor $\vee$ induces two linear maps
  \begin{equation} \label{vvv} \begin{aligned} & \vee: \Hom_{\mathcal O}(M_{I_i, r},M^{\mathfrak p_{I_i}}(\lambda)^\vee) \rightarrow \Hom_{\mathcal O}(M^{\mathfrak p_{I_i}}(\lambda),M_{I_i, r}), \\
& \vee: \Hom_{\mathcal O}(M^{\mathfrak p_{I_i}}(\lambda),M_{I_i, r}) \rightarrow \Hom_{\mathcal O}(M_{I_i, r},M^{\mathfrak p_{I_i}}(\lambda)^\vee).\\
\end{aligned}\end{equation}
Since the square of $\vee$ is isomorphic  to the identity functor (c.f. \cite[Theorem~3.2(a)]{Hum}),  two linear maps in \eqref{vvv} are bijective.
  Note that $\text{End}_{\mathcal O}(M_{I_i, r})$ acts on $\Hom_{\mathcal O}(M_{I_i, r},M^{\mathfrak p_{I_i}}(\lambda)^\vee)$ via
\begin{equation}\label{actionofvee}
gh:= h\circ g^\vee, \text{ for } h\in \Hom_{\mathcal O}(M_{I_i, r},M^{\mathfrak p_{I_i}}(\lambda)^\vee), g\in \text{End}_{\mathcal O}(M_{I_i, r}).
\end{equation}
For any $h\in \Hom_{\mathcal O}(M_{I_i, r},M^{\mathfrak p_{I_i}}(\lambda)^\vee), h_1\in\Hom_{\mathcal O}(M^{\mathfrak p_{I_i}}(\lambda),M_{I_i, r}),  g\in \text{End}_{\mathcal O}(M_{I_i, r})$,
\begin{equation}\label{mouldeact}
(gh)^\vee\overset{\eqref{actionofvee}}=(h\circ g^\vee)^\vee=g\circ h^\vee=g(h^\vee),\quad (gh_1)^\vee=(g\circ h_1)^\vee=h_1^\vee\circ g^\vee\overset{\eqref{actionofvee}}=g(h_1^\vee).
\end{equation}
This shows that  the bijective maps $\vee$  in \eqref{vvv}  are  $\text{End}_{\mathcal O}(M_{I_i, r})$-homomorphisms.
\end{proof}
The arguments  in the proof of Lemmas~\ref{isotiltin}--\ref{porfcover} are essentially the formal arguments  in \cite{RS1,RS2}.
\begin{Lemma}\label{isotiltin}
Keep the assumption in Theorem~\ref{main123}.
Let $T$ be any indecomposable direct summand of $M_{I_i, r}$. Then $\mathbf g\mathbf f(T)\cong T$.
\end{Lemma}
\begin{proof}We have the isomorphism $\phi: \mathbf g\mathbf f(M_{I_i, r})\rightarrow  M_{I_i, r} $ such that $\phi(v\otimes h)=h(v)$, $\forall v\in M_{I_i, r} $ and $h\in\text{End}_{\mathcal O}(M_{I_i, r})$. Since $T$ is a direct summand of $M_{I_i, r}$, the projection $\pi:M_{I_i, r}\rightarrow T$ induces the homomorphism $1\otimes \mathbf f(\pi): \mathbf g\mathbf f(M_{I_i, r})\rightarrow\mathbf g\mathbf f(T)$ such that
\begin{equation}\label{piphiemi}
\pi\circ\phi=\tilde \phi\circ(1\otimes\mathbf f(\pi) ),
\end{equation}
where $\tilde \phi:\mathbf g\mathbf f(T)\rightarrow T$, $v\otimes h\mapsto h(v)$, for $v\in M_{I_i, r} $ and $h\in \mathbf f(T)$. By \eqref{piphiemi}, $\tilde \phi $ is surjective. Comparing characters  yields that $\tilde \phi $ is an isomorphism.
\end{proof}

  \begin{Lemma}\label{porfcover}Keep the assumption in Theorem~\ref{main123}. Define  $P(\lambda):= \Hom_{\mathcal O}(M_{I_i, r},T^{\mathfrak p_{I_i}}(\lambda))$, $\lambda\in\bar\Lambda_{I_i,r}$. Then
  $P(\lambda)$ is the projective cover of $D(\lambda)$, $\lambda\in\bar\Lambda_{I_i,r}$.
  \end{Lemma}
  \begin{proof}
  It follows from standard arguments (c.f. \cite[II,Proposition~2.1(c)]{ARS}) that $P(\lambda)$ is a principal  indecomposable left $\text{End}_{\mathcal O}(M_{I_i, r})$-module.
  For any $\lambda\in \bar \Lambda_{I_i,r}$ and $\mu\in\Lambda_{I_i,r}$,
  \begin{equation}\label{dec123}
  \begin{aligned}
  \Hom_{\mathcal O}(T^{\mathfrak p_{I_i}}(\lambda), M^{\mathfrak p_{I_i}}(\mu)^\vee )&\cong\Hom_{\mathcal O}(\mathbf g\mathbf f(T^{\mathfrak p_{I_i}}(\lambda)), M^{\mathfrak p_{I_i}}(\mu)^\vee ), \text{ by Lemma~\ref{isotiltin}},\\
  &\cong \Hom_{\text{End}_{\mathcal O}(M_{I_i, r})}(\mathbf f( T^{\mathfrak p_{I_i}}(\lambda)), \mathbf f(M^{\mathfrak p_{I_i}}(\mu)^\vee))\\
  &\cong \Hom_{\text{End}_{\mathcal O}(M_{I_i, r})}(P(\lambda), S(\mu)), \text{ by Lemma~\ref{isomorcell}}, \\
  \end{aligned}
 \end{equation}
  where the isomorphism in the second row follows from the fact that $\mathbf f$ and $\mathbf g$ are adjoint pairs. Thus
  \begin{equation}\label{dimcomp}
  \text{dim}\Hom_{\mathcal O}(T^{\mathfrak p_{I_i}}(\lambda), M^{\mathfrak p_{I_i}}(\mu)^\vee )=\text{dim}\Hom_{\text{End}_{\mathcal O}(M_{I_i, r})}(P(\lambda), S(\mu)).
  \end{equation}
  If $P(\lambda)$ is the projective cover of $D(\nu)$, then \begin{equation}\label{muttil} \text{dim}\Hom_{\text{End}_{\mathcal O}(M_{I_i, r})}(P(\lambda), S(\mu))=[S(\mu):D(\nu)],\end{equation}  where $[S(\mu):D(\nu)]$, called  a decomposition number,  is  the multiplicity of  the simple module $D(\nu)$ in the cell module $S(\mu)$. If we assume $\mu=\lambda$, by  \eqref{dec123}--\eqref{muttil},  $[S(\lambda):D(\nu)]\neq 0$.
Thanks to    \cite[Corollary~4.10]{AST}, $\lambda\le \nu$.  If we assume $\mu=\nu$, we have   $\text{dim}\Hom_{\mathcal O}(T^{\mathfrak p_{I_i}}(\lambda), M^{\mathfrak p_{I_i}}(\nu)^\vee )\neq0$. Since $\lambda$ is the highest weight of  $T^{\mathfrak p_{I_i}}(\lambda)$ (c.f. \cite[\S11.8]{Hum}), $\lambda\ge \nu$, forcing $\lambda=\nu$.
  \end{proof}

\begin{Cor}Keep the assumption in Theorem~\ref{main123}. Then
 \begin{equation} \label{decp} [S(\lambda):D(\mu)]=(T^{\mathfrak p_{I_i}}(\mu):M^{\mathfrak p_{I_i}}(\lambda))\ \  \text{ for all  $\lambda  \in \Lambda_{I_i,r}$ and $\mu  \in \bar\Lambda_{I_i,r}$,} \end{equation}
 where   $(T^{\mathfrak p_{I_i}}(\mu):M^{\mathfrak p_{I_i}}(\lambda)) $ is the multiplicity of $ M^{\mathfrak p_{I_i}}(\lambda)$ in any parabolic Verma   filtration of the tilting module $ T^{\mathfrak p_{I_i}}(\mu)$.
 \end{Cor}
 \begin{proof}
 The result follows from \eqref{dimcomp}--\eqref{muttil} and the well-known fact (c.f.\cite[\S9.8]{Hum}) that
 $$ (T^{\mathfrak p_{I_i}}(\mu):M^{\mathfrak p_{I_i}}(\lambda))=\text{dim}\Hom_{\mathcal O}(T^{\mathfrak p_{I_i}}(\mu), M^{\mathfrak p_{I_i}}(\lambda)^\vee ).$$
 \end{proof}
Thanks to Theorem~\ref{main123}, $\mathcal  W_{a, r}(\omega)\cong \text{End}_{\mathcal O}(M_{I_i, r})$. Therefore,   one can use \eqref{decp}  to compute  the decomposition numbers of $\mathcal  W_{a, r}(\omega) $. More explicitly, these numbers can be computed via the Kazhdan-Lusztig  polynomials associated with Weyl groups of types $B, C, D$. For details, see, e.g, \cite[Chapter~11]{Hum} and references therein.
   If  $\mathfrak g=\mathfrak{so}_{2n}$, $i=1$ and $k=1$, the isomorphism    $\mathcal  W_{2, r}(\omega)\cong \text{End}_{\mathcal O}(M_{I_2, r})$ has been given in \cite{ES}.

  We close the paper by giving  some  examples. One can  see that
    the condition `` $ \mfg\neq \mathfrak{so}_{2n+1}$ if $i=1$ '' in Theorem~\ref{main123}  could not  be removed.  We  remark that we have defined the $a$-mutilpartition $\tilde \lambda$  for a  $\mathfrak p_{I_i}$-dominant integral weight $\lambda$ in the proof of  Theorem~\ref{main123}.

 \begin{example}\label{tableux}
Suppose that $2r\leq\min\{q_1,q_2,\ldots,q_k\}$,
 $r=2$ and  $\lambda=\lambda_{I_i,\bf c}$ such that $M^{\mathfrak p_{I_i}}(\lambda_{I_i, \mathbf c})$ is tilting.

 If $i=1$, $k=2$ and $\mathfrak g\in \{\mathfrak {so}_{2n}, \mathfrak {sp}_{2n}\}$, then $a=4$ and  $\tilde{\lambda}=(\emptyset,\emptyset,\emptyset,\emptyset)$. There are four parabolic Verma paths for $M^{\mathfrak p_1}(\lambda_{I_1,\bf c} )$ in $M_{I_1, 2}$ as follows:
\begin{enumerate}
\item $M^{\mathfrak p_1}(\lambda_{I_1,\bf c} ) \rightarrow M^{\mathfrak p_1}(\lambda_{I_1,\bf c}+\varepsilon_1 ) \rightarrow M^{\mathfrak p_1}(\lambda_{I_1,\bf c} )$,
\item $M^{\mathfrak p_1}(\lambda_{I_1,\bf c} ) \rightarrow M^{\mathfrak p_1}(\lambda_{I_1,\bf c}+\varepsilon_{p_1+1} ) \rightarrow M^{\mathfrak p_1}(\lambda_{I_1,\bf c} )$,
\item $M^{\mathfrak p_1}(\lambda_{I_1,\bf c} ) \rightarrow M^{\mathfrak p_1}(\lambda_{I_1,\bf c}-\varepsilon_n ) \rightarrow M^{\mathfrak p_1}(\lambda_{I_1,\bf c} )$,
\item $M^{\mathfrak p_1}(\lambda_{I_1,\bf c} ) \rightarrow M^{\mathfrak p_1}(\lambda_{I_1,\bf c}-\varepsilon_{p_1} ) \rightarrow M^{\mathfrak p_1}(\lambda_{I_1,\bf c} )$.
\end{enumerate}
\noindent The corresponding  up-down $\tilde{\lambda}$-tableaux  are given as follows:
\begin{enumerate}
\item $(\emptyset,\emptyset,\emptyset,\emptyset)\rightarrow((1),\emptyset,\emptyset,\emptyset)\rightarrow (\emptyset,\emptyset,\emptyset,\emptyset)$,
\item $(\emptyset,\emptyset,\emptyset,\emptyset)\rightarrow(\emptyset,(1),\emptyset,\emptyset)\rightarrow (\emptyset,\emptyset,\emptyset,\emptyset)$,
\item $(\emptyset,\emptyset,\emptyset,\emptyset)\rightarrow(\emptyset,\emptyset,(1),\emptyset)\rightarrow (\emptyset,\emptyset,\emptyset,\emptyset)$,
\item $(\emptyset,\emptyset,\emptyset,\emptyset)\rightarrow(\emptyset,\emptyset,\emptyset,(1))\rightarrow (\emptyset,\emptyset,\emptyset,\emptyset)$.
\end{enumerate}

 If  $ i=2$, $k=2$ and  $\mathfrak g\in \{\mathfrak {so}_{2n}, \mathfrak {sp}_{2n},\mathfrak {so}_{2n+1}\}$, then $a=3$, $\tilde{\lambda}=(\emptyset,\emptyset,\emptyset)$. There are three parabolic Verma paths for $M^{\mathfrak p_2}(\lambda_{2,\mathbf c} )$ in $M_{I_2, 2}$ as follows:

\begin{enumerate}
\item $M^{\mathfrak p_1}(\lambda_{I_1,\bf c} ) \rightarrow M^{\mathfrak p_1}(\lambda_{I_1,\bf c}+\varepsilon_1 ) \rightarrow M^{\mathfrak p_1}(\lambda_{I_1,\bf c} )$,
\item $M^{\mathfrak p_1}(\lambda_{I_1,\bf c} ) \rightarrow M^{\mathfrak p_1}(\lambda_{I_1,\bf c}+\varepsilon_{p_1+1} ) \rightarrow M^{\mathfrak p_1}(\lambda_{I_1,\bf c} )$,
\item $M^{\mathfrak p_1}(\lambda_{I_1,\bf c} ) \rightarrow M^{\mathfrak p_1}(\lambda_{I_1,\bf c}-\varepsilon_{p_1} ) \rightarrow M^{\mathfrak p_1}(\lambda_{I_1,\bf c} )$.
\end{enumerate}

\noindent The corresponding  up-down  $\tilde{\lambda}$-tableaux are given as follows:
\begin{enumerate}
\item $(\emptyset,\emptyset,\emptyset)\rightarrow((1),\emptyset,\emptyset)\rightarrow (\emptyset,\emptyset,\emptyset)$,
\item $(\emptyset,\emptyset,\emptyset)\rightarrow(\emptyset,(1),\emptyset)\rightarrow (\emptyset,\emptyset,\emptyset)$,
\item $(\emptyset,\emptyset,\emptyset)\rightarrow(\emptyset,\emptyset,(1))\rightarrow (\emptyset,\emptyset,\emptyset)$.
\end{enumerate}

If $ i=1$, $k=1$ and  $\mathfrak g=\mathfrak {so}_{2n+1}$, then $a=3$ and $\dim_{\mathbb C}\mathcal W_{3,2}(\omega)=27$.
In this case,  $\Lambda_{I_1,2}=\Lambda_1\cup \Lambda_2\cup \Lambda_3$, where
$\Lambda_1=\{\lambda_{I_1,\bf c}+2\varepsilon_1, \lambda_{I_1,\bf c}+\varepsilon_1+\varepsilon_2, \lambda_{I_1,\bf c}-2\varepsilon_n, \lambda_{I_1,\bf c}-\varepsilon_n-\varepsilon_{n-1}\}$, $\Lambda_2=\{\lambda_{I_1,\bf c}+ \varepsilon_1-\varepsilon_n, \lambda_{I_1,\bf c}+ \varepsilon_1, \lambda_{I_1,\bf c}- \varepsilon_n\}$, $\Lambda_3=\{ \lambda_{I_1,\bf c}\}$.
Further, $(M_{I_1,2}:M^{\mathfrak p_1}(\mu ))=j$     if  $\mu\in\Lambda_j$, for $j=1,2,3$. So, $\dim_{\mathbb C}\End_{\mathcal O}( M_{I_1,2})=25$.
Therefore, $\Psi_{M^{\mathfrak p_1}(\lambda_{I_1,\bf c} ) }: \End_{\CB^{f_1}(\omega)}(\ob 2)\rightarrow \End_{\mathcal O}( M_{I_1,2})$ can not be injective.

\end{example}

\end{document}